\documentclass[11pt,a4paper,reqno]{amsart}       

\usepackage[vcentermath]{youngtab}
\usepackage{clrscode}    
\usepackage{mathrsfs,amssymb,pgf}          
\usepackage{amsmath,latexsym,pstricks}
\usepackage{graphicx}
\usepackage{comment}
\usepackage{multicol}
\usepackage{verbatim}
\usepackage{cite}
\usepackage{tikz-cd}

\usepackage{geometry}
\numberwithin{equation}{section}

\newtheorem{theorem}[equation]{Theorem}  

\newtheorem{corollary}[equation]{Corollary}
\newtheorem{proposition}[equation]{Proposition}
\newtheorem{lemma}[equation]{Lemma}

\theoremstyle{definition}
\newtheorem{definition}[equation]{Definition}
\newtheorem{example}[equation]{Example}
\newtheorem{remark}[equation]{Remark}

\newcommand{\nc}{\newcommand}
\nc{\rnc}{\renewcommand}

\def\tab(#1){\,\mbox{\tiny$\young(#1)$}\,}
\def\etab(#1){\,\mbox{\tiny$\yng(#1)$}\,}

\newcommand{\dom}{\operatorname{dom}}
\newcommand{\im}{\operatorname{im}}

\newcommand{\mfs}{\mathfrak{s}}
\newcommand{\mft}{\mathfrak{t}}
\newcommand{\mfu}{\mathfrak{u}}
\newcommand{\mfv}{\mathfrak{v}}

\newcommand{\gr}{\mathscr{R}}
\newcommand{\gl}{\mathscr{L}}
\newcommand{\gh}{\mathscr{H}}
\newcommand{\gd}{\mathscr{D}}
\newcommand{\gj}{\!\mathscr{J}\!}

\newcommand{\B}{\mathcal{B}}
\newcommand{\Bell}{B}
\newcommand{\J}{\mathcal{J}}

\newcommand{\M}{\mathcal{M}^0}

\newcommand{\lb}{\langle}
\newcommand{\rb}{\rangle}

\newcommand{\rank}{\mathrm{rank}}
\newcommand{\idrank}{\mathrm{idrank}}

\renewcommand{\P}{\mathcal{P}}

\newcommand{\G}{\mathcal{G}}
\renewcommand{\S}{\mathcal S}

\newcommand{\X}{\mathcal{X}}

\newcommand{\sm}{\setminus}

\usepackage{tikz}
\usetikzlibrary{decorations.markings}
\usetikzlibrary{arrows,matrix}
\usepgflibrary{arrows}
  
\tikzset{->-/.style={decoration={
  markings,
  mark=at position #1 with {\arrow{>}}},postaction={decorate}}}
\tikzset{-<-/.style={decoration={
  markings,
  mark=at position #1 with {\arrow{<}}},postaction={decorate}}}

\newcommand{\dde}[2]{
\draw [->-=0.5] (#1) to [bend left=15] (#1#2);
\draw [->-=0.5] (#1#2) to [bend left=15] (#2);
\draw [-<-=0.5] (#1) to [bend right=15] (#1#2);
\draw [-<-=0.5] (#1#2) to [bend right=15] (#2);
}

\newcommand{\dden}[2]{
\draw [->-=0.5] (#1) to [bend left=10] (#1#2);
\draw [->-=0.5] (#1#2) to [bend left=10] (#2);
\draw [-<-=0.5] (#1) to [bend right=10] (#1#2);
\draw [-<-=0.5] (#1#2) to [bend right=10] (#2);
}

\newcommand{\ddeb}[2]{
\draw [->-=0.5] (#1) to [bend left=15] (#2#1);
\draw [->-=0.5] (#2#1) to [bend left=15] (#2);
\draw [-<-=0.5] (#1) to [bend right=15] (#2#1);
\draw [-<-=0.5] (#2#1) to [bend right=15] (#2);
}

\newcommand{\ddenb}[2]{
\draw [->-=0.5] (#1) to [bend left=10] (#2#1);
\draw [->-=0.5] (#2#1) to [bend left=10] (#2);
\draw [-<-=0.5] (#1) to [bend right=10] (#2#1);
\draw [-<-=0.5] (#2#1) to [bend right=10] (#2);
}

\nc{\arcup}[2]{
\draw(#1,2)arc(180:270:.4) (#1+.4,1.6)--(#2-.4,1.6) (#2-.4,1.6) arc(270:360:.4);
}
\nc{\arcdn}[2]{
\draw(#1,0)arc(180:90:.4) (#1+.4,.4)--(#2-.4,.4) (#2-.4,.4) arc(90:0:.4);
}
\nc{\cve}[2]{
\draw(#1,2) to [out=270,in=90] (#2,0);
}

\nc{\barcup}[3]{
\draw(#1,2)arc(180:270:#3) (#1+#3,2-#3)--(#2-#3,2-#3) (#2-#3,2-#3) arc(270:360:#3);
}
\nc{\barcdn}[3]{
\draw(#1,0)arc(180:90:#3) (#1+#3,#3)--(#2-#3,#3) (#2-#3,#3) arc(90:0:#3);
}
\nc{\garcup}[3]{
\draw[lightgray](#1,2)arc(180:270:#3) (#1+#3,2-#3)--(#2-#3,2-#3) (#2-#3,2-#3) arc(270:360:#3);
}
\nc{\garcdn}[3]{
\draw[lightgray](#1,0)arc(180:90:#3) (#1+#3,#3)--(#2-#3,#3) (#2-#3,#3) arc(90:0:#3);
}
\nc{\gcve}[2]{
\draw[lightgray](#1,2) to [out=270,in=90] (#2,0);
}

\nc{\catarc}[3]{
\draw(#1,2)arc(180:270:#3) (#1+#3,2-#3)--(#2-#3,2-#3) (#2-#3,2-#3) arc(270:360:#3);
}

\nc{\longarcup}[2]{
\draw(#1,2)arc(180:270:.6) (#1+.6,1.4)--(#2-.6,1.4) (#2-.6,1.4) arc(270:360:.6);
}
\nc{\longarcdn}[2]{
\draw(#1,0)arc(180:90:.6) (#1+.6,.6)--(#2-.6,.6) (#2-.6,.6) arc(90:0:.6);
}

\nc{\bluebox}[4]{
\draw[color=blue!20, fill=blue!20] (#1,#2)--(#3,#2)--(#3,#4)--(#1,#4)--(#1,#2);
}

\nc{\siijk}[3]{
	\fill (#1,0)circle(.1)
	      (#2,0)circle(.1)
	      (#3,0)circle(.1)
	      (#1,2)circle(.1)
	      (#2,2)circle(.1)
	      (#3,2)circle(.1);
  \draw(#1,2)node[above]{{\tiny $\phantom{jk}i\phantom{jk}$}};
  \draw(#2,2)node[above]{{\tiny $\phantom{jk}j\phantom{jk}$}};
  \draw(#3,2)node[above]{{\tiny $\phantom{jk}k\phantom{jk}$}};
	\cve{#3}{#1}
}

\nc{\uv}[1]{\pscircle*(#1,2){2pt}}
\nc{\lv}[1]{\pscircle*(#1,0){2pt}}
\nc{\stline}[2]{\psline(#1,2)(#2,0)}
\nc{\uline}[2]{\psline(#1,2)(#2,2)}
\nc{\lline}[2]{\psline(#1,0)(#2,0)}
\nc{\uvw}[1]{\pscircle*[linecolor=white](#1,2){2pt}\pscircle(#1,2){2pt}}
\nc{\lvw}[1]{\pscircle*[linecolor=white](#1,0){2pt}\pscircle(#1,0){2pt}}
\nc{\guv}[1]{\pscircle*[linecolor=lightgray](#1,2){2pt}}
\nc{\glv}[1]{\pscircle*[linecolor=lightgray](#1,0){2pt}}
\nc{\gstline}[2]{\psline[linecolor=lightgray](#1,2)(#2,0)}
\nc{\ul}[2]{\psline(#1,2)(#2,2)}
\rnc{\ll}[2]{\psline(#1,0)(#2,0)}
\nc{\uc}{\pscurve(0,2)(1,1.6)(2,2)}
\nc{\lc}{\pscurve(0,0)(1,0.4)(2,0)}
\nc{\gul}[2]{\psline[linecolor=lightgray](#1,2)(#2,2)}
\nc{\gll}[2]{\psline[linecolor=lightgray](#1,0)(#2,0)}
\nc{\guc}{\pscurve[linecolor=lightgray](0,2)(1,1.6)(2,2)}
\nc{\glc}{\pscurve[linecolor=lightgray](0,0)(1,0.4)(2,0)}
\nc{\uvtx}[1]{\pscircle*(#1,4){2pt}}
\nc{\uvertex}[1]{\pscircle*(#1,2){2pt}}
\nc{\lvertex}[1]{\pscircle*(#1,0){2pt}}
\nc{\stlinebd}[2]{\psline[border=0.5 mm](#1,2)(#2,0)}
\nc{\guvertex}{\guv} \nc{\glvertex}{\glv}
\nc{\suv}[1]{\pscircle*(#1,2){1.5pt}}
\nc{\slv}[1]{\pscircle*(#1,0){1.5pt}}
\nc{\wsuv}[1]{\pscircle*[linecolor=white](#1,2){1.5pt}\pscircle(#1,2){1.5pt}}
\nc{\wslv}[1]{\pscircle*[linecolor=white](#1,0){1.5pt}\pscircle(#1,0){1.5pt}}
\nc{\uds}[2]{\psline[linestyle=dotted](#1,2)(#2,2)}
\nc{\lds}[2]{\psline[linestyle=dotted](#1,0)(#2,0)}

\nc{\bn}{[n]}
\nc{\al}{\alpha}
\nc{\be}{\beta}
\nc{\ga}{\gamma}
\nc{\de}{\delta}
\nc{\ve}{\varepsilon}
\nc{\si}{\sigma}
\nc{\lam}{\lambda}
\nc{\Lam}{\Lambda}
\nc{\set}[2]{\{ {#1} : {#2} \}} 
\nc{\bigset}[2]{\big\{ {#1} : {#2} \big\}} 
\nc{\Bigset}[2]{\Big\{ \,{#1}\, \,\Big|\, \,{#2}\, \Big\}}
\nc{\codom}{\operatorname{codom}}
\nc{\coker}{\operatorname{coker}}
\nc{\sub}{\subseteq}
\rnc{\sp}{\supseteq}
\nc{\speq}{\hspace{-.25cm}&=&\hspace{-.25cm}}
\nc{\spra}{\Rightarrow}
\nc{\De}{\Delta}
\nc{\mt}{\mapsto}
\nc{\COMMA}{,\quad}
\nc{\la}{\langle}
\nc{\ra}{\rangle}
\nc{\oijn}{1\leq i<j\leq n}
\nc{\E}{\mathcal E}

\nc{\PP}{\mathcal{PP}}
\nc{\JrPn}{J_r(\P_n)}
\nc{\JrpPn}{J_{r+1}(\P_n)}
\nc{\IrPn}{I_r(\P_n)}
\nc{\JnPn}{J_n(\P_n)}
\nc{\InPn}{I_n(\P_n)}
\nc{\InmPn}{I_{n-1}(\P_n)}
\nc{\JnmPn}{J_{n-1}(\P_n)}
\nc{\JrPPn}{J_r(\PP_n)}
\nc{\JrpPPn}{J_{r+1}(\PP_n)}
\nc{\IrPPn}{I_r(\PP_n)}
\nc{\JnPPn}{J_n(\PP_n)}
\nc{\InPPn}{I_n(\PP_n)}
\nc{\InmPPn}{I_{n-1}(\PP_n)}
\nc{\JnmPPn}{J_{n-1}(\PP_n)}
\nc{\JrBn}{J_r(\B_n)}
\nc{\JrpBn}{J_{r+2}(\B_n)}
\nc{\IrBn}{I_r(\B_n)}
\nc{\JnBn}{J_n(\B_n)}
\nc{\InBn}{I_n(\B_n)}
\nc{\InmBn}{I_{n-2}(\B_n)}
\nc{\JnmBn}{J_{n-2}(\B_n)}
\nc{\JrJn}{J_r(\J_n)}
\nc{\JrpJn}{J_{r+2}(\J_n)}
\nc{\IrJn}{I_r(\J_n)}
\nc{\JnJn}{J_n(\J_n)}
\nc{\InJn}{I_n(\J_n)}
\nc{\InmJn}{I_{n-2}(\J_n)}
\nc{\JnmJn}{J_{n-2}(\J_n)}
\nc{\PnSn}{\P_n\setminus\S_n}
\nc{\BnSn}{\B_n\setminus\S_n}
\nc{\JnSn}{\J_n\setminus\{1\}}
\nc{\T}{\mathcal T}

\nc{\partn}[4]{\left( \begin{array}{c|c} %fine
#1 \ & \ #3 \ \ \\ \cline{2-2}
#2 \ & \ #4 \ \
\end{array} \!\!\! \right)}

\nc{\partnlong}[6]{\left( \begin{array}{c|c|c} %fine
#1 \ & \ #3 \ & \ #5 \ \ \\ \cline{3-3}
#2 \ & \ #4 \ & \ #6 \ \
\end{array} \!\!\! \right)}

\nc{\partnlongvar}[6]{\left( \begin{array}{c|c|c} %fine
#1 \ & \ #3 \ & \ #5 \ \ \\ \cline{2-3}
#2 \ & \ #4 \ & \ #6 \ \
\end{array} \!\!\! \right)}

\nc{\partnlonger}[8]{\left( \begin{array}{c|c|c|c} %fine
#1 \ & \ #3 \ & \ #5 \ & \ #7 \ \ \\  \cline{4-4}
#2 \ & \ #4 \ & \ #6 \ & \ #8 \ \
\end{array} \!\!\! \right)}

\usepackage{hhline}

\begin{document}

\title[Diagram monoids and Graham--Houghton graphs] 
{Diagram monoids and Graham--Houghton graphs: idempotents and generating sets of ideals}  

\maketitle

\begin{center}

    James East\footnote{Centre for Research in Mathematics; School of Computing, Engineering and Mathematics; Western Sydney University; Locked Bag 1797, Penrith, NSW, 2751, Australia.  {\tt J.East@WesternSydney.edu.au}} and Robert~D.~Gray\footnote{School of Mathematics; University of East Anglia; Norwich NR4 7TJ, UK.  {\tt Robert.D.Gray@uea.ac.uk} \\ This author was partially supported by the EPSRC grant EP/N033353/1.} 
    
    \medskip
    
    \today

\end{center}

\begin{abstract}
We study the ideals of the partition, Brauer, and Jones monoid, establishing various combinatorial results on generating sets and idempotent generating sets via an analysis of their  Graham--Houghton graphs.  We show that each proper ideal of the partition monoid $\mathcal P_n$ is an idempotent generated semigroup, and obtain a formula for the minimal number of elements (and the minimal number of idempotent elements) needed to generate these semigroups.  In particular, we show that these two numbers, which are called the rank and idempotent rank (respectively) of the semigroup, are equal to each other, and we characterize the generating sets of this minimal cardinality.  We also characterize and enumerate the minimal idempotent generating sets for the largest proper ideal of $\mathcal P_n$, which coincides with the singular part of $\mathcal P_n$. Analogous results are proved for the ideals of the Brauer and Jones monoids; in each case, the rank and idempotent rank turn out to be equal, and all the minimal generating sets are described.  
We also show how the rank and idempotent rank results obtained, when applied to the corresponding twisted semigroup algebras (the partition, Brauer, and Temperley--Lieb algebras), allow one to recover formulae for the dimensions of their cell modules (viewed as cellular algebras) which, in the semisimple case, are formulae for the dimensions of the irreducible representations of the algebras.
As well as being of algebraic interest, our results relate to several well-studied topics in graph theory including the problem of counting perfect matchings (which relates to the problem of computing permanents of $\{0,1\}$-matrices and the theory of Pfaffian orientations), and the problem of finding factorizations of Johnson graphs. Our results also bring together several well-known number sequences such as Stirling, Bell, Catalan and Fibonacci numbers.
\end{abstract}

\section{Introduction}\label{sect_intro}
There has been a lot of interest recently in algebras with a basis consisting of diagrams that are multiplied in some natural diagrammatic way. Examples of such ``diagram algebras'' include the Brauer algebra \cite{Brauer1937}, Temperley--Lieb algebra \cite{Goodman1993}, and the Jones algebra \cite{Jones1994}. All of these examples arise in a natural way as subalgebras of the partition algebra \cite{Martin1994}, whose basis consists of all set-partitions of a $2n$-element set (see below for a formal definition). The partition algebra first appeared independently in the work of Martin \cite{Martin1991, Martin1994} and Jones \cite{Jones1994_2}. In both cases, their motivation for studying this algebra was as a generalization of the Temperley--Lieb algebra and the Potts model in statistical mechanics. Since its introduction, the partition algebra has received a great deal of attention in the literature; see for example \cite{
GW1995,
Haldel,
Enyang2013,
Guo2009,
Halverson2001,
Halverson2004,
Halverson2005,
Halverson2010,
Hartmann2010,
Konig1999, 
Martin1996,
Martin2013,
Martin1999,
Martin2004,
Wilcox2007,
Xi1999}.

All of the diagram algebras mentioned above are examples of cellular algebras, an important class of algebras introduced by Graham and Lehrer in \cite{GrahamLehrer1996}. The fact that these algebras are cellular allows one to obtain information about the semisimplicity of the algebra and about its representation theory, even in the non-semisimple case. 
In the partition, Brauer, and Temperley--Lieb algebras, the product of two diagram basis elements is always a scalar multiple of another basis element. Using this observation as a starting point, Wilcox \cite{Wilcox2007}, showed that these algebras are isomorphic to certain twisted semigroup algebras. By realizing the algebras in this way, many questions concerning the algebras can be related to questions for the corresponding semigroups. For instance, cellularity of the algebra can be deduced from various aspects of the structure of the monoid. The original study of cellular semigroup algebras may be found in \cite{East2006}; see also \cite{Guo2009,Wilcox2014,May2015,May2015_2} for some recent developments.  
Another example of how the study of these semigroups can give information about the associated algebras may be found in work of the first author \cite{East2011_1}, who gives presentations for the partition monoid and shows how these presentations give rise to presentations for the partition algebra; see also \cite{JE2016_2,JE2016_1,East2011_2}. 
A further example is given in the paper \cite{enum} where idempotents in the partition, Brauer and partial Brauer monoids are described and enumerated, and then the results are applied to determine the number of idempotent basis elements in the finite dimensional partition, Brauer and partial Brauer algebras; see also \cite{enum2}.

The corresponding semigroups---the partition, Brauer, and Jones monoids, and other related semigroups---have been studied, for instance, in \cite{DE2015_1,DEG2015,JE2016_1,JE2016_2,DE2016_1,MarMaz,ACHLV, KudMaz2006, Maz2002,Maz1998, enum,Aui2012, Aui2013, ADV, East2013, East2011_1, East2011_2, East2012, FitzGerald2011, FitzGerald2006, Maltcev2007}. Recently, the first author \cite{East2011_2} considered the subsemigroup generated by the set of idempotents in the partition monoid $\P_n$, showing in particular that every non-invertible element is expressible as a product of idempotents (we shall see in Theorem~\ref{thm_main1} below that this result holds more generally for any proper two-sided ideal of the partition monoid); presentations were also obtained in \cite{East2011_2}, and the infinite case was considered in \cite{East2012,East2013_6}. So the singular part of $\P_n$  is an idempotent generated semigroup, and this is a property that $\P_n$ has in common with several other naturally arising monoids. 
For instance, every non-invertible matrix from the full linear monoid $M_n(Q)$ of $n \times n$ matrices over an arbitrary division ring $Q$ is expressible as a product of idempotent matrices \cite{Erdos1967, laffey83}, and the same result is true for the full transformation semigroup of all maps from a finite set to itself \cite{Howie1966}. Presentations for certain idempotent generated semigroups may be found in \cite{JE2016_1,JE2016_2,East2013_3,East2011_2,East2013_2,Maltcev2007}.
More recently, in a significant extension of Erdos's result from \cite{Erdos1967}, Putcha \cite{putcha06} gave necessary and sufficient conditions for a reductive linear algebraic monoid to have the property that every non-unit is a product of idempotents. Another reason idempotent generated semigroups have received considerable attention in the literature is that they possess a universal property: every semigroup embeds into an idempotent generated semigroup \cite{Howie1966} (indeed, in an idempotent generated regular $*$-semigroup \cite{East2012}), and if the semigroup is (finite) countable it can be embedded in a (finite) semigroup generated by three idempotents \cite{Byleen1984}. 

The Graham--Houghton graph of a semigroup is a bipartite graph with one part indexed by the $\gr$-classes of the semigroup, the other part indexed by the $\gl$-classes, and edges corresponding precisely to those $\gh$-classes that contain idempotents (see Section~\ref{sec_small} for the definition of Green's relations). Graham introduced these graphs in \cite{Graham68} to study the idempotent generated subsemigroup of a $0$-simple semigroup.  Graham's results were later rediscovered by Houghton who gave them a topological interpretation \cite{Houghton1977}.  In the case that the semigroup is regular and idempotent generated, the connected components of this graph are in natural bijective correspondence with the $\gd$-classes of the semigroup. 
Graham's results show that these graphs are important tools for studying idempotent generated semigroups. 
More background on these graphs and their applications in semigroup theory may be found in \cite[Section~4.13]{SteinbergBook2009}.

In addition to being a fundamental tool for studying products of idempotents, there are several other reasons that motivate the problem of obtaining a better understanding of Graham--Houghton graphs. 
Firstly, these graphs are important  in the study of \emph{free idempotent generated semigroups}. There has  been a recent resurgence of interest in the study of these semigroups~\cite{GY2014,DDG2015,DG2015,Dolinka2012,Brittenham2009, DolinkaGray, Easdown1985,Easdown2010, Gray2012(2), Gray2012(1),Nambooripad1979}, with a particular focus on describing their maximal subgroups. The theory developed in \cite{Brittenham2009} shows that maximal subgroups of free idempotent generated semigroups are precisely the fundamental groups of Graham--Houghton complexes. The Graham--Houghton complex of an idempotent generated semigroup is a $2$-complex, whose $1$-skeleton is the Graham--Houghton graph of the semigroup, and which has $2$-cells glued in for each \emph{singular square} of idempotents (in the sense of Nambooripad \cite{Nambooripad1979}). Thus a necessary first step in determining these maximal subgroups is to obtain a  description of the underlying Graham--Houghton graph. This has been done with success, and maximal subgroups have been computed, for certain fundamental examples such as the full transformation monoid \cite{Gray2012(2)} and the full linear monoid \cite{DolinkaGray}. In contrast, currently nothing is known about the maximal subgroups of free idempotent generated semigroups arising from the partition, Brauer or Jones monoids. We hope the descriptions we obtain here of the Graham--Houghton graphs of the $\gd$-classes of these semigroups will help with this research program. 

Graham--Houghton graphs also play an important role in the representation theory of finite semigroups. 
Fundamental results of Clifford \cite{Clifford1942,Clifford1960}, Munn \cite{Munn1955,Munn1957} and Ponizovski{\u\i} \cite{Ponizovskii1956}
show that the irreducible representations of any finite semigroup can be parametrised in terms of the 
irreducible representations of its maximal subgroups. 
A full account of this theory may be found in \cite[Chapter~5]{CliffordAndPreston}. 
A short modern proof of the Clifford-Munn-Ponizovski{\u\i} theorem is given in \cite[Theorem~7]{Ganyushkin2009}. 
This result gives a bijection between the irreducible representations of $S$ and irreducible representations of the maximal subgroups $G_{e_J}$ of $S$, where 
$\{ e_J : J \in \mathcal{U}(S) \}$ is an idempotent transversal of the set $\mathcal{U}(S)$ of regular $\mathcal{J}$-classes of $S$. 
The work of Munn and Ponizovski{\u\i} also gives necessary and sufficient conditions for the semigroup algebra $KS$ over a field $K$ to be semisimple. 
These conditions use Rees's theorem and the notion of a Rees matrix semigroup. 
Associated to a regular $\gj$-class $J$ of a finite semigroup $S$ is a (so-called) \emph{Rees (0-)matrix semigroup} whose structure is governed by a group $G_J$ and a ``structure matrix'' $P$ with entries from $G_J\cup\{0\}$ (see Section \ref{sec_small} for more precise definitions). 
The locations of the non-zero entries of $P$ are determined by the edges of the Graham--Houghton graph.  
The semigroup algebra $KS$ is semisimple if and only if $S$ is a regular semigroup and for each $\gj$-class $J$, each $KG_J$ is semisimple and the structure matrix corresponding to $J$ is square and invertible. (Note that $\gj = \gd$ here since $S$ is finite.)  
%
%
%
%More precisely, let $S$ be a finite semigroup, let $Q = \{ e_J : I \in \mathcal{U}(S) \}$ be an idempotent transversal of the set $\mathcal{U}(S)$ of regular $\mathcal{J}$-classes of $S$, and let $G_J$ denote the maximal subgroup $G_{e_J}$. Then for each $\gj$-class $J$ there is a method by which from any simple $KS$-module (where $K$ is a field) produces a simple $K{G_J}$-module, and vice-versa (see Ganyushkin et al. \cite[Theorem~7]{Ganyushkin2009}  for a precise statement, and a modern proof of this result). This gives a 
%
%
%, $J \in \mathcal{U}(S)$.  
%%
%%
%
%%
%%
%%
%%
%

These general results of course apply to the semigroup algebras of the monoids studied in this paper. 
In fact, as mentioned above, the algebras associated with the semigroups studied here actually all belong to a special class called 
cellular algebras. It was shown by Wilcox \cite{Wilcox2007} that the partition, Brauer, and Temperley--Lieb algebras can 
be realised as twisted semigroup algebras of, respectively, the partition, Brauer and Jones monoids. There are twisted 
semigroup algebra analogues of the general results mentioned above: e.g., an analogous result giving necessary and sufficient conditions for semisimplicity; see \cite[Theorem~15]{Wilcox2007}. In the particular case of the diagram algebras we consider here, the results of Wilcox show how a cellular structure for
the twisted semigroup algebra can be built from cellular structures of the group algebras of the maximal subgroups of the semigroup. 
Both the size and the structure of the Graham--Houghton graphs then encodes important information about these cellular structures, and thus about the 
corresponding diagram algebras and their representation theory. 
For example, the sizes of the Graham--Houghton graphs relate to the dimensions of the cell modules of the algebras. 
Indeed, in each case the Graham--Houghton graph is a balanced bipartite graph with each part of the bipartition having size equal to the number of $\gr$-classes in the corresponding $\gj$-class. These numbers can be used to relate the dimensions of cell modules for the twisted semigroup algebra to the dimensions of the cell modules corresponding to the group algebras of the maximal subgroups (see Equation~\ref{eq_connection}).  
%
%
%Specifically, in the semisimple case, these numbers give the dimensions of the irreducible representations. 
%
%
In this way, our results (specifically, Theorems~\ref{thm_main1},~\ref{thm_main3} and~\ref{thm_main4}, which count the number of $\gr$- and $\gl$-classes in the $\gj$-classes of the partition, Brauer and Jones monoids) allow for alternative derivations of (and in some cases, alternative formulae for) the dimensions of the cell modules of the partition, Brauer and Temperley-Lieb algebras; see \cite{Halverson2005,Pan95,Wenzl88,GrahamLehrer1996, Martin1990_2,Martin1990, GL1998,BR1999,EG1999,Martin1987,Grood1999,KX1998,KX2001} for earlier derivations.  
Full details of this will be given in Section~\ref{sec_irreps}.
%
%
%
%See also Remark \ref{rem:irreducibles}.
%
%
%
%
% 
The structure of the Graham--Houghton graph is also relevant when considering products of basis elements of the cellular 
basis of these twisted semigroup algebras. Indeed, the elements $C_{\mfs, \mft}^{\lambda}$ of the cellular bases of these diagram algebras are all sums over elements from certain $\gh$-classes in a corresponding diagram semigroup. A key step in understanding the representation theory of these cellular algebras is to investigate the bilinear form
$
\phi_\lambda: W(\lambda) \times W(\lambda) \rightarrow R$, 
$
\phi_\lambda(C_\mfs, C_\mft) = \phi(\mfs,\mft), 
$
which is defined by considering products of the form 
$C_{\mfs_1, \mft_1}^{\lambda}C_{\mfs_2, \mft_2}^{\lambda}$ (see Section~\ref{sec_irreps} for details).  
In order to describe this bilinear form 
%
%in a twisted semigroup algebra (e.g., in the partition, Brauer or Temperly--Lieb algebra) 
%
it is necessary to know when such a product 
``moves down'' in the algebra, and in particular whether this sum moves down in the $\gj$-order of the semigroup. 
This is determined precisely by the location of the idempotents in the $\gj$-class of the elements arising in the sums defining $C_{\mfs_1, \mft_1}^{\lambda}$ and $C_{\mfs_2, \mft_2}^{\lambda}$. The Graham--Houghton graphs of these $\gj$-classes record exactly this information.  

In this paper we shall primarily be concerned with using Graham--Houghton graphs as a tool to study generating sets, idempotents, and subsemigroups generated by certain sets of idempotents, in the partition monoid and some of its key submonoids. Idempotent generated semigroups have been investigated using this approach in, for example \cite{GrayHowieIssuePaper, Howie1978, Howie1990, Levi2002, Levi2003}. There are numerous results in the literature concerning the problem of finding small generating sets (and generating sets of idempotents) for certain naturally arising semigroups, most often semigroups of transformations, matrix semigroups, and more generally semigroups of endomorphisms of various combinatorial or algebraic structures. One of the earliest results of this kind may be found in the work of Howie \cite{Howie1978}, where minimal idempotent generating sets of $\mathrm{Sing}_n$, the singular part of the full transformation semigroup, were classified (by associating a tournament in a natural way with certain sets of idempotents); in subsequent work with McFadden \cite{Howie1990}, the ranks and idempotent ranks of arbitrary ideals of the full transformation monoid were calculated. Since then, many more results of this flavour have appeared in the literature \cite{DEM2015,DE2014,DE2014_2,DE2015,Dawlings1981,Dawlings1982,Fountain1992,Fountain1993,Garba1990,Gomes1987,Gomes1992,Gray2007,Gray2008,Levi2002,Levi2003}.
Motivated in part by this work, for each of the proper ideals of the subsemigroups of the partition monoid we consider, we shall prove that the ideal is idempotent generated, and then investigate its generating sets, and idempotent generating sets, with a particular emphasis on describing small generating sets (and idempotent generating sets). 
We shall establish formulae for the smallest number of elements needed to generate the ideal (called the \emph{rank}) and also the smallest number of idempotents required to generate the ideal (the \emph{idempotent rank}). In all cases we will show that these two numbers coincide. This fact, together with some general results given in Section~\ref{sect_0-simple}, can then be used to completely describe all the minimal generating sets for these ideals. We then go on to use  Graham--Houghton graphs (and certain quotients of them) to study the  idempotent generating sets in more detail. Specifically we investigate the problem of whether we can count all the idempotent generating sets of minimal cardinality. We shall see that this problem is equivalent to counting the number of perfect matchings in the corresponding Graham--Houghton graph. Counting perfect matchings in bipartite graphs is a well-studied problem in combinatorics, and relates to the problem of computing permanents of $\{0,1\}$-matrices and the theory of Pfaffian orientations; see \cite{Robertson1999}.  
We note that throughout the paper by a \emph{minimal (idempotent) generating set} we will always mean a generating set (idempotent generating set) of minimal possible cardinality, as opposed to simply a generating set that is minimal with respect to set-theoretic inclusion (these two notions of minimality do not always coincide; see Example \ref{Example3}). 

%As indicated above, 
%an underlying philosophy behind studies of diagram monoids is that discoveries regarding the combinatorics of these semigroups can often be translated into corresponding statements about the associated twisted semigroup algebras: the partition, Brauer and  Temperley--Lieb algebras. 
%This is indeed the case here, and in Section~\ref{sec_irreps} we shall explain how the rank and idempotent rank results we obtain in this paper, when applied to the corresponding twisted semigroup algebras, give rise to known formulae for the dimensions of the cell modules of these algebras (viewed as cellular algebras) which, in the semisimple case, are formulae for the dimensions of the irreducible representations of the algebras. For example, in Proposition~\ref{prop_partition_alg} we give a formula that relates ranks of ideals of the partition monoid, dimensions of irreducible representations of partition algebras, and dimensions of irreducible representations of symmetric group algebras. Similar observations are also made for both the Brauer and Temperley--Lieb algebras (see Propositions~\ref{prop_dim_B} and \ref{prop_dim_TL}). 

Including this introduction, the article comprises ten sections and is structured as follows. 
In Section~2, we present some results of Howie about idempotent generators in the full transformation monoid $\T_n$. We also take the opportunity to correct a mistake in the formula given by Howie \cite{Howie1978} for the number of distinct minimal idempotent generating sets for the singular part of $\T_n$. In Section~3, we present some general theory showing how (idempotent) generating sets for finite semigroups can be related to (idempotent) generating sets of their principal factors. In Section~4, we give some background on the theory of Graham--Houghton graphs, and how they can be used to investigate (idempotent) generating sets. In Section~5, we develop the theory introduced in earlier sections, tailoring it to the study of regular $*$-semigroups. This is a class of regular semigroups that includes the partition monoid and all of the submonoids of the partition monoid considered in this article. In Section~6, we relate the ideas of Sections~4 and 5, showing how certain graphs defined in Section~5 are isomorphic to natural quotient graphs of the Graham--Houghton graphs discussed in Section~4. Specifically, we present a result that gives necessary and sufficient conditions for a set of idempotents to generate a given regular $*$-semigroup, given in terms of an associated two-coloured directed graph. In Section~7, we turn our attention to the partition monoid $\P_n$. We prove that the proper two-sided ideals of $\P_n$ are idempotent generated, and we give a formula for the rank and idempotent rank, showing that these are equal. We completely characterize the minimal generating sets, and minimal idempotent generating sets, and for the singular part of $\P_n$ we also enumerate the minimal idempotent generating sets. We also apply the results of Section~6 to give necessary and sufficient conditions for a set of idempotents to generate the singular part of~$\P_n$. The Brauer monoid is the subject of Section~8, where, as for the partition monoid, we consider its proper two sided ideals, showing they are idempotent generated, computing the idempotent rank, showing that it is equal to the rank, and characterizing the minimal generating sets. We also establish a bijection between the minimal idempotent generating sets of its singular part with certain factorizations of particular Johnson graphs. 
In Section~9, we consider the Jones monoid, where we establish analogous results for (idempotent) generating sets of its two-sided ideals, and also show that the number of distinct minimal idempotent generating sets is given by the $n$th Fibonacci number.   
Finally, in Section~10 we explain the connection between ranks (and idempotent ranks) of ideals of partition, Brauer and Jones monoids, and dimensions of cell modules (and irreducible representations) of the corresponding partition, Brauer, and Temperley--Lieb algebras.

\section{
Idempotent generators in the full transformation monoid
}
\label{sec_Howie}

In this section, we summarize some results of Howie and his collaborators \cite{Gomes1987,Howie1966, Howie1978, Howie1990} on idempotent generators in the full transformation monoid. We do this in part because it will give a flavour of the kind of results we aim to obtain later on for ideals of the partition monoid and related monoids. 
We have also included this material so that we can correct a mistake in a formula given by Howie in \cite{Howie1978}. 
Denote by $\T_n$ and $\S_n$ the full transformation semigroup and symmetric group on the set $\bn=\{1,\ldots,n\}$, respectively.
For $\al\in\T_n$ and $i\in\bn$, we write $i\al$ for the image of $i$ under $\al$; in this way, transformations compose from left to right.

As mentioned in the introduction, the study of idempotent generated semigroups dates back to the paper \cite{Howie1966}, where Howie shows that the semigroup $\mathrm{Sing}_n=\T_n \setminus \S_n$ is idempotent generated. In particular, he shows that  $\mathrm{Sing}_n$ is generated by the set of idempotents with image size $n-1$. In a later paper \cite{Howie1978}, he went on to study generating sets of idempotents in $\mathrm{Sing}_n$ in more detail, giving a combinatorial characterization of the minimal idempotent generating sets of the semigroup, and counting the number of such idempotent generating sets. In a subsequent work with McFadden \cite{Howie1990}, the ranks and idempotent ranks of the proper ideals of $\T_n$ were obtained.  See \cite{GomesRuskuc2014} for a historical overview of Howie's work, including this particular research program.

In more detail, let $S = \mathrm{Sing}_n = \T_n \setminus \S_n$. Set
\[
F = \{ \alpha \in E(S) : |\im \alpha| = n-1 \}, 
\]
where $E(S)$ denotes the set of idempotents in the semigroup $S$.  In \cite{Howie1966}, it is shown that $\mathrm{Sing}_n=\langle F \rangle $. Each $\beta \in F$ has the property that $i \beta = i$ for all $i \in \im \beta$ (since $\beta^2 = \beta$) and since $|\im \beta | = n-1$ there is exactly one $i \in [n]$ such that $i \beta = j$ where $j \neq i$. In this case, we denote $\beta$ by 
$\binom ij$, meaning $\beta$ maps $i$ to $j$ and fixes every other point. 
Defining relations were given for $\text{Sing}_n$ with respect to the generating set $F$ in \cite{East2013_2}; see also \cite{East2013_3,East2013_5,East2011_2,Maltcev2007} for presentations of other singular semigroups.

Let $X \subseteq F$.  Define a digraph $\Gamma(X)$ with vertex set $[n]$ and an arc $j \rightarrow i$ if and only if $\binom ij \in X$.  With this notation, in \cite{Howie1978} the following result is obtained (recall that a digraph is a \emph{tournament} if for each pair of distinct vertices $x,y$, the graph contains exactly one of the edges $x\to y$ or $y\to x$).  

\begin{theorem}\label{thm_How}
Let $X$ be a set of idempotents from $F$ in $\T_n$, where $n\geq3$. Then $X$ is a minimal (idempotent) generating set for $\mathrm{Sing}_n$ if and only if $\Gamma(X)$ is a strongly connected tournament. 
\end{theorem}

In particular, counting the number of arcs in a strongly connected tournament, one sees that 
\[
\rank(\mathrm{Sing}_n) = \idrank(\mathrm{Sing}_n) = \binom n2 = \frac{n(n-1)}{2}. 
\]
It follows from the correspondence given by Theorem~\ref{thm_How} that the number of minimal idempotent generating sets for $\mathrm{Sing}_n$ is precisely the number of strongly connected labeled tournaments on $n$ vertices.  The formula for this number stated in the paper \cite{Howie1978} is actually incorrect. The correct formula may be found in a paper of Wright \cite{Wright1970}; it may also be found as number sequence A054946 on \cite{Sloan:OEIS}. Let $w_n$ denote the number of strongly connected labeled tournaments on $n$ vertices. Then $w_n$ is given by the recurrence
\[
w_1 = 1 \quad\text{and}\quad w_n = F_n - \sum_{s=1}^{n-1} \binom ns w_s F_{n-s} \quad\text{for $n\geq2$,}
\]
where $F_k = 2^{\scriptscriptstyle{\binom k2}}= 2^{k(k-1)/2}$. The first few terms in this sequence are displayed in Table~\ref{tab_wk}.   So, via the correspondence established by Theorem~\ref{thm_How}, for $n \geq 3$, $w_n$ is the number of distinct minimal idempotent generating sets for $\mathrm{Sing}_n$.  Arbitrary minimal generating sets of $\mathrm{Sing}_n$ were characterized in \cite{Ayik2013}.  Arbitrary idempotent generating sets of $\mathrm{Sing}_n$ were classified and enumerated in \cite{DE2014}.

\begin{table}[ht]%
\begin{center}
\begin{tabular}{|c|cccccccccc|}
\hline
$n$ & $1$ & $2$ & $3$ & $4$ & $5$ & $6$ & $7$ & $8$ & $9$ & $10$ \\
\hline
$w_n$  & $1$ &  $0$ &  $2$  & $24$  & $544$  & $22320$  & $1677488$  & $236522496$  & $64026088576$ & $33832910196480$ \\
\hline
\end{tabular}
\end{center}
\caption{The sequence $w_n$.  For $n\geq3$, $w_n$ is equal to the number of minimal idempotent generating sets for $\text{Sing}_n=\T_n\setminus\S_n$.}
\label{tab_wk}
\end{table}

In subsequent work, Howie and McFadden \cite{Howie1990} investigated the ideals of $\T_n$.  These are the sets
\[
I_r(\T_n) = \{ \alpha \in \T_n : |\im\al| \leq r \}, \quad\text{$1\leq r\leq n$.}
\]
Typically, the ideal $I_r(\T_n)$ is denoted $K(n,r)$, but we use the current notation for consistency with later usage.

\begin{theorem}\label{thm_How2}
For $1\leq r\leq n-1$, the ideal $I_r(\T_n)$ is idempotent generated, and 
\[
\rank(I_r(\T_n)) = \idrank(I_r(\T_n)) = \begin{cases}
n &\text{if $r=1$}\\
S(n,r) &\text{if $2\leq r\leq n-1$,}
\end{cases}
\]
where $S(n,r)$ denotes the Stirling number of the second kind.
\end{theorem}

These (idempotent) ranks are given in Table \ref{tab_Snk}.

\begin{table}[ht]%
\begin{tabular}{|c|cccccccccc|}
\hline
$n$ $\setminus$ $r$		&1	&2	&3	&4	&5	&6	&7 &8 &9 &10     \\
\hline
%0	&1	&	&	&	&	&	&	&	&&&      \\
1		&1	&	&	&	&	&	&	&&&       \\
2		&1	&1	&	&	&	&	&	&&&       \\
3		&1	&3	&1	&	&	&	&	&&&       \\
4		&1	&7	&6	&1	&	&	&	&&&       \\
5		&1	&15	&25	&10	&1	&	&	&&&       \\
6		&1	&31		&90		&65		&15	&1	&	&&&       \\
7		&1	&63		&301		&350		&140		&21	&1 &&& \\
8 	&1	&127		&966		&1701		&1050		&266		&28	&1		& & \\
9 	&1	&255		&3025		&7770		&6951		&2646		&462		&36	&1	& \\
10 	&1	&511		&9330		&34105		&42525		&22827		&5880		&750		&45		&1 \\
\hline
\end{tabular}
\caption{
The Stirling numbers of the second kind, $S(n,r)$.  For $2\leq r\leq n-1$, $S(n,r)$ is equal to $\idrank(I_r(\T_n)) = \rank(I_r(\T_n))$.
}
\label{tab_Snk}
\end{table}

\section{Small generating sets: ideals and principal factors}
\label{sec_small}

We begin this section with some general observations about generating sets of semigroups that will explain why, in many situations, the problem of finding small generating sets for a given semigroup $S$ reduces to the problem of finding small generating sets for a certain principal factor of $S$.

Let $S$ be a semigroup, let $S^1$ denote $S$ with an identity element adjoined (if necessary), and let $u,v \in S$. Recall that Green's relations $\gr$, $\gl$, $\gj$, $\gh$ and $\gd$ are equivalence relations on $S$ that reflect the ideal structure of the semigroup $S$ and are
defined by 
\[
u \gr v  \ \Leftrightarrow  \ uS^1 = vS^1, 
\quad 
u \gl v \  \Leftrightarrow \  S^1u = S^1v,
\quad
u \gj v \  \Leftrightarrow \  S^1uS^1 = S^1vS^1, 
\]
\[
\gh=\gr\cap\gl,\quad \gd = \gr\vee\gl=\gr \circ \gl = \gl \circ \gr. 
\]
All of the semigroups we consider in this paper will be finite, and in finite semigroups the relations $\gj$ and $\gd$ coincide. 
The $\gj$ relation gives rise to a natural preorder on $S$ given by 
\[
u \leq_{\gj} v \ \Leftrightarrow \ S^1 u S^1 \subseteq S^1 v S^1.
\]
This preorder induces in a natural way a partial order on the set $S / \gj$ of $\gj$-classes of~$S$. By a maximal $\gj$-class of a semigroup $S$ we mean a $\gj$-class that is maximal in this partially ordered set of $\gj$-classes of $S$.  An element $s \in S$ is called (von Neumann) \emph{regular} if $s \in s S s$. If every element of $S$ is regular we say that $S$ is a \emph{regular semigroup}. All the semigroups considered in this paper will be regular.

Recall that the principal factors of a semigroup $S$ are the basic building blocks of the semigroup, and are obtained by taking a $\gj$-class $J$ and forming a semigroup $J^* = J \cup \{ 0 \}$, where $0$ is a symbol that does not belong to $J$, with multiplication given, for $s,t\in J\cup\{0\}$, by
\[
s * t = 
\begin{cases}
st & \mbox{if $s$, $t$ and $st$ belong to $J$} \\
0 & \mbox{otherwise}. 
\end{cases}
\] 
If $S$ is finite, then every principal factor $J^*$ is either a semigroup with zero multiplication or a completely $0$-simple semigroup (see the next paragraph).  If, in addition, $S$ is regular then every principal factor $J^*$ is completely $0$-simple.

We recall that a subset $I$ of a semigroup is an \emph{ideal} if $SI\sub S$ and $IS\sub S$.
A semigroup $S$ with zero is called \emph{$0$-simple} if $\{0\}$ and $S$ are its only ideals (and $S^2 \neq \{ 0 \}$). It is said to be \emph{completely $0$-simple} if it is $0$-simple and has $0$-minimal left and right ideals. In particular every finite $0$-simple semigroup is completely $0$-simple. Howie's monograph \cite{howie} may be consulted for more information on completely $0$-simple semigroups and other background on semigroup theory.  An important result for our purposes is the Rees Theorem \cite[Theorem 3.2.3]{howie} which says that any finite completely $0$-simple semigroup is isomorphic to a (so-called) \emph{regular Rees (0-)matrix semigroup} $\M[G;I,\Lambda;P]$ defined as follows.  First, $G$ is a finite group, $I$ and $\Lambda$ finite sets, and $P=(p_{\lam,i})$ a $\Lambda\times I$ matrix with entries in $G\cup\{0\}$, where $0$ is assumed not to belong to $G$, and such that each row and each column of $P$ contains at least one non-zero entry.  The semigroup $\M[G;I,\Lambda;P]$ is then defined on the set $(I\times G\times \Lambda)\cup\{0\}$, and with product defined by setting
\[
(i,g,\lambda)(j,h,\mu)=(i,gp_{\lambda, j}h,\mu) \qquad\text{if $p_{\lambda, j}\not=0$}
\]
and with all other products defined to be $0$.

Recall that the \emph{Rees quotient} of $S$ by an ideal $I$, denoted $S/I$, is defined on underlying set $(S\setminus I)\cup\{0\}$, where $0$ is a symbol not belonging to $S$, and with product $\circ$ defined, for $s,t\in(S\setminus I)\cup\{0\}$, by
\[
s\circ t =\begin{cases}
st &\text{if $s$, $t$ and $st$ belong to $S\setminus I$}\\
0 &\text{otherwise.}
\end{cases}
\]
A principal factor $J^*$ (as defined above) is of course (isomorphic to) a special case of this construction: we take $S=\la J\ra$ and $I=\la J\ra\setminus J$.

The following observation lies at the heart of what is to follow: it is similar to \cite[Lemma~3.2]{Quick2012} (which is proved for monoids but may easily be adapted for semigroups); its proof is routine, and is omitted.

\begin{lemma}
\label{lem_easy}
Let $S$ be a finite semigroup and let $I$ be an ideal of $S$. Moreover, suppose that $S$ is generated by $U = S \setminus I$. Let $A$ be a generating set of minimal cardinality for $S$. Then $A \subseteq U$ and 
\[
\rank(S) = \rank(S/I).
\]
Furthermore, $S$ is idempotent generated if and only if $S/I$ is idempotent generated, in which case
\[
\idrank(S) = \idrank(S/I).
\]
\end{lemma}

A special case of the above situation is when $U$ is a maximal $\gj$-class of $S$.

\begin{lemma}
\label{lem_maxJ}
Let $S$ be a semigroup with a maximal $\gj$-class $J$ such that $\lb J \rb = S\not=J$. Then
\[
\rank(S) = \rank(J^*).
\]
Furthermore, $S$ is idempotent generated if and only if $J^*$ is idempotent generated, in which case
\[
\idrank(S) = \idrank(J^*).
\]
\end{lemma}
\begin{proof}
This is a straightforward consequence of the preceding lemma since $S\sm J$ is an ideal of $S$, and the quotient $S/(S\sm J)$ is isomorphic to the principal factor $J^*$. 
\end{proof}

\begin{remark}
Throughout this article, when we consider semigroups with a zero element (e.g., completely $0$-simple semigroups) we will use $\langle X \rangle$ to denote all the elements that can be written as products of elements of $X$, together with the zero element if it is not already generated. In almost all cases where we apply results for completely $0$-simple semigroups to principal factors of semigroups, including the partition monoid and related semigroups, this convention will not make any difference since all the completely $0$-simple semigroups that arise will have the property that the zero element is generated by the other elements of the semigroup; the only exception is in the bottom ideal (which consists of a single $\gj$-class).  With this convention, we need not require that $S\not=J$ in the statement of Lemma~\ref{lem_maxJ}.
\end{remark}

Principal factors naturally come into play when analyzing idempotent generated semigroups, in part because of the following classical result of FitzGerald, which is distilled from the proof of \cite[Lemma 2]{FitzGerald1972}; see also \cite[Lemma 4.13.1]{SteinbergBook2009}.

\begin{theorem}\label{thm_FitzGerald}
Let $S$ be a semigroup and let $e_1, \ldots, e_m \in E(S)$. If $e_1  \cdots e_m $ is regular, then there is a sequence of (not necessarily distinct) idempotents 
$f_1, \ldots, f_m \in E(S)$ such that $f_i \,\gj e_1 \cdots e_m$ for all $i$ and
\[
f_1  \cdots f_m
=
e_1  \cdots e_m. 
\]
\end{theorem}

The next result quickly follows.

\begin{corollary}\label{cor_FitzGerald}
Let $S$ be a finite regular idempotent generated semigroup. Then for every $\gj$-class $J$ of $S$, the principal factor $J^*$ is an idempotent generated completely $0$-simple semigroup. 
\end{corollary}

\begin{example}[Full transformation monoid]
As before, let $\T_n$ denote the full transformation monoid on $n$ points. The proper two-sided ideals of $\T_n$ are the sets
\[
I_r = I_r(\T_n) = \{ \alpha \in \T_n : \rank(\alpha) \leq r \}, \quad 1 \leq r < n.
\]
(The \emph{rank}, $\rank(\al)$, of a transformation $\al\in\T_n$ is equal to $|\im(\al)|$.)  Let 
\[
J_r = J_r(\T_n) = \{ \alpha \in \T_n : \rank(\alpha) = r \}.
\]
Then it follows from \cite[Lemma~2.2]{Fountain1992} that $I_r$ is an idempotent generated semigroup with a unique maximal $\gj$-class $J_r$, and $I_r = \lb J_r \rb$. It follows that $J_r^*$ is idempotent generated, with 
\[
\rank(I_r) = \rank(J_r^*)
\quad \mbox{and} \quad
\idrank(I_r) = \idrank(J_r^*), 
\]
where, since $\T_n$ is regular, $J_r^*$ is a completely $0$-simple semigroup. 
\end{example}

\begin{example}[Full linear monoid]
Let $M_n(F)$ denote the full linear monoid of all $n \times n$ matrices over a finite field $F$. The proper two-sided ideals of $M_n(F)$ are the sets 
\[
I_r = I_r(M_n(F)) = \{ A \in M_n(F) : \rank(A) \leq r \}, \quad 0 \leq r < n.
\]
(The \emph{rank}, $\rank(A)$, of a matrix $A\in M_n(F)$ is equal to the dimension of its row or column space.)  Let 
\[
J_r = J_r(M_n(F)) = \{ A \in M_n(F) : \rank(A) = r \}.
\]
Then it also follows from \cite[Lemma~2.2]{Fountain1992} that $I_r$ is an idempotent generated semigroup with a unique maximal $\gj$-class $J_r$ and $I_r = \lb J_r \rb$. It follows that $J_r^*$ is idempotent generated, with 
\[
\rank(I_r) = \rank(J_r^*)
\quad \mbox{and} \quad
\idrank(I_r) = \idrank(J_r^*), 
\]
where, since $M_n(F)$ is regular, $J_r^*$ is a completely $0$-simple semigroup. 
\end{example}

\begin{remark}
The two examples above admit a common generalization to endomorphism monoids of finite dimensional \emph{independence algebras} \cite{Fountain1992}; see also \cite{Fountain1993} for a study of the infinite dimensional case.
\end{remark}

\begin{example}[Partition monoid]\label{ex_partition}
Let $\P_n$ denote the partition monoid. There is a natural notion of \emph{rank} for the elements of $\P_n$ (see below for the definition). 
The proper two-sided ideals of $\P_n$ are the sets
\[
I_r = I_r(\P_n) = \{ \alpha \in \P_n : \rank(\alpha) \leq r \}, \quad 0 \leq r < n.
\]
Let
\[
J_r = J_r(\P_n) = \{ \alpha \in \P_n : \rank(\alpha) = r \}.
\]
Then $I_r$ is an idempotent generated semigroup with unique maximal $\gj$-class $J_r$ and $I_r = \lb J_r \rb$ (see Lemma~\ref{lem_dropdownPn} below). Since $\P_n$ is regular, it follows that $J_r^*$ is an idempotent generated completely $0$-simple semigroup satisfying
\[
\rank(I_r) = \rank(J_r^*)
\quad \mbox{and} \quad
\idrank(I_r) = \idrank(J_r^*). 
\]
We will also see that similar statements are true of the Brauer and Jones monoids $\B_n$ and~$\J_n$.
\end{example}

Thus, in all the examples above---$I_r(\T_n)$, $I_r(M_n(F))$ and $I_r(\P_n)$---the problem of determining rank and idempotent rank reduces to answering the same question for some idempotent generated completely $0$-simple semigroup, so let us now turn our attention to this class. 

\section{Idempotent generated completely $0$-simple semigroups}\label{sect_0-simple}

The study of idempotent generated completely $0$-simple semigroups is a classical topic, with early results by Graham, Houghton, Howie, etc. Here we summarize some results from \cite{Gray2008}. 

Let $S = \M[G;I,\Lambda;P]$ be a finite completely $0$-simple semigroup represented as a Rees matrix semigroup over a group $G$ with $I$ indexing the $\gr$-classes of $S$, $\Lambda$ indexing the $\gl$-classes, and structure matrix $P$ with entries from $G \cup \{ 0 \}$. Our interest is in the case that $S$ is idempotent generated. In this situation the problem of determining the rank has a straightforward answer.  The next result has an easy inductive proof; see \cite[Theorem~2.4]{Gray2005} and \cite[Lemma~2.2, Lemma~2.3]{Gray2008}. 

\begin{lemma}
\label{lem_rank}
Let $S = \M[G;I,\Lambda;P]$ be a finite completely $0$-simple semigroup. If $S$ is idempotent generated, then
\[
\rank(S) = \max\{|I|,|\Lambda|\}. 
\]
\end{lemma} 

In particular, in all of the examples discussed in Section \ref{sec_small}, the rank is given simply by counting the number of $\gr$-classes and $\gl$-classes in each $\gd = \gj$-class and choosing the larger of the two numbers.

Now, of course, in general for an idempotent generated semigroup $S$, we have
\begin{equation}
\label{eqn_star}
\idrank(S) \geq \rank(S).
\end{equation}
It is easy to construct examples (even when $S$ is a finite completely $0$-simple semigroup) where $\rank(S)$ and $\idrank(S)$ are not equal; see for example \cite[Example~2.7]{Gray2008}.  A natural question, therefore, is that of when they are equal.  The main result of \cite[Section~2]{Gray2008} gives necessary and sufficient for equality in equation \eqref{eqn_star} where $S$ is an arbitrary idempotent generated completely $0$-simple semigroup; see \cite[Theorem~2.16]{Gray2008}. We shall not need the full generality of those results here but shall concentrate on a particularly nice special situation: namely, the case when $|I| = |\Lambda|$. As we shall see below, this condition is satisfied for the principal factors of the partition, Brauer and Jones monoids.  

Before stating the results, we first need to introduce some definitions. The key idea is that the question of whether $\rank(S) = \idrank(S)$ comes down to a consideration of the way that the idempotents are distributed in $S$; that is, the position of the non-zero elements in the structure matrix $P$.

\begin{definition}[Graham--Houghton graph]\label{GHGraph_defn}
Let $S=\M[G;I,\Lambda;P]$ be a Rees matrix semigroup. %?
Let $\Delta(S)$ denote the bipartite graph with vertex set $I \cup \Lambda$ and an edge connecting $i \in I$ to $\lambda \in \Lambda$ if and only if $p_{\lambda i} \neq 0$.  We call $\Delta(S)$ the \emph{Graham--Houghton graph} of $S$.
\end{definition}
Thus, the edges in $\Delta(S)$ record the positions of the idempotents in $S$. 
 Graham--Houghton graphs will be discussed in more detail in Section~\ref{sec_GHGraphs} below. 

Let $\Gamma$ be a graph and let $v \in V(\Gamma)$ be a vertex of $\Gamma$. Define the neighbourhood of $v$ to be the set
\[
\Gamma(v) = \{ w \in V(\Gamma) : w \sim v \}.
\]
Here we use $\sim$ for the adjacency relation in the graph. 
More generally, for a subset $W$ of $V(\Gamma)$, the neighbourhood is $W$ is the set
\[
\Gamma(W)  =  \bigcup_{v \in W} \Gamma(v) \\
                      =  \{ w \in V(\Gamma) : w \sim v \ (\exists v \in W) \}.
\]
In the case that $S=\M[G;I,\Lambda;P]$ is an idempotent generated completely $0$-simple semigroup with $|I| = |\Lambda|$, we have $\rank(S) = |I| = |\Lambda|$. If $\idrank(S) = \rank(S)$, then there is an idempotent generating set with size $|I|=|\Lambda|$ which, since every generating set must clearly intersect every non-zero $\gr$-class and every non-zero $\gl$-class of $S$, must correspond to a perfecting matching in the bipartite graph $\Delta(S)$. 

Thus, a necessary condition for equality in equation \eqref{eqn_star} in the case $|I|=|\Lambda|$ is that $\Delta(S)$ admits a perfect matching. It is a classical result from combinatorics, due to Hall \cite{HallBook}, that a balanced bipartite graph admits a perfect matching if and only if it satisfies the so-called \emph{Hall's condition}.  In what follows, when we refer to a bipartite graph $\Gamma=X\sqcup Y$, we mean that the vertex set of $\Gamma$ is the disjoint union of $X$ and $Y$, and all edges of $\Gamma$ join a vertex from $X$ and a vertex from $Y$.

\begin{theorem}[Hall \cite{HallBook}]
A bipartite graph $\Gamma = X \sqcup Y$ with $|X| = |Y|$ has a perfect matching if and only if the following condition is satisfied\emph{:}
\begin{equation}
|\Gamma(A)| \geq |A| \ \mbox{for all} \ A \subseteq X.  \tag{HC}
\end{equation}
\end{theorem}
Although necessary, this property is not sufficient to guarantee $\idrank(S) = \rank(S)$ here, and what we actually require is a slight strengthening of Hall's condition:
\begin{definition}
A bipartite graph $\Gamma = X \sqcup Y$ with $|X|=|Y|$ is said to satisfy the \emph{strong Hall condition} if it satisfies
\begin{equation}
|\Gamma(A)| > |A| \ \mbox{for all} \ \varnothing\not=A \subsetneq X.  \tag{SHC}
\end{equation}
\end{definition}
\begin{definition}[Sparse cover]
We call a subset $A$ of a Rees matrix semigroup %?
$S = \M[G;I,\Lambda;P]$ a \emph{sparse cover} of $S$ if $|A| = \max\{ |I|, |\Lambda|\}$ and $A$ intersects every non-zero $\gr$-class and every non-zero $\gl$-class of $S$. 
\end{definition}
If $S$ is a finite idempotent generated completely $0$-simple semigroup, then every generating set of $S$ of minimum cardinality is a sparse cover, but in general the converse does not hold. 
We now state Theorem~2.10 from \cite{Gray2008}, which ties these ideas together. 

\begin{theorem} 
\label{thm_Gray}
Let $S=\M[G;I,\Lambda;P]$ be a finite idempotent generated completely $0$-simple semigroup with $|I|=|\Lambda|$. 
Then the following are equivalent\emph{:}
\begin{enumerate}
\item[(i)] $\rank(S)=\idrank(S)$\emph{;}
\item[(ii)] the bipartite graph $\Delta(S)$ satisfies the strong Hall condition\emph{;}
\item[(iii)] every sparse cover of $S$ generates $S$. 
\end{enumerate}
\end{theorem}

\section{Generating sets of regular $*$-semigroups}\label{sect_regular-*}
  
The partition, Brauer and Jones monoids (see Section~\ref{sect_Pn} for the definitions) belong to the variety, introduced by Nordahl and Scheiblich \cite{Nordahl1978}, of \emph{regular $*$-semigroups}.  More generally, every proper two-sided ideal of each of the above-mentioned diagram semigroups is a regular $*$-semigroup that happens to be idempotent generated (see below). So, for all these examples, the problems of determining the rank and idempotent rank, and of describing minimal (idempotent) generating sets, may all be considered. 

In this section we present some general results about small generating sets of regular $*$-semigroups, and then see how they can be applied to answer the above questions for the two-sided ideals of all of the examples mentioned above. 
In particular, in each case, by applying Theorem~\ref{thm_Gray}(iii), we will be able to completely describe the generating sets of minimal cardinality. 

\begin{definition}
A semigroup $S$ is called a \emph{regular $*$-semigroup} if there is a unary operation $*:S \rightarrow S:s\mapsto s^*$ satisfying the following conditions, for all $a,b \in S$:
\[
(a^*)^* = a \COMMA (ab)^* = b^* a^* \COMMA a a^* a = a.
\]
(We note that regular $*$-semigroups are occasionally called \emph{regular involution semigroups}~\cite{NP1985} or  \emph{$*$-regular semigroups} \cite{Polak2001}, even though the latter also refers to a completely different class of (involution) semigroups \cite{Drazin1979,CD2002}.  We will use the original \emph{regular $*$-semigroup} terminology \cite{Nordahl1978}.)
\end{definition}

(Note that $a^*aa^*=a^*$ follows from these axioms.)  
Clearly every regular $*$-semigroup is regular, and $a a^*$ and $a^* a$ are idempotents $\gr$- and $\gl$-related, respectively, to $a$. But not all regular semigroups are regular $*$-semigroups; for instance, regular $*$-semigroups have square $\gd$-classes (the number of $\gr$- and $\gl$-classes contained in any given $\gd$-class must be equal), and not every regular semigroup has this property (consider the full transformation monoid, for example). On the other hand, every inverse semigroup (and hence every group) is easily seen to be a regular $*$-semigroup, with $*$ taken as the inverse operation.  

Many naturally arising semigroups are regular $*$-semigroups, including many semigroups whose elements admit a diagrammatic  representation, where the $*$ operation corresponds to taking the vertical mirror image of the diagram representing the element; this is the case for the Brauer, Jones and partition monoids (see below).

There is a special type of idempotent in regular $*$-semigroups, the so-called \emph{projections}; these play an important role in understanding the structure and generating sets of the semigroup. 

\begin{definition}
An idempotent $p$ in a regular $*$-semigroup $S$ is called a \emph{projection} if $p^* = p$.  If $A\sub S$, we write $E(A)$ (respectively, $P(A)$) for the set of all idempotents (respectively, projections) of $S$ contained in $A$.
\end{definition} 

\subsection{Idempotent generated regular $*$-semigroups} Of course, not every regular $*$-semigroup is idempotent generated; consider groups or inverse semigroups, for instance. However, some important, naturally arising semigroups of diagrams turn out to be idempotent generated; we shall see several examples below. 

Our next aim is to show how generating sets of idempotents may be replaced by small generating sets of idempotents consisting only of projections. 
We begin with some basic observations about the behaviour of projections. These results are well known and we include proofs only for the sake of completeness. 
\begin{lemma}
\label{lem_projections}
Let $S$ be a regular $*$-semigroup. Then\emph{:} 
\begin{enumerate}
\item[(i)] $P(S)=\{aa^*:a\in S\}=\{a^*a:a\in S\}$\emph{;}
\item[(ii)] $E(S)=P(S)^2$.  In particular, the subsemigroup generated by the idempotents coincides with the subsemigroup generated by the projections, and $S$ is idempotent generated if and only if it is generated by its projections\emph{;}
\item[(iii)]
Every $\gr$-class of $S$ contains precisely one projection, as does every $\gl$-class.  
\end{enumerate}
\end{lemma}
\begin{proof}
(i) \ If $p\in P(S)$, then $p=pp^*=p^*p$.  Conversely, it is easy to check that $aa^*,a^*a\in P(S)$ for any $a\in S$.

(ii) \ If $e\in E(S)$, then $e=ee^*e=e(ee)^*e=(ee^*)(e^*e)$.  Conversely, if $p,q\in P(S)$, then $pq=pq(pq)^*pq=pqq^*p^*pq=pqqppq=(pq)^2$.

(iii) \ We prove the result for $\gr$-classes; the result for $\gl$-classes is dual. Let $R$ be an $\gr$-class. For any $a \in R$, $aa^* \in R$ is a projection. So each $\gr$-class contains at least one projection. Now let $p, q \in R$ be projections. Then $p = qp$ and $q = pq$ (as $p\gr q$), so
$
p = p^* = (qp)^* = p^* q^* = pq = q,
$  
as required. 
\end{proof}

\begin{theorem}\label{thm_czsproj}
Let $S$ be a finite idempotent generated completely $0$-simple regular $*$-semigroup with set of non-zero $\gr$-classes and $\gl$-classes indexed by $I$ and $\Lambda$ respectively. 
Then\emph{:}
\begin{enumerate}
\item[(i)] $\rank(S) = \idrank(S) = |I| = |\Lambda|$\emph{;}
\item[(ii)] a subset $X$ of $S$ is a minimal generating set for $S$ if and only if $X$ intersects each non-zero $\gr$- and $\gl$-class of $S$ precisely once. 
\end{enumerate}
\end{theorem}
\begin{proof}
By Lemma~\ref{lem_projections}, since $S$ is idempotent generated it is generated by its set $P$ of non-zero projections and, since this set intersects each $\gr$-class and $\gl$-class exactly once, we conclude that $|P| = |I| = |\Lambda|$. Clearly (thinking of $S$ as a Rees matrix semigroup), every generating set must intersect every non-zero $\gr$- and $\gl$-class, giving:
\[
|I| = |\Lambda| \leq \rank(S) \leq \idrank(S) \leq |P| = |I| = |\Lambda|. 
\]
Part (ii) is then an immediate consequence of Theorem~\ref{thm_Gray}. 
\end{proof}

\begin{definition}[Projection graph]\label{def_GammaS}
Let $S$ be a finite idempotent generated completely $0$-simple regular $*$-semigroup.  Let $\Gamma(S)$ be the digraph with vertex set $P(S)\setminus\{0\}$ and edges $p\to q$ if and only if $pq\not=0$.  We call $\Gamma(S)$ the \emph{projection graph} of $S$. If $T$ is a finite regular $*$-semigroup generated by the idempotents in a maximal $\gj$-class $J$, we will often write $\Gamma(T)=\Gamma(J^*)$.
\end{definition}

For some examples, see Figures \ref{fig_Gamma_5}, \ref{fig_Lambda_5} and \ref{fig_Xi_5}, below.  Note that the edge relation on $\Gamma(S)$ is symmetric and reflexive (symmetry holds because $pq=0$  $\Rightarrow$ $qp=q^*p^*=(pq)^*=0^*=0$, for all $p,q\in P(S)\setminus\{0\}$).  The relationship between the projection graph $\Gamma(S)$ and the Graham--Houghton graph $\Delta(S)$ will be explained in Section~\ref{sec_GHGraphs}. 
By Lemma~\ref{lem_projections}, the edges of $\Gamma(S)$ are in one-one correspondence with the nonzero idempotents of $S$.  

\begin{definition}[Balanced subgraph]\label{def_balanced_subgraph}
We say a subgraph $H$ of a digraph $G$ is \emph{balanced} if $V(H)=V(G)$ and the in- and out-degree of each vertex of $H$ is equal to $1$.
\end{definition}

So a balanced subgraph partitions the digraph into disjoint directed cycles, which may contain one, two, or more vertices. Note that this generalises the notion of disjoint cycle decompositions of permutations. For example, if we begin with the complete directed graph ${KD}_n$ (the digraph with $n$ vertices, a loop at each vertex, and one arc in each direction between every pair of distinct vertices), then there is a natural correspondence between balanced subgraphs of ${KD}_n$ and elements of the symmetric group $\S_n$, given by the obvious translation between balanced subgraphs and the elements of $\S_n$ written as products of disjoint cycles. Thus, balanced subgraphs may be thought of as \emph{permutation subgraphs}: subgraphs that involve all the vertices of the graph, and whose edges induce a permutation of the vertex set of the digraph. The set of all permutations arising from the balanced subgraphs of a directed graph is in general not a group.

\begin{theorem}\label{thm_balanced}
Let $S$ be a finite idempotent generated completely $0$-simple regular $*$-semigroup.  Let $G$ be a subgraph of $\Gamma(S)$, and let $X_G=\{ pq : \text{$p\to q$ is an edge of $G$}\}$ be the corresponding set of idempotents of $S$.  Then $X_G$ is a minimal generating set for $S$ if and only if $G$ is balanced.
\end{theorem}

\begin{proof}
Let $p\in P(S)\setminus\{0\}$.  It is easy to check that $|X_G\cap R_p|$ is equal to the out-degree of $p$ in $G$ (here, $R_p$ denotes the $\gr$-class containing $p$).  Together with the dual statement, and Theorem~\ref{thm_czsproj}~(ii), this proves the result.
\end{proof}

\begin{theorem}
\label{thm_maxJclasses}
Let $S$ be a finite regular semigroup, let $J$ be a $\gj$-class of $S$, $X$ a subset of $J$, and $T = \lb X \rb$, the subsemigroup of $S$ generated by $X$. If $T$ is regular, then
\[
J \cap T = J_1 \cup \cdots \cup J_l,
\]
where $J_1, \ldots, J_l$ are the maximal $\gj$-classes of $T$. 
\end{theorem}

\nc{\gjS}{\gj^S}
\nc{\gjT}{\gj^T}

\begin{proof}
Throughout the proof, we write $\gjS$ and $\gjT$ for Green's $\gj$-relations on $S$ and $T$, and similarly for the other Green's relations.  We also write $J^T(x)$ for the $\gjT$-class in $T$ of $x\in T$.
Now let $K$ be a maximal $\gjT$-class of $T$. Since $\lb X \rb = T$ it follows that $K \leq_{\gjT} J^T(x)$ for some $x \in X$. Since $K$ is maximal, we conclude that $K = J^T(x) \subseteq J$. Thus, every maximal $\gjT$-class of $T$ is a subset of $J$. 

Conversely, suppose $K$ is a $\gjT$-class of $T$ with $K\sub J$.  The proof will be complete if we can show that $K$ is maximal (in the ordering on $\gjT$-classes).  Now, $K\leq_{\gjT} M$ for some maximal $\gjT$-class $M$ of $T$.  Let $a\in K$ and $b\in M$ be arbitrary.  Since $K\leq_{\gjT} M$ (and since $T$ is regular), there exist $u,v\in T$ such that $a=ubv$.  Since $K,M\sub J$, we have $a\gj^Sb$, so 
$b \leq_{\gjS} ub \leq_{\gjS} ubv =a \leq_{\gjS} b$.
Thus, all these elements are $\gjS$-related.  In particular, $b\gjS ub$ and $ub\gjS ubv$.  Since $S$ is finite, stability (see \cite[Definition A.2.1 and Theorem A.2.4]{SteinbergBook2009}) gives $b\gl^S ub$ and $ub\gr^S ubv=a$.  Since $T$ is regular, it follows that $b\gl^T ub\gr^T a$, whence $b\gd^T a$ and $b\gj^T a$.  Thus, $K=J^T(a)=J^T(b)=M$.
\end{proof}

The above result fails if one lifts the assumption that the subsemigroup $T$ is regular. One can easily construct a counterexample example where $T$ is a $4$-element subsemigroup of the $5$-element Brandt semigroup (see \cite[Section~5.1]{howie} for more on Brandt semigroups).  

\begin{lemma}\label{lem_subranks}
Let $S$ be a finite semigroup generated by the elements in its maximal $\gj$-classes $J_1, \ldots, J_m$. Then
\[
\rank(S) = \sum_{i=1}^m \rank(J_i^*).
\]
If, in addition, $S$ is idempotent generated, then so are all of $J_1^*, \ldots, J_m^*$ and
\[
\idrank(S) = \sum_{i=1}^m \idrank(J_i^*).
\]
\end{lemma}
\begin{proof}
This is an easy generalization of Lemma~\ref{lem_maxJ} above. 
\end{proof}

The following theorem, which is the main result of this section, says that in general, generating sets consisting of projections all taken from the same $\gj$-class of a finite regular $*$-semigroup always constitute a minimal generating set for the semigroup they generate. As we shall see, all of the examples of subsemigroups of $\P_n$ discussed above, and their ideals, may be obtained as subsemigroups of the partition monoid generated by projections of fixed rank, and this allows us to apply the following result to determine the ranks and idempotent ranks in all cases, and ultimately to describe all the minimal generating sets. 

\begin{theorem}\label{thm_projgenerated}
Let $S$ be a finite regular $*$-semigroup, let $J$ be a $\gj$-class of $S$ and let $X \subseteq J$ be a set of projections. Then the subsemigroup $T=\lb X\rb$ of $S$ generated by $X$ satisfies
\[
\rank(T) = \idrank(T) = |X|. 
\]
Furthermore, a subset $Y$ of $T$ is a minimal generating set for $T$ if and only if $Y$ forms a transversal of the set of $\gr$-classes, and set of $\gl$-classes, contained in the maximal $\gj$-classes of $T$.  
\end{theorem}
\begin{proof}
First observe that $T$ is an idempotent generated regular $*$-semigroup. 
Indeed, to see that $T$ is closed under the $*$ operation let $t \in T$ be arbitrary. Then 
$
t = p_1 \cdots p_k 
$
for some $p_1,\ldots,p_k\in X$, and so  
$
t^* = p_k^* \cdots p_1^* = p_k \cdots p_1 \in T.
$

Next, we claim that $T$ is generated by the elements in its maximal $\gj$-classes. 
Let $J_1, \ldots, J_l$ be the maximal $\gj$-classes of $T$. It follows from Theorem~\ref{thm_maxJclasses} that
$
J_1 \cup \cdots \cup J_l = 
J \cap T \supseteq X.
$
Therefore, $T$ is an idempotent generated regular $*$-semigroup generated by the elements in its maximal $\gj$-classes $J_1 \cup \cdots \cup J_l$. By Lemma~\ref{lem_subranks},
\[
\rank(T) = \sum_{i=1}^l \rank(J_i^*) \quad\text{and}\quad
\idrank(T) = \sum_{i=1}^l \idrank(J_i^*).
\]
For each $i$, Theorem~\ref{thm_czsproj} gives
\[
\rank(J_i^*) = \idrank(J_i^*) = |X \cap J_i|, 
\]
since $X$ intersects each $\gr$- and $\gl$-class of $J_i$ exactly once.  It follows that $\rank(T)=\idrank(T)=|X|$.

The last clause follows from Theorem~\ref{thm_czsproj}.
\end{proof}

Note that in general, in the above theorem, the semigroup $T$ will have more than one maximal $\gj$-class.

In each of the examples we will study, the $\gj$-classes will form a chain, so the next general result will be of use.

\begin{proposition}\label{prop_chain}
Suppose $S$ is a finite idempotent generated regular $*$-semigroup and that the $\gj$-classes of $S$ form a chain $J_0<J_1<\cdots<J_k$.  For $0\leq r\leq k$, let $I_r=J_0\cup J_1\cup\cdots\cup J_r$.  Then the ideals of $S$ are precisely the sets $I_0,I_1,\ldots,I_k$.

Suppose further that $P(J_s)\sub \langle J_{s+1}\rangle$ for all $0\leq s\leq k-1$. Then for each $0\leq r\leq k$\emph{:}
\begin{itemize}
	\item[(i)] $I_r=\langle J_r\rangle$\emph{;}
	\item[(ii)] $I_r$ is idempotent generated\emph{;}
	\item[(iii)] $\rank(I_r)=\idrank(I_r)=\rho_r$, where $\rho_r$ is the number of $\gr$-classes (which equals the number of $\gl$-classes) in $J_r$.
\end{itemize}
\end{proposition}

\begin{proof}
The statement concerning the ideals of $S$ is easily checked (and is true for any semigroup $S$ whose $\gj$-classes form a finite chain).

	(i) \ Let $0\leq s\leq k-1$.  We first show that
	\begin{equation}
	\label{dagger}
	J_s\sub\la J_{s+1}\ra.
	\end{equation}
First note that if $e\in E(J_s)$, then, as in the proof of Lemma \ref{lem_projections}, $e=(ee^*)(e^*e)$ where $ee^*,e^*e\in P(J_s)$.  So $E(J_s)\sub\la P(J_s)\ra\sub\langle J_{s+1}\rangle$.  By Corollary \ref{cor_FitzGerald}, every element of $J_s$ is a product of idempotents from $J_s$, and \eqref{dagger} follows.  Next note that $I_0=J_0=\langle J_0\rangle$.  If $1\leq r\leq k$, then by \eqref{dagger} and an induction hypothesis, $I_r=I_{r-1}\cup J_r=\la J_{r-1}\ra\cup J_r\sub \la J_r\ra$.

	(ii) \ Since $S$ is idempotent generated, it follows from Corollary \ref{cor_FitzGerald} that $J_r^*$ is an idempotent generated completely $0$-simple semigroup.  Lemma~\ref{lem_maxJ} then implies that $I_r=\langle J_r\rangle$ is idempotent generated.

	(iii) \ Since $I_r$ is a regular $*$-semigroup, Theorem~\ref{thm_czsproj} gives $\rank(J_r^*)=\idrank(J_r^*)=\rho_r$.  By Lemma \ref{lem_maxJ}, $\rank(I_r)=\rank(J_r^*)$ and $\idrank(I_r)=\idrank(J_r^*)$. 
\end{proof}

\begin{remark}
In the notation of the previous proposition, it is clear that the number of minimal generating sets of the ideal $I_r$ is equal to $\rho_r!\times h_r^{\rho_r}$, where $h_r$ is the (common) size of the $\gh$-classes of $S$ contained in $J_r$.  In general, a formula for the number of minimal \emph{idempotent} generating sets is harder to come by.  
For example, to the authors' knowledge, such a general formula is unknown even in the case of the proper ideals of the singular part $S=\text{Sing}_n$ of the full transformation monoid $\T_n$.
\end{remark}

\section{Non-minimal generating sets of idempotents}
\label{sec_GHGraphs}

Theorem~\ref{thm_balanced} gives a necessary and sufficient condition, in terms of the projection graph $\Gamma(S)$, for a set of idempotents to be a minimal idempotent generating set for a finite idempotent generated completely $0$-simple regular $*$-semigroup $S$. Here we consider idempotent generating sets in general, not just those of minimal size. 
We begin by giving some background on Graham--Houghton graphs taken from  \cite[Section~4.13]{SteinbergBook2009}.  

\begin{definition}
Let $S=\M[G;I,\Lambda;P]$ be a completely $0$-simple semigroup represented as a Rees matrix semigroup over a group $G$. Recall from Definition \ref{GHGraph_defn} that the \emph{Graham--Houghton graph}, denoted $\Delta(S)$, is the directed bipartite graph with vertex set $I \cup \Lambda$ and an arc $\lambda\to i$ if and only if $p_{\lambda i} \neq 0$.  Now let $F$ be a set of non-zero idempotents of $S$. We denote by $\Delta(S,F)$ the directed bipartite graph obtained from $\Delta(S)$ by adding arcs $i \rightarrow \lambda$ for every pair $(i,\lambda)$ such that $H_{i,\lambda} \cap F \neq \varnothing$. (Here, $H_{i,\lam}=R_i\cap L_\lam$ denotes the $\gh$-class that is the intersection of the $\gr$- and $\gl$-classes indexed by $i\in I$ and $\lam\in\Lam$, respectively, and we note that such an $\gh$-class contains an idempotent if and only if $p_{\lam i}\not=0$, in which case this idempotent is $(i,p_{\lam i}^{-1},\lam)$.)
\end{definition}

The various parts of the next result follow from \cite[Theorem 2 and subsequent discussion]{Graham68}; we include a proof for convenience.

\begin{proposition}\label{prop_Fgens}
Let $S=\M[G;I,\Lambda;P]$ be a finite idempotent generated completely $0$-simple semigroup, and let $F$ be a set of non-zero idempotents of $S$.
Then the following are equivalent\emph{:}
\begin{enumerate}
\item[(i)] $S = \langle F \rangle$\emph{;}
\item[(ii)] For all $(i,\lambda) \in I \times \Lambda$ there is a directed path in $\Delta(S,F)$ from $i$ to $\lambda$\emph{;}
\item[(iii)] The digraph $\Delta(S,F)$ is strongly connected. 
\end{enumerate}
\end{proposition}

\begin{proof}
%(i) $\Rightarrow$ (ii):  
%Suppose first that $S=\la F\ra$.  
First assume (i).
Let $i\in I$ and $\lam\in\Lam$ be arbitrary, and consider a product $(i,1,\lam)=(i_1,p_{\lam_1, i_1}^{-1},\lam_1)\cdots(i_k,p_{\lam_k, i_k}^{-1},\lam_k)$, where each $(i_m,p_{\lam_m, i_m}^{-1},\lam_m)\in F$ (and where~$1$ denotes the identity of $G$).  First, we must have $i_1=i$ and $\lam_k=\lam$.  For the term $(i_m,p_{\lam_m, i_m}^{-1},\lam_m)$ to belong to $F$, the graph $\De(S,F)$ must contain the edge $i_m\to\lam_m$.  And for the product to be non-zero, each $p_{\lam_m, i_{m+1}}$ must be non-zero, so $\De(S,F)$ contains the edge $\lam_m\to i_{m+1}$.  So $\De(S,F)$ contains the path $i=i_1\to\lam_1\to i_2\to\cdots\to\lam_k=\lam$, giving~(ii).

Next assume (ii).  To prove (iii), it remains to show that for any $i,j\in I$ and $\lam,\mu\in\Lam$, $\De(S,F)$ contains a path from $i$ to $j$, and from $\lam$ to $\mu$.  We just do the first of these, as the other is similar.  By definition, there exists some $\kappa\in\Lam$ such that $p_{\kappa, j}\not=0$, so $\De(S,F)$ has the edge $\kappa\to j$.  But, by assumption, $\De(S,F)$ also contains a path from $i$ to $\kappa$.  Appending the edge $\kappa\to j$ to this path yields a path from $i$ to $j$.

Finally, assume (iii).  Let $e\in E(S)\setminus\{0\}$.  So $e=(i,p_{\lam, i}^{-1},\lam)$ for some $i\in I$ and $\lam\in\Lam$.  By assumption, there is a path $i=i_1\to\lam_1\to i_2\to\cdots\to\lam_k=\lam$ in $\De(S,F)$.  It follows that the product $(i_1,p_{\lam_1, i_1}^{-1},\lam_1)\cdots(i_k,p_{\lam_k, i_k}^{-1},\lam_k)$ is non-zero and is $\gh$-related to $e=(i,p_{\lam, i}^{-1},\lam)$; note that each term in this product belongs to $F$.  Since $S$ is finite, and since $H_{i,\lam}$ is a group (as it contains an idempotent), some positive power of this product is equal to $e$.  In particular, $e\in\la F\ra$.  Since this is true of an arbitrary non-zero idempotent $e$, and since $S$ is idempotent generated, it follows that $S=\la F\ra$, giving (i).
\end{proof}

In the special case that $S$ is also a regular $*$-semigroup, the digraph $\Delta(S)$ has an especially nice form, and this allows us to re-express Proposition~\ref{prop_Fgens} in terms of the digraph $\Gamma(S)$, defined in Definition~\ref{def_GammaS}, as follows. 

Let $S=\M[G;I,\Lambda;P]$ be an idempotent generated completely $0$-simple regular $*$-semigroup. Then we may set $\Lambda = I'=\{i':i\in I\}$, with the $\gr$- and $\gl$-classes indexed in such a way that the projections lie on the main diagonal $\gh$-classes, $H_{i,i'}$ $(i \in I)$. From this it then follows that the idempotents are distributed in the non-zero $\gd$-class of $S$ with diagonal symmetry, that is $H_{i,j'}$ contains an idempotent if and only if $H_{j,i'}$ does. 

Using these observations it is easy to see that the graph $\Gamma(S)$ is isomorphic to the quotient digraph of $\Delta(S)$ obtained by identifying the pairs of vertices $\{i,i'\}$ for all $i \in I$. More explicitly, $\Gamma(S)$ is isomorphic to the digraph with vertex set $\big\{ \{i,i'\} : i \in I \big\}$ where there is an arc $\{ i, i' \} \rightarrow \{ j, j' \}$ if and only if $H_{i,j'}$ is a group (equivalently, $H_{j, i'}$ is a group). 

The digraph $\Gamma(S)$ is clearly symmetric, and has loops at every vertex. No arcs are identified when passing from $\Delta(S)$ to $\Gamma(S)$, and the arcs of $\Delta(S)$ (and thus also in $\Gamma(S)$) are in natural bijective correspondence with the non-zero idempotents of $S$.  

Now we would like to re-express Proposition~\ref{prop_Fgens} in terms of the digraph $\Gamma(S)$. 

\begin{definition}
Let $F$ be a set of non-zero idempotents from a finite idempotent generated completely $0$-simple regular $*$-semigroup $S=\M[G;I,\Lambda;P]$. We use $\Gamma(S,F)$ to denote the $2$-coloured digraph obtained by first colouring all the edges blue in the projection graph $\Gamma(S)$, and then adding the additional red arcs $i {\boldsymbol{\red\to}} j$ for each idempotent $f \in F \cap H_{i,j}$.  
\end{definition}

\begin{definition}
Let $G$ be a digraph with directed edges coloured red or blue.  Call a directed path $p$ in $G$ an \emph{RBR-alternating path} if the first and last arcs are red, and, as we traverse the path, the arcs alternate in colour between red and blue. An RBR-alternating circuit is an RBR-alternating path whose initial and terminal vertices coincide. In particular, a red loop at a vertex is an example of an RBR-alternating circuit. 
\end{definition}

Since $S$ is generated by the projections, $F$ will generate $S$ if and only if every projection may be expressed as a product of elements from $F$. Combining this with Proposition~\ref{prop_Fgens} we obtain the following result. 

\begin{theorem}\label{thm_generalRBR}
Let $S=\M[G;I,\Lambda;P]$ be a finite idempotent generated completely $0$-simple regular $*$-semigroup, and let $F$ be a set of non-zero idempotents from $S$. Then $S = \langle F \rangle$ if and only if every vertex in $\Gamma(S,F)$ is the base point of some RBR-alternating circuit.  
\end{theorem}

\begin{remark}
To directly see the significance of RBR-alternating paths and circuits, consider a finite idempotent generated completely $0$-simple regular $*$-semigroup $S$ with non-zero projections $p_1,\ldots,p_k$.  Let $F$ be a set of non-zero idempotents from $S$, and let $$p_{i_1}{\boldsymbol{\red\to}}p_{i_2}{\blue\to}p_{i_3}{\boldsymbol{\red\to}}\cdots{\boldsymbol{\red\to}}p_{i_r}$$ be an RBR-alternating path in the graph $\Gamma(S,F)$.  The red edges mean that the (non-zero) idempotents $p_{i_1}p_{i_2},p_{i_3}p_{i_4},\ldots,p_{i_{r-1}}p_{i_r}$ all belong to $F$, so the product $s=p_{i_1}p_{i_2}\cdots p_{i_r}$ belongs to $\la F\ra$.  The blue edges mean that the products $p_{i_2}p_{i_3},\ldots p_{i_{r-2}}p_{i_{r-1}}$ are all non-zero, and it follows (thinking of $S$ as a Rees matrix semigroup) that $s$ is non-zero too.  In the case that $i_r=i_1$ (so the path is a circuit), $s\gh p_{i_1}$ and so, since $S$ is finite, some power $s^t\in\la F\ra$ is equal to $p_{i_1}$.  See Example \ref{Example6}.
\end{remark}

\begin{remark}\label{rem_generalRBR0}
The condition that every vertex is the base of an RBR-circuit is not equivalent to saying that each vertex is simply \emph{contained} in some RBR-alternating circuit.  For example, consider the graph in Figure~\ref{Fig_TriangleRBGraph}. In this example $1{\boldsymbol{\red\to}}2{\blue\to}3{\boldsymbol{\red\to}}1$ is an RBR-circuit containing each vertex, but vertices $2$ and $3$ are clearly not the base of any RBR-circuit.
\end{remark}

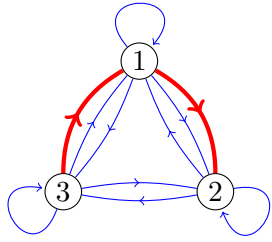
\begin{figure}
\begin{center}
%\begin{tikzpicture}[scale=1]
%\tikzstyle{vertex}=[circle,draw=black, fill=white, inner sep = 0.06cm]
%%
%\node[vertex] (A) at (0,0) {$3$};
%\node[vertex] (B) at (2,0) {$2$};
%\node[vertex] (C) at (1,1.732) {$1$};
%%
%\draw [ultra thick, blue,->-=0.5] (A)        to [bend left=10] (B);
%\draw [ultra thick, blue,->-=0.5] (B)        to [bend left=10] (C);
%\draw [ultra thick, blue,->-=0.5] (C)        to [bend left=10] (A);
%%
%\draw [ultra thick, blue,->-=0.5] (B)        to [bend left=10] (A);
%\draw [ultra thick, blue,->-=0.5] (A)        to [bend left=10] (C);
%\draw [ultra thick, blue,->-=0.5] (C)        to [bend left=10] (B);
%%
%\draw [red,->-=0.5] (A)        to [bend left=30] (C);
%\draw [red,->-=0.5] (C)        to [bend left=30] (B);
%%
%\draw [ultra thick, blue,->] (C) edge [out=130,in=50,loop] ();
%\draw [ultra thick, blue,->] (B) edge [out=130+240,in=50+240,loop] ();
%\draw [ultra thick, blue,->] (A) edge [out=130+120,in=50+120,loop] ();
%\end{tikzpicture}
%\qquad
\begin{tikzpicture}[scale=1]
\tikzstyle{vertex}=[circle,draw=black, fill=white, inner sep = 0.06cm]
\node[vertex] (A) at (0,0) {$3$};
\node[vertex] (B) at (2,0) {$2$};
\node[vertex] (C) at (1,1.732) {$1$};
\draw [blue,->-=0.5] (A)        to [bend left=10] (B);
\draw [blue,->-=0.5] (B)        to [bend left=10] (C);
\draw [blue,->-=0.5] (C)        to [bend left=10] (A);
\draw [blue,->-=0.5] (B)        to [bend left=10] (A);
\draw [blue,->-=0.5] (A)        to [bend left=10] (C);
\draw [blue,->-=0.5] (C)        to [bend left=10] (B);
\draw [ultra thick, red,->-=0.5] (A)        to [bend left=30] (C);
\draw [ultra thick, red,->-=0.5] (C)        to [bend left=30] (B);
\draw [blue,->] (C) edge [out=130,in=50,loop] ();
\draw [blue,->] (B) edge [out=130+240,in=50+240,loop] ();
\draw [blue,->] (A) edge [out=130+120,in=50+120,loop] ();
\end{tikzpicture}
\end{center}
\vspace{-6mm}
\caption{
A graph showing that the condition that every vertex is the base of an RBR-circuit is not equivalent to saying that each vertex is contained in some RBR-alternating circuit.  
In all diagrams of graphs with red and blue edges, the red edges are drawn to be thicker than blue edges (for the convenience of readers with black and white copies of the article).
}
\label{Fig_TriangleRBGraph}
\end{figure}

\begin{remark}\label{rem_generalRBR}
We have obtained a necessary and sufficient condition for $S=\M[G;I,\Lambda;P]$ to be generated by a given set $F$ of non-zero idempotents from $S$ in terms of the graph $\Gamma(S,F)$.  It would be nice to give a necessary and sufficient condition just in terms of the subgraph induced by the edges corresponding to the elements from $F$; that is, the red edges.  One might be tempted to conjecture that $S=\lb F\rb$ if and only if, in $\Gamma(S,F)$, either:
\begin{itemize}
	\item[(i)] each vertex has at least one red edge coming in to it and at least one going out of it, or
	\item[(ii)] each vertex is contained in a red circuit.
\end{itemize}
By Theorem~\ref{thm_generalRBR}, we see immediately that condition (i) is necessary, while condition (ii) is sufficient.  However, we will see that (i) need not be sufficient, and (ii) need not be necessary.  See Examples~\ref{Example3} and~\ref{Example6}, Remarks \ref{rem810} and \ref{rem911}, and also Theorem~\ref{arb_gen_set_Jn}.
\end{remark}

\section{The partition monoid}\label{sect_Pn}

In the following sections, we will apply the general results of the previous sections to calculate the ranks and idempotent ranks for the proper two-sided ideals of the partition monoid, and several of its submonoids.  In each case, we will also describe the (minimal) idempotent generating sets for the largest proper two-sided ideal, and also enumerate the minimal such generating sets for the partition and Jones monoids; a solution is not currently available for the latter problem with respect to the Brauer monoid (see Remark~\ref{rem:dn}).  

Let $n$ be a positive integer, which we fix throughout this section, and write
$\bn$ for the finite set~$\{1,\ldots,n\}$. If $1\leq r\leq s\leq n$, we write $[r,s]=\{r,r+1,\ldots,s\}$.  We also write $\bn'=\{1',\ldots,n'\}$
for a set in one-one correspondence with $\bn$. An \emph{$n$-partition} (or simply a \emph{partition} if $n$ is understood from context) is an equivalence relation on
$\bn\cup\bn'$. 
We regard the elements of $\P_n$ formally as subsets of $(\bn\cup\bn')\times(\bn\cup\bn')$ satisfying reflexivity, symmetry and transitivity; so we will write $(x,y)\in\al$ to mean that $x$ and $y$ are $\al$-related elements of $\bn\cup\bn'$.
The set~$\P_n$ of all $n$-partitions forms a monoid of size $|\P_n|=B_{2n}$ (the $2n$th Bell number), known as the
\emph{partition monoid}, under an associative operation we will describe shortly.  

A partition $\al\in\P_n$ may be represented by a graph on the vertex set
$\bn\cup\bn'$ as follows. We arrange vertices $1,\ldots,n$ in a row (increasing
from left to right) and vertices $1',\ldots,n'$ in a parallel row directly below.
We then add edges in such a way that two vertices $x,y$ are connected by a path if and
only if $(x,y)\in\al$. For example, the partition from $\P_6$ with equivalence classes
$
\{1\} , \{2,3',4'\} , \{3,4\} , \{5,6,1',5',6'\} , \{2'\}
$
is represented by the graph 
${\psset{xunit=.3cm, yunit=.2cm}
 \begin{pspicture}(0,-.2)(5.5,0.3) \psset{origin={0.2,-0.6}}
 \suv0
 \suv1
 \suv2
 \suv3
 \suv4
 \suv5
 \slv0
 \slv1
 \slv2
 \slv3
 \slv4
 \slv5
 \psline(1,2)(2,0)(3,0)
 \psline(2,2)(3,2)
 \psline(4,2)(5,2)(5,0)(4,0)
 \pscurve(4,0)(3,0.6)(1,0.6)(0,0)
\end{pspicture}}$.
Of course, such a graphical representation is not unique, but we will identify two
graphs on the vertex set $\bn\cup\bn'$ if they have the same connected components.
In the same way, we will not distinguish between a partition and a graph that
represents it. 

In order to describe the product alluded to above, let
$\al,\be\in\P_n$. We first stack (graphs representing)~$\al$ and~$\be$ so that
vertices $1',\ldots,n'$ of $\al$ are identified with vertices $1,\ldots,n$ of~$\be$. The connected components of this graph are then constructed, and we finally
delete the middle row of vertices as well as any connected components that are contained within the middle row; the resulting graph is the product $\al\be$. As an example, consider the two partitions
$\al=
{\psset{xunit=.3cm, yunit=.2cm}
 \begin{pspicture}(0,-.2)(5.5,0.3) \psset{origin={0.2,-0.6}}
 \suv0
 \suv1
 \suv2
 \suv3
 \suv4
 \suv5
 \slv0
 \slv1
 \slv2
 \slv3
 \slv4
 \slv5
 \stline23
 \uline12
 \uline45
 \lline34
 \pscurve(0,2)(1,1.4)(2,1.4)(3,2)
 \pscurve(0,0)(0.66,0.6)(1.33,0.6)(2,0)
 \pscurve(2,0)(3,0.6)(4,0.6)(5,0)
\end{pspicture}}
$
and
$\be=
{\psset{xunit=.3cm, yunit=.2cm}
 \begin{pspicture}(0,-.2)(5.5,0.3) \psset{origin={0.2,-0.6}}
 \suv0
 \suv1
 \suv2
 \suv3
 \suv4
 \suv5
 \slv0
 \slv1
 \slv2
 \slv3
 \slv4
 \slv5
 \stline32
 \stline44
 \lline35
 \pscurve(0,2)(0.66,1.4)(1.33,1.4)(2,2)
 \pscurve(1,2)(1.66,1.4)(2.33,1.4)(3,2)
\end{pspicture}}
$
from $\P_6$.  The product $\al\be={\psset{xunit=.3cm, yunit=.2cm}
 \begin{pspicture}(0,-.2)(5.5,0.3) \psset{origin={0.2,-0.6}}
 \suv0
 \suv1
 \suv2
 \suv3
 \suv4
 \suv5
 \slv0
 \slv1
 \slv2
 \slv3
 \slv4
 \slv5
 \stline22
 \uline12
 \uline45
 \lline25
 \pscurve(0,2)(1,1.4)(2,1.4)(3,2)
\end{pspicture}}$ is found by moving through the following intermediate stages:
\[
{\psset{xunit=.3cm, yunit=.2cm}
 \begin{pspicture}(0,-.2)(5.5,3.8) \psset{origin={0.2,2.4}}
 \suv0
 \suv1
 \suv2
 \suv3
 \suv4
 \suv5
 \slv0
 \slv1
 \slv2
 \slv3
 \slv4
 \slv5
 \stline23
 \uline12
 \uline45
 \lline34
 \pscurve(0,2)(1,1.4)(2,1.4)(3,2)
 \pscurve(0,0)(0.66,0.6)(1.33,0.6)(2,0)
 \pscurve(2,0)(3,0.6)(4,0.6)(5,0)
 \rput[r](-0.4,3.4){{\footnotesize $\al=$}}
\psset{origin={0.2,-1.6}}
 \suv0
 \suv1
 \suv2
 \suv3
 \suv4
 \suv5
 \slv0
 \slv1
 \slv2
 \slv3
 \slv4
 \slv5
 \stline32
 \stline44
 \lline35
 \pscurve(0,2)(0.66,1.4)(1.33,1.4)(2,2)
 \pscurve(1,2)(1.66,1.4)(2.33,1.4)(3,2)
 \rput[r](-0.4,-.6){{\footnotesize $\be=$}}
\end{pspicture}}
\quad
{\psset{xunit=.3cm, yunit=.2cm}
 \begin{pspicture}(0,-.2)(2.5,3.8) \psset{origin={0.2,1.4}}
 \psline{->}(0,1.9)(2,1.1)
 \psline{->}(0,-1.9)(2,-1.1)
\end{pspicture}}
\quad
{\psset{xunit=.3cm, yunit=.2cm}
 \begin{pspicture}(0,-.2)(5.5,3.8) \psset{origin={0.2,1.4}}
 \suv0
 \suv1
 \suv2
 \suv3
 \suv4
 \suv5
 \slv0
 \slv1
 \slv2
 \slv3
 \slv4
 \slv5
 \stline23
 \uline12
 \uline45
 \lline34
 \pscurve(0,2)(1,1.4)(2,1.4)(3,2)
 \pscurve(0,0)(0.66,0.6)(1.33,0.6)(2,0)
 \pscurve(2,0)(3,0.6)(4,0.6)(5,0)
\psset{origin={0.2,-0.6}}
 \suv0
 \suv1
 \suv2
 \suv3
 \suv4
 \suv5
 \slv0
 \slv1
 \slv2
 \slv3
 \slv4
 \slv5
 \stline32
 \stline44
 \lline35
 \pscurve(0,2)(0.66,1.4)(1.33,1.4)(2,2)
 \pscurve(1,2)(1.66,1.4)(2.33,1.4)(3,2)
\end{pspicture}}
\quad
{\psset{xunit=.3cm, yunit=.2cm}
 \begin{pspicture}(0,-.2)(2.5,3.8) \psset{origin={0.2,1.4}}
 \psline{->}(0,0)(2,0)
\end{pspicture}}
\quad
{\psset{xunit=.3cm, yunit=.2cm}
 \begin{pspicture}(0,-.2)(5.5,3.3) \psset{origin={0.2,1.4}}
 \suv0
 \suv1
 \suv2
 \suv3
 \suv4
 \suv5
 \stline23
 \uline12
 \uline45
 \lline34
 \pscurve(0,2)(1,1.4)(2,1.4)(3,2)
\psset{origin={0.2,-0.6}}
 \slv0
 \slv1
 \slv2
 \slv3
 \slv4
 \slv5
 \stline32
 \stline44
 \lline35
 \pscurve(1,2)(1.66,1.4)(2.33,1.4)(3,2)
\end{pspicture}}
\quad
{\psset{xunit=.3cm, yunit=.2cm}
 \begin{pspicture}(0,-.2)(2.5,3.3) \psset{origin={0.2,1.4}}
 \psline{->}(0,0)(2,0)
\end{pspicture}}
\quad
{\psset{xunit=.3cm, yunit=.2cm}
 \begin{pspicture}(0,-.2)(5.5,0.3) \psset{origin={0.2,0.4}}
 \suv0
 \suv1
 \suv2
 \suv3
 \suv4
 \suv5
 \slv0
 \slv1
 \slv2
 \slv3
 \slv4
 \slv5
 \stline22
 \uline12
 \uline45
 \lline25
 \pscurve(0,2)(1,1.4)(2,1.4)(3,2)
 \rput[l](5.8,1.4){{\footnotesize $=\al\be.$}}
\end{pspicture}}
\]
We now introduce some notation and terminology that we will use throughout our study.  Let $\al\in\P_n$. The equivalence classes of $\al$ are called its \emph{blocks}.  A block of $\al$ is called a \emph{transversal block} if it has nonempty intersection with both $\bn$ and $\bn'$, and a \emph{nontransversal block} otherwise.  The \emph{rank} of $\al$, denoted $\rank(\al)$, is equal to the number of transversal blocks of $\al$.  For $x\in\bn\cup\bn'$, let $[x]_\al$ denote the block of $\al$ containing $x$.  We define the \emph{domain} and \emph{codomain} of $\al$ to be the sets
\begin{align*}
\dom(\al) &= \bigset{ x\in\bn } { [x]_\al\cap\bn'\not=\varnothing}, \\
\codom(\al) &= \bigset{ x\in\bn } { [x']_\al\cap\bn\not=\varnothing}.
\intertext{We also define the \emph{kernel} and \emph{cokernel} of $\al$ to be the equivalences}
\ker(\al) &= \bigset{(x,y)\in\bn\times\bn}{[x]_\al=[y]_\al}, \\
\coker(\al) &= \bigset{(x,y)\in\bn\times\bn}{[x']_\al=[y']_\al}.
\end{align*}
To illustrate these ideas, let $\al={\psset{xunit=.3cm, yunit=.2cm}
 \begin{pspicture}(0,-.2)(5.5,0.3) \psset{origin={0.2,-0.6}}
 \suv0
 \suv1
 \suv2
 \suv3
 \suv4
 \suv5
 \slv0
 \slv1
 \slv2
 \slv3
 \slv4
 \slv5
 \psline(1,2)(2,0)(3,0)
 \psline(2,2)(3,2)
 \psline(4,2)(5,2)(5,0)(4,0)
 \pscurve(4,0)(3,0.6)(1,0.6)(0,0)
\end{pspicture}}\in\P_6$. Then $\rank(\al)=2$, $\dom(\al)=\{2,5,6\}$,
$\codom(\al)=\{1,3,4,5,6\}$, and $\al$ has
kernel-classes~$\{1\}$, $\{2\}$, $\{3,4\}$, $\{5,6\}$ and cokernel-classes
$\{1,5,6\}$, $\{2\}$, $\{3,4\}$.

It is immediate from the definitions that
\[
\begin{array}{rclcrcl}
\dom(\al\be) \hspace{-.25cm}&\sub&\hspace{-.25cm} \dom(\al), & &
\ker(\al\be)\hspace{-.25cm} &\sp&\hspace{-.25cm} \ker(\al),\\
\codom(\al\be) \hspace{-.25cm}&\sub&\hspace{-.25cm} \codom(\be), & &
\coker(\al\be)\hspace{-.25cm} &\sp&\hspace{-.25cm} \coker(\be)
\end{array}
\]
for all $\al,\be\in\P_n$.  Let $\De$ denote the trivial equivalence relation (that is, the equality relation) on $\bn$.  It is also clear that
\[
\begin{array}{rclcrcl}
\dom(\be) \speq \bn & \spra & \dom(\al\be) \speq \dom(\al), \\
\codom(\al) \speq \bn & \spra & \codom(\al\be) \speq \codom(\be), \\
\ker(\be) \speq \De & \spra & \ker(\al\be) \speq \ker(\al), \\
\coker(\al) \speq \De & \spra & \coker(\al\be) \speq \coker(\be)
\end{array}
\]
for all $\al,\be\in\P_n$.  In particular, the sets
\[
\begin{array}{cc}
\set{\al\in\P_n}{\dom(\al)=[n]}, \qquad&
\set{\al\in\P_n}{\ker(\al)=\De}, \\
\set{\al\in\P_n}{\codom(\al)=[n]}, \qquad&
\set{\al\in\P_n}{\coker(\al)=\De}
\end{array}
\]
are all submonoids of $\P_n$.  The intersection of these four submonoids is (isomorphic to) the symmetric group $\S_n$, which is easily seen to be the group of units of $\P_n$.

If $x\in\bn$, we write $x''=x$.  For $\al\in\P_n$, we define $\al^*={\bigset{(x',y')}{(x,y)\in\al}}$.  Diagrammatically, $\al^*$ is obtained by reflecting (a graph representing)~$\al$ in a horizontal axis.  For example, if $\al={\psset{xunit=.3cm, yunit=.2cm}
 \begin{pspicture}(0,-.2)(5.5,0.3) \psset{origin={0.2,-0.6}}
 \suv0
 \suv1
 \suv2
 \suv3
 \suv4
 \suv5
 \slv0
 \slv1
 \slv2
 \slv3
 \slv4
 \slv5
 \psline(1,2)(2,0)(3,0)
 \psline(2,2)(3,2)
 \psline(4,2)(5,2)(5,0)(4,0)
 \pscurve(4,0)(3,0.6)(1,0.6)(0,0)
\end{pspicture}}\in\P_6$, then $\al^*={\psset{xunit=.3cm, yunit=.2cm}
 \begin{pspicture}(0,-.2)(5.5,0.3) \psset{origin={0.2,-0.6}}
 \suv0
 \suv1
 \suv2
 \suv3
 \suv4
 \suv5
 \slv0
 \slv1
 \slv2
 \slv3
 \slv4
 \slv5
 \psline(1,0)(2,2)(3,2)
 \psline(2,0)(3,0)
 \psline(4,0)(5,0)(5,2)(4,2)
 \pscurve(4,2)(3,1.4)(1,1.4)(0,2)
\end{pspicture}}$.
The map $\P_n\to\P_n:\al\mt\al^*$ illustrates the regular $*$-semigroup structure of $\P_n$; for all $\al,\be\in\P_n$, we have
\[
(\al^*)^*=\al \COMMA 
(\al\be)^*=\be^*\al^* \COMMA
\al\al^*\al=\al.
\]
We also have $\codom(\al)=\dom(\al^*)$ and $\coker(\al)=\ker(\al^*)$ and other such identities.

We say a partition $\al\in\P_n$ is \emph{planar} if it has a graphical representation without any crossings.  The set of all planar partitions forms a submonoid of $\P_n$, and we denote this submonoid by $\PP_n$.  The Brauer monoid $\B_n$ is the submonoid of $\P_n$ consisting of all partitions whose blocks all have size~$2$.  The Jones monoid $\J_n$ is the intersection of $\PP_n$ with $\B_n$.  We will concentrate on the partition monoid itself in this section, and will return to the three submonoids in subsequent sections.  The next result was first proved (using different language) in \cite{Wilcox2007}; see also \cite{FitzGerald2011}.  It also follows from some of the above identities.

\begin{theorem}[Wilcox {\cite[Theorem~17]{Wilcox2007}}]
\label{thm_wilcox}
For each $\alpha, \beta \in \P_n$, we have\emph{:}
\begin{enumerate}
\item[(i)] $\alpha \gr \beta$ if and only if $\dom(\alpha)=\dom(\beta)$ and $\ker(\alpha)=\ker(\beta)$\emph{;}
\item[(ii)] $\alpha \gl \beta$ if and only if $\codom(\alpha)=\codom(\beta)$ and $\coker(\alpha)=\coker(\beta)$\emph{;} 
\item[(iii)] $\alpha \gj \beta$ if and only if $\rank(\alpha) = \rank(\beta)$.
\end{enumerate}
\end{theorem}

We will also require the following result from \cite{East2011_2}; see also \cite{JE2016_2,East2012}.  For $1\leq i\leq n$, let $\pi_i\in\P_n$ be the projection with domain $[n]\setminus\{i\}$ and kernel $\De$.  For $1\leq i<j\leq n$, let $\mathcal E_{ij}$ be the equivalence on $[n]$ whose only non-trivial equivalence class is $\{i,j\}$, and let $\pi_{ij}\in\P_n$ be the projection with domain $[n]$ and kernel $\mathcal E_{ij}$.  See Figure \ref{fig_rank_n-1_projections} for an illustration.

\begin{figure}[ht]
\begin{center}
\begin{tikzpicture}[xscale=.35,yscale=0.35]
	\fill (0,0)circle(.15)
	      (4,0)circle(.15)
	      (5,0)circle(.15)
	      (6,0)circle(.15)
	      (10,0)circle(.15)
	      (0,2)circle(.15)
	      (4,2)circle(.15)
	      (5,2)circle(.15)
	      (6,2)circle(.15)
	      (10,2)circle(.15);
	\draw (0,2)--(0,0)
	      (4,2)--(4,0)
	      (6,2)--(6,0)
	      (10,2)--(10,0);
  \draw [dotted] (0,2)--(4,2)
                 (6,2)--(10,2)
                 (0,0)--(4,0)
                 (6,0)--(10,0);
  \draw(5,2)node[above]{{\tiny $\phantom{j}i\phantom{j}$}};
  \draw(0,2)node[above]{{\tiny $\phantom{j}1\phantom{j}$}};
  \draw(10,2)node[above]{{\tiny $\phantom{j}n\phantom{j}$}};
  \draw(-.1,.9)node[left]{${}_{\phantom{j}}\pi_i=$};
	\end{tikzpicture}
\quad
\begin{tikzpicture}[xscale=.35,yscale=0.35]
\draw(3,2)arc(180:270:.4) (3.4,1.6)--(6.6,1.6) (6.6,1.6) arc(270:360:.4);
\draw(3,0)arc(180:90:.4) (3.4,.4)--(6.6,.4) (6.6,.4) arc(90:0:.4);
	\fill (0,0)circle(.15)
	      (2,0)circle(.15)
	      (3,0)circle(.15)
	      (4,0)circle(.15)
	      (6,0)circle(.15)
	      (7,0)circle(.15)
	      (8,0)circle(.15)
	      (10,0)circle(.15)
	      (0,2)circle(.15)
	      (2,2)circle(.15)
	      (3,2)circle(.15)
	      (4,2)circle(.15)
	      (6,2)circle(.15)
	      (7,2)circle(.15)
	      (8,2)circle(.15)
	      (10,2)circle(.15);
	\draw (0,2)--(0,0)
	      (2,2)--(2,0)
	      (3,2)--(3,0)
	      (4,2)--(4,0)
	      (6,2)--(6,0)
	      (7,2)--(7,0)
	      (8,2)--(8,0)
	      (10,2)--(10,0);
  \draw [dotted] (0,2)--(2,2)
                 (4,2)--(6,2)
                 (8,2)--(10,2)
                 (0,0)--(2,0)
                 (4,0)--(6,0)
                 (8,0)--(10,0);
  \draw(3,2)node[above]{{\tiny $\phantom{j}i\phantom{j}$}};
  \draw(7,2)node[above]{{\tiny $\phantom{j}j\phantom{j}$}};
  \draw(0,2)node[above]{{\tiny $\phantom{j}1\phantom{j}$}};
  \draw(10,2)node[above]{{\tiny $\phantom{j}n\phantom{j}$}};
  \draw(-.1,.9)node[left]{$\pi_{ij}=$};
	\end{tikzpicture}
\end{center}
\caption{The projections $\pi_i,\pi_{ij}\in\P_n$.}
\label{fig_rank_n-1_projections}
\end{figure}
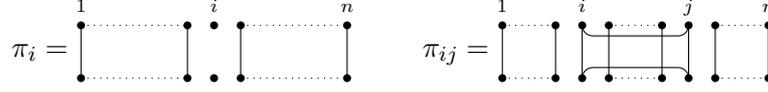

\begin{theorem}\label{thm_PnSn}
The singular part $\P_n\setminus\S_n$ of the partition monoid $\P_n$ is idempotent generated.  The set $\set{\pi_i}{1\leq i\leq n}\cup\set{\pi_{ij}}{1\leq i<j\leq n}$ is a minimal idempotent generating set.
\end{theorem}

Note that this is precisely the set of all projections of rank $n-1$. Defining relations for this generating set were given in \cite{East2011_2}, but we will not need them here.

If $A\subseteq[n]$, we write $A'=\{a':a\in A\}\subseteq[n]'$.
As in \cite{East2012}, if $\alpha \in \P_n$, we will write
\[
\alpha=\partn{A_i}{B_i}{C_j}{D_k}_{i \in I, \ j \in J, \ k \in K}
\]
to indicate that $\alpha$ has transversal blocks $A_i \cup B_i'$ $(i \in I)$, and nontransversal blocks $C_j$ $(j \in J)$ and $D_k'$ $(k \in K)$.  Note that one or more (but not all) of $I,J,K$ may be empty, and that $|I|=\rank(\al)$.  We will often use variations of this notation but it should always be clear what is meant.
The proof of the following lemma from \cite{East2012} is straightforward. 

\begin{lemma}\label{lem_projections_Pn}
A partition is a projection if and only if it is of the form
\[
\partn{A_i}{A_i}{C_j}{C_j}_{i \in I, \ j \in J}.
\]
\end{lemma}

For $0\leq r\leq n$, let
\[
\JrPn= \set{\al\in\P_n}{\rank(\al)=r}.
\]
By Theorem \ref{thm_wilcox}, these sets are precisely the $\gj$-classes of $\P_n$, and they form a chain:
\[
J_0(\P_n)< J_1(\P_n)<\cdots< J_{n-1}(\P_n)< \JnPn.
\]
It follows from Proposition \ref{prop_chain} that the ideals of $\P_n$ are precisely the sets
\[
\IrPn=J_0(\P_n)\cup J_1(\P_n)\cup\cdots\cup J_r(\P_n)=\set{\al\in\P_n}{\rank(\al)\leq r}.
\]
Note that $\InPn=\P_n$, $\JnPn=\S_n$, $\InmPn=\P_n\setminus\S_n$, and that the maximal subgroups contained in $J_r(\P_n)$ are all isomorphic to $\S_r$.  In what follows, we will apply the general results of Sections \ref{sect_0-simple} and \ref{sect_regular-*} to the (finite idempotent generated regular $*$-) semigroup $S=\InmPn=\P_n\setminus\S_n$.

\subsection{Rank and idempotent rank of ideals in $\P_n$}

The key lemma that allows us to reduce the problem to the consideration of principal factors is the following, which shows how elements of large rank may be used to generate elements of smaller rank. 

\begin{lemma}\label{lem_dropdownPn}
If $0 \leq r \leq n-2$, then $\JrPn \subseteq \langle \JrpPn \rangle$. 
\end{lemma}

\begin{proof}
Let $\alpha \in \JrPn$ be a projection where $0 \leq r \leq n-2$.  By (the proof of) Proposition~\ref{prop_chain}, it is enough to show that $\al\in\la\JrpPn\ra$.
Write
\[
\alpha =    \partn{A_i}{A_i}{C_j}{C_j}_{i \in I, \ j \in J},
\]
where $|I| = r$ and $|J| = k$. Without loss of generality, we may suppose that $I=[r]$ and $J=[k]$.  There are two cases to consider. 

\

\noindent \textbf{Case 1:} First suppose $|J| = k \geq 1$. In this case $\alpha = \beta \gamma$ where
\begin{align*}
\beta &= 
\partnlong{A_i}{i}{C_1}{r+1}{C_j}{x}_{i\in I,\ j\in [2,k],\ x\in[r+2,n]},
\\
\gamma  &= 
\partnlong{i}{A_i}{n}{C_1}{x}{C_j}_{i\in I,\ j\in [2,k],\ x\in[r+1,n-1]}
\end{align*}
both belong to $J_{r+1}(\P_n)$. 

\

\noindent \textbf{Case 2:} Next suppose $|J| = k =0$.  Without loss of generality, we may suppose that $|A_r|\geq2$.  Consider a non-trivial decomposition $A_r=A_r'\cup A_r''$ where $A_r'\cap A_r''=\varnothing$.  Then $\alpha=\beta\gamma$ where
\begin{align*}
\beta &= \partnlonger{A_i}{i}{A_r'}{r}{A_r''}{r+1,r+2}{\varnothing}{x}_{i\in [r-1],\ x\in[r+3,n]}, \\
\gamma &= \partnlonger{i}{A_i}{r,r+1}{A_r'}{r+2}{A_r''}{x}{\varnothing}_{i\in [r-1],\ x\in[r+3,n]}
\end{align*}
both belong to $J_{r+1}(\P_n)$. 
\end{proof}

\begin{theorem}
\label{thm_main1}
For $0\leq r\leq n-1$, the ideal $\IrPn$ is idempotent generated, and
\[
\rank(\IrPn) = \idrank(\IrPn) = \sum_{j=r}^n S(n,j)\binom jr = \sum_{j=r}^n \binom nj S(j,r) \Bell_{n-j}
\] 
where $S(j,r)$ denotes the Stirling number of the second kind, and $\Bell_k$ denotes the $k$th Bell number. 
Moreover, a subset $A \subseteq \IrPn$ of this cardinality is a generating set for $\IrPn$ if and only if the following three conditions hold\emph{:}
\begin{enumerate}
\item $\rank(\alpha) = r$ for all $\alpha \in A$\emph{;}
\item for all $\alpha, \beta \in A$ with $\alpha \neq \beta$, either $\ker(\al)\not=\ker(\be)$ or $\dom(\al)\not=\dom(\be)$\emph{;}
\item for all $\alpha, \beta \in A$ with $\alpha \neq \beta$, either $\coker(\al)\not=\coker(\be)$ or $\codom(\al)\not=\codom(\be)$.
\end{enumerate}
\end{theorem}

\begin{proof}
It follows from Lemma~\ref{lem_dropdownPn} and Proposition \ref{prop_chain} that $\IrPn$ is idempotent generated and $\rank(\IrPn) = \idrank(\IrPn) = \rho_{nr}$, where $\rho_{nr}$ is the number of $\gr$-classes in $\JrPn$.  To specify an $\gr$-class in $\JrPn$, we first choose $j$ kernel classes for some $j\in[r,n]$.  From these, we then choose $r$ classes to be part of the transversal blocks.  These choices may be made in $S(n,j)$ and $\binom jr$ ways, respectively.
Multiplying these and summing over all $j\in[r,n]$ gives 
\[
\rho_{nr} = \sum_{j=r}^n S(n,j)\binom jr.
\]
(Counting the $\gr$-classes in another way shows that $\rho_{nr} = \sum_{j=r}^n \binom nj S(j,r) \Bell_{n-j}$: here we first choose the domain, say of size $j$, in $\binom nj$ ways; we then choose the $r$ kernel-classes contained in the domain in $S(j,r)$ ways; we choose the remaining kernel-classes in $\Bell_{n-j}$ ways; finally, we multiply these values, then sum over all $j\in[r,n]$.)
Finally, a subset $A\sub\IrPn$ with $|A|=\rho_{nr}$ generates $\IrPn=\la\JrPn\ra$ if and only if it generates the principal factor $\JrPn^*$ which, by Theorem~\ref{thm_czsproj}, occurs if and only if $A$ is a transversal of the $\gr$- and $\gl$-classes of $\JrPn$.  By Theorem \ref{thm_wilcox}, this is equivalent to saying that conditions (1), (2) and (3) hold.
\end{proof}

\begin{remark}
As expected, this theorem agrees with Theorem~\ref{thm_PnSn} in the particular case $r=n-1$.  Note also that $\rank(I_0(\P_n))=\idrank(I_0(\P_n))=\Bell_n$, and that the identity $\rho_{n0}+\rho_{n1}=\rho_{n+1,0}$ is a consequence of the well-known recurrence $\Bell_{n+1}=\sum_{i=0}^n\binom ni\Bell_i$.  See Table~\ref{tab_rankIrPn} for some computed values of $\rank(\IrPn)=\idrank(\IrPn)$.  These numbers do not appear on \cite{Sloan:OEIS}.
\end{remark}

\begin{remark}\label{rem:|IrPn|}
We note that the formula given in the proof of Theorem \ref{thm_main1} for $\rho_{nr}$, the number of $\gr$-classes (and $\gl$-classes) contained in $\JrPn$, is valid for $r=n$, giving a value of $\rho_{nn}=1$ (but not agreeing with the value $\rank(\InPn)=\rank(\P_n)=4$, as calculated in \cite{East2011_2}).
Since also any $\gh$-class contained in $\JrPn$ has size $r!$, it follows that $|\JrPn|=\rho_{nr}^2r!$.  This yields a formula for the cardinality of the ideals of $\P_n$:
\[
|\IrPn| = \sum_{i=0}^r \rho_{ni}^2i! \qquad\text{for any $0\leq r\leq n$.}
\]
\end{remark}

\begin{table}[ht]%
\begin{tabular}{|c|cccccccccc|}
\hline
$n$ $\setminus$ $r$	&0	&1	&2	&3	&4	&5	&6	&7 &8 &9      \\
\hline
1&            1&            &            &            &            &            &            &            &            &            \\
2&            2&           3&            &            &            &            &            &            &            &            \\
3&            5&          10&           6&            &            &            &            &            &            &            \\
4&           15&          37&          31&          10&            &            &            &            &            &            \\
5&           52&         151&         160&          75&          15&            &            &            &            &            \\
6&          203&         674&         856&         520&         155&          21&            &            &            &            \\
7&          877&        3263&        4802&        3556&        1400&         287&          28&            &            &            \\
8&         4140&       17007&       28337&       24626&       11991&        3290&         490&          36&            &            \\
9&        21147&       94828&      175896&      174805&      101031&       34671&        6972&         786&          45&            \\
10&      115975&      562595&     1146931&     1279240&      853315&      350889&       88977&       13620&        1200&          55\\
\hline
\end{tabular}
\caption{Values of $\rank(I_r(\P_n)) = \idrank(I_r(\P_n))$.}
\label{tab_rankIrPn}
\end{table}

\subsection{Minimal idempotent generating sets of $\boldmath \P_n\setminus\S_n$}
\label{sect_PnSn}

Theorem \ref{thm_balanced} above gives a correspondence between minimal idempotent generating sets of $$\PnSn=\InmPn=\la\JnmPn\ra$$ and balanced subgraphs of the projection graph $\Gamma(\PnSn)=\Gamma(\JnmPn^*)$, in the sense of Definitions~\ref{def_GammaS} and \ref{def_balanced_subgraph}, which, for simplicity, we will denote by $\Gamma_n$.  We will also write~$\G_n$ for the set of all balanced subgraphs of $\Gamma_n$.
Parts of the next lemma were also used in~\cite{East2013}.

\begin{lemma}\label{lem_idempotents_in_J_{n-1}}
The set of idempotents of $\JnmPn$ is
\[
\{\pi_i:1\leq i\leq n\}\cup\{\pi_{ij},\lambda_{ij},\lambda_{ji},\rho_{ij},\rho_{ji}:1\leq i<j\leq n\},
\]
where these partitions are illustrated in Figures \ref{fig_rank_n-1_projections} and \ref{fig_rank_n-1_idempotents}.  The set of projections of $\JnmPn$ is
\[
\{\pi_i:1\leq i\leq n\}\cup\{\pi_{ij}:1\leq i<j\leq n\}.
\]
In the principal factor $\JnmPn^*$, the only nonzero products of pairs of projections are
\[
\pi_{ij}^2=\pi_{ij}, \ \  \pi_i^2=\pi_i, \ \ 
\pi_{ij}\pi_j=\lambda_{ij}, \ \  \pi_{ij}\pi_i=\lambda_{ji}, \ \ 
\pi_i\pi_{ij}=\rho_{ij}, \ \  \pi_j\pi_{ij}=\rho_{ji}.
\]
\end{lemma}

\begin{proof}
The statement about projections follows quickly from Lemma \ref{lem_projections_Pn}.  By Lemma \ref{lem_projections}, any idempotent is the product of two projections, and it is easy to check that the products of the stated projections give only the stated idempotents as well as lower rank ones.
\end{proof}

\begin{figure}[ht]
\begin{center}
\begin{tikzpicture}[xscale=.35,yscale=0.35]
\draw(3,2)arc(180:270:.4) (3.4,1.6)--(6,1.6) (6,1.6) to[out=0,in=30] (7,2);
	\fill (0,0)circle(.15)
	      (2,0)circle(.15)
	      (3,0)circle(.15)
	      (4,0)circle(.15)
	      (6,0)circle(.15)
	      (7,0)circle(.15)
	      (8,0)circle(.15)
	      (10,0)circle(.15)
	      (0,2)circle(.15)
	      (2,2)circle(.15)
	      (3,2)circle(.15)
	      (4,2)circle(.15)
	      (6,2)circle(.15)
	      (7,2)circle(.15)
	      (8,2)circle(.15)
	      (10,2)circle(.15);
	\draw (0,2)--(0,0)
	      (2,2)--(2,0)
	      (3,2)--(3,0)
	      (4,2)--(4,0)
	      (6,2)--(6,0)
	      (7,2)--(3,0)
	      (8,2)--(8,0)
	      (10,2)--(10,0);
  \draw [dotted] (0,2)--(2,2)
                 (4,2)--(6,2)
                 (8,2)--(10,2)
                 (0,0)--(2,0)
                 (4,0)--(6,0)
                 (8,0)--(10,0);
  \draw(3,2)node[above]{{\tiny $\phantom{j}i\phantom{j}$}};
  \draw(7,2)node[above]{{\tiny $\phantom{j}j\phantom{j}$}};
  \draw(0,2)node[above]{{\tiny $\phantom{j}1\phantom{j}$}};
  \draw(10,2)node[above]{{\tiny $\phantom{j}n\phantom{j}$}};
  \draw(-.1,.9)node[left]{$\lambda_{ij}=$};
	\end{tikzpicture}
\quad
\begin{tikzpicture}[xscale=.35,yscale=0.35]
\draw(3,2) to[out=330,in=180] (4,1.6) (4,1.6)--(6.6,1.6) (6.6,1.6)arc(270:360:.4);
	\fill (0,0)circle(.15)
	      (2,0)circle(.15)
	      (3,0)circle(.15)
	      (4,0)circle(.15)
	      (6,0)circle(.15)
	      (7,0)circle(.15)
	      (8,0)circle(.15)
	      (10,0)circle(.15)
	      (0,2)circle(.15)
	      (2,2)circle(.15)
	      (3,2)circle(.15)
	      (4,2)circle(.15)
	      (6,2)circle(.15)
	      (7,2)circle(.15)
	      (8,2)circle(.15)
	      (10,2)circle(.15);
	\draw (0,2)--(0,0)
	      (2,2)--(2,0)
	      (3,2)--(7,0)
	      (4,2)--(4,0)
	      (6,2)--(6,0)
	      (7,2)--(7,0)
	      (8,2)--(8,0)
	      (10,2)--(10,0);
  \draw [dotted] (0,2)--(2,2)
                 (4,2)--(6,2)
                 (8,2)--(10,2)
                 (0,0)--(2,0)
                 (4,0)--(6,0)
                 (8,0)--(10,0);
  \draw(3,2)node[above]{{\tiny $\phantom{j}i\phantom{j}$}};
  \draw(7,2)node[above]{{\tiny $\phantom{j}j\phantom{j}$}};
  \draw(0,2)node[above]{{\tiny $\phantom{j}1\phantom{j}$}};
  \draw(10,2)node[above]{{\tiny $\phantom{j}n\phantom{j}$}};
  \draw(-.1,.9)node[left]{$\lambda_{ji}=$};
	\end{tikzpicture}
\end{center}
\begin{center}
\begin{tikzpicture}[xscale=.35,yscale=0.35]
\draw(3,0) to[out=30,in=180] (4,.4) (4,.4)--(6.6,.4) (6.6,.4)arc(90:0:.4);
	\fill (0,0)circle(.15)
	      (2,0)circle(.15)
	      (3,0)circle(.15)
	      (4,0)circle(.15)
	      (6,0)circle(.15)
	      (7,0)circle(.15)
	      (8,0)circle(.15)
	      (10,0)circle(.15)
	      (0,2)circle(.15)
	      (2,2)circle(.15)
	      (3,2)circle(.15)
	      (4,2)circle(.15)
	      (6,2)circle(.15)
	      (7,2)circle(.15)
	      (8,2)circle(.15)
	      (10,2)circle(.15);
	\draw (0,2)--(0,0)
	      (2,2)--(2,0)
	      (7,2)--(7,0)
	      (4,2)--(4,0)
	      (6,2)--(6,0)
	      (7,2)--(3,0)
	      (8,2)--(8,0)
	      (10,2)--(10,0);
  \draw [dotted] (0,2)--(2,2)
                 (4,2)--(6,2)
                 (8,2)--(10,2)
                 (0,0)--(2,0)
                 (4,0)--(6,0)
                 (8,0)--(10,0);
  \draw(3,2)node[above]{{\tiny $\phantom{j}i\phantom{j}$}};
  \draw(7,2)node[above]{{\tiny $\phantom{j}j\phantom{j}$}};
  \draw(0,2)node[above]{{\tiny $\phantom{j}1\phantom{j}$}};
  \draw(10,2)node[above]{{\tiny $\phantom{j}n\phantom{j}$}};
  \draw(-.1,.9)node[left]{$\rho_{ij}=$};
	\end{tikzpicture}
\quad
\begin{tikzpicture}[xscale=.35,yscale=0.35]
\draw(3,0)arc(180:90:.4) (3.4,.4)--(6,.4) (6,.4) to[out=0,in=150] (7,0);
	\fill (0,0)circle(.15)
	      (2,0)circle(.15)
	      (3,0)circle(.15)
	      (4,0)circle(.15)
	      (6,0)circle(.15)
	      (7,0)circle(.15)
	      (8,0)circle(.15)
	      (10,0)circle(.15)
	      (0,2)circle(.15)
	      (2,2)circle(.15)
	      (3,2)circle(.15)
	      (4,2)circle(.15)
	      (6,2)circle(.15)
	      (7,2)circle(.15)
	      (8,2)circle(.15)
	      (10,2)circle(.15);
	\draw (0,2)--(0,0)
	      (2,2)--(2,0)
	      (3,2)--(7,0)
	      (4,2)--(4,0)
	      (6,2)--(6,0)
	      (3,2)--(3,0)
	      (8,2)--(8,0)
	      (10,2)--(10,0);
  \draw [dotted] (0,2)--(2,2)
                 (4,2)--(6,2)
                 (8,2)--(10,2)
                 (0,0)--(2,0)
                 (4,0)--(6,0)
                 (8,0)--(10,0);
  \draw(3,2)node[above]{{\tiny $\phantom{j}i\phantom{j}$}};
  \draw(7,2)node[above]{{\tiny $\phantom{j}j\phantom{j}$}};
  \draw(0,2)node[above]{{\tiny $\phantom{j}1\phantom{j}$}};
  \draw(10,2)node[above]{{\tiny $\phantom{j}n\phantom{j}$}};
  \draw(-.1,.9)node[left]{$\rho_{ji}=$};
	\end{tikzpicture}
\end{center}
\caption{The non-projection idempotents of $\JnmPn$; see also Figure \ref{fig_rank_n-1_projections}.}
\label{fig_rank_n-1_idempotents}
\end{figure}

\begin{remark}
So there are $n+5\binom n2=(5n^2-3n)/2$ idempotents in $\JnmPn$.  The idempotents in an arbitrary $\gj$-class $J_r(\P_n)$ are enumerated in \cite{enum}.
\end{remark}

In light of Lemma \ref{lem_idempotents_in_J_{n-1}}, we see that the projection graph $\Gamma_n$ is obtained from the complete graph on vertex set $[n]$ by replacing each edge
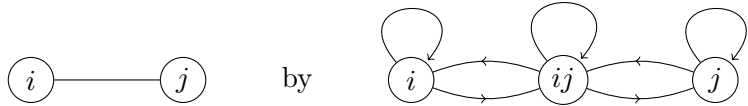
\begin{center}
\begin{tikzpicture}[scale=1]
\tikzstyle{vertex}=[circle,draw=black, fill=white,minimum size=17pt, inner sep = 0.07cm]
\tikzstyle{vertexdash}=[circle,draw=white, fill=white, inner sep = 0.07cm]
\node[vertex] (a) at (-5,0) {  $i$ };
\node[vertex] (b) at (-3,0) {  $j$ };
\draw [-] (a) to (b);
\draw(-1.5,0)node{by};
\node[vertex] (1) at (0,0) {  $i$ };
\node[vertex] (2) at (2,0) {  $ij$ };
\node[vertex] (3) at (4,0) {  $j$ };
\draw [->-=0.5] (1)        to [bend right=20] (2);
\draw [-<-=0.5] (1)        to [bend left=20] (2);
\draw [->-=0.5] (2)        to [bend right=20] (3);
\draw [-<-=0.5] (2)        to [bend left=20] (3);
\draw [->] (1) edge [out=130,in=50,loop] ();
\draw [->] (2) edge [out=130,in=50,loop] ();
\draw [->] (3) edge [out=130,in=50,loop] ();
\draw(4.3,0)node{\phantom{hj}.};
\end{tikzpicture}
%}
\end{center}
For convenience, we have labeled the vertices $i$ and $ij$ instead of $\pi_i$ and $\pi_{ij}$.  The graph $\Gamma_5$ is pictured in Figure \ref{fig_Gamma_5}, with the same labelling convention, and with loops omitted.

\begin{figure}[ht]
\begin{center}
\scalebox{0.8}
{
\begin{tikzpicture}[scale=3]
\tikzstyle{vertex}=[circle,draw=black, fill=white, inner sep = 0.07cm]
\node[vertex] (1) at (0	,1) {\Large  $1$ };
\node[vertex] (2) at (0.951058238	,0.309011695) {\Large  $2$ };
\node[vertex] (3) at (0.587781603	,-0.809019646) {\Large  $3$ };
\node[vertex] (4) at (-0.587788043	,-0.809014967) {\Large  $4$ };
\node[vertex] (5) at (-0.951055778	,0.309019266) {\Large  $5$ };
\node[vertex] (12) at (0.475529119	,0.654505847) {  $12$ };
\node[vertex] (13) at (0.293890801	,0.095490177) {  $13$ };
\node[vertex] (14) at (-0.293894022	,0.095492517) {  $14$ };
\node[vertex] (15) at (-0.475527889	,0.654509633) {  $15$ };
\node[vertex] (23) at (0.769419921	,-0.250003976) {  $23$ };
\node[vertex] (24) at (0.181635098	,-0.250001636) {  $24$ };
\node[vertex] (25) at (0	,0.30901548) {  $25$ };
\node[vertex] (34) at (0	,-0.809017306) {  $34$ };
\node[vertex] (35) at (-0.181637088	,-0.25000019) {  $35$ };
\node[vertex] (45) at (-0.769421911	,-0.24999785) {  $45$ };
\dde12
\dde23
\dde34
\dde45
\ddeb51
\dden13
\dden35
\ddenb52
\dden24
\ddenb41
\end{tikzpicture}
}
\end{center}
\caption{The graph $\Gamma_5=\Gamma(\P_5\setminus\S_5)$ with loops omitted.}
\label{fig_Gamma_5}
\end{figure}
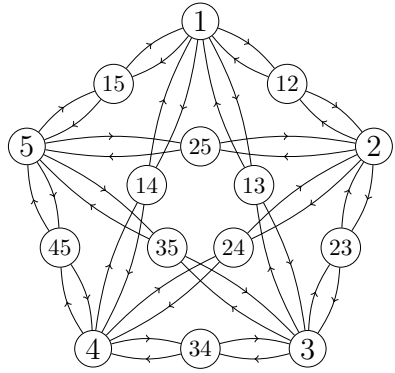

The edges of $\Gamma_n$ correspond to the idempotents of $\JnmPn$ via:
\[
\begin{array}{rclrcl}
i\to i \,\,\,&\equiv&\,\,\, \pi_i ,\quad 
&ij\to ij \,\,\,&\equiv&\,\,\, \pi_{ij} ,\\
ij\to j \,\,\,&\equiv&\,\,\, \lambda_{ij} ,\quad 
&ij\to i \,\,\,&\equiv&\,\,\, \lambda_{ji} ,\\
i\to ij \,\,\,&\equiv&\,\,\, \rho_{ij} ,\quad 
&j\to ij \,\,\,&\equiv&\,\,\, \rho_{ji}.
\end{array}
\]
As mentioned above, Theorem \ref{thm_balanced} shows that the minimal idempotent generating sets of $\PnSn=\InmPn$ correspond to the balanced subgraphs of $\Gamma_n$.  The generating set from Theorem \ref{thm_PnSn}, which consists of all projections from $\JnmPn$, corresponds to the balanced subgraph whose edges are the $n+\binom n2=\binom {n+1}2$ loops of $\Gamma_n$.

We now turn to the task of enumerating $\G_n$, the set of all balanced subgraphs of $\Gamma_n$.  In what follows, for convenience, we will use symmetric notation for the vertices $ij$, and use $ij$ and $ji$ interchangeably without intending to imply $i<j$ or $i>j$.

Let $G\in\G_n$.  The in-degree/out-degree condition is equivalent to saying that $G$ is a disjoint union of circuits.  By inspecting $\Gamma_n$, we see that the circuits of $G$ must be of one of the following four types:
\begin{enumerate}
	\item $i_1 \to i_1i_2 \to i_2 \to i_2i_3 \to i_3 \to \cdots \to i_k \to i_ki_1 \to i_1$ where $k\geq3$ and $i_1,i_2,\ldots,i_k$ are distinct,
	\item $ij\to ij$,
	\item $i\to i$,
	\item $i\to ij\to i$.
\end{enumerate}
Note that if $n$ is large, then most connected components of $G$ will be loops of type (2).  Also note that $G$ is completely determined by its circuits of type (1), (3) and (4).  Our goal is to show that $G$ determines (and is determined by) a pair $(\pi_G,\tau_G)$ where $\pi_G$ is a permutation of a subset $A_G\subseteq[n]$ that has no fixed points and no $2$-cycles, and $\tau_G$ is a function $[n]\setminus A_G\to[n]$ that has no $2$-cycles.  Here we say that a function $\phi:Y\to X$ with $Y\subseteq X$ has a $2$-cycle if there exists $x,y\in Y$ with $x\not=y$, $x\phi=y$ and $y\phi=x$.  With this goal in mind, we define
\[
A_G = \{ i\in[n] : \text{vertex $i$ is contained in a circuit of type (1)} \}.
\]
Note that $|A_G|\in\{0\}\cup[3,n]$.  We define $\pi_G:A_G\to A_G$ to be the permutation whose cycle decomposition includes a cycle $(i_1,i_2,\ldots,i_k)$ corresponding to each circuit of $G$ of type (1).  Note that $\pi_G$ has no $2$-cycles or fixed points.  Note also that if $A_G=\varnothing$, then $\pi_G$ is the unique function $\varnothing\to\varnothing$.  We also define $\tau_G:[n]\setminus A_G\to[n]$ by
\[
i\tau_G = \begin{cases}
i &\text{if $G$ contains the loop $i\to i$}\\
j &\text{if $G$ contains the circuit $i\to ij\to i$.}
\end{cases}
\]
Note that $\tau_G$ contains no $2$-cycles, but might have fixed points.  So $G$ uniquely determines the pair $(\pi_G,\tau_G)$.  Conversely, a pair $(\pi,\tau)$ for which
\begin{itemize}
	\item[(B1)] $\pi$ is a permutation, of some subset $A\subseteq[n]$, with no fixed points and $2$-cycles, and
	\item[(B2)] $\tau:[n]\setminus A\to[n]$ has no $2$-cycles
\end{itemize}
determines a balanced subgraph of $\Gamma_n$ 
in such a way that gives a bijective correspondence between $G \in \G_n$ and  
pairs $(\pi, \tau)$ satisfying (B1) and (B2). 
So it suffices to enumerate such pairs.
\begin{remark}
The functions $\pi_G$ and $\tau_G$ could be combined to give a transformation ${\mu_G:[n]\to[n]}$ defined by
\[
i\mu_G=\begin{cases}
i\pi_G &\text{if $i\in A_G$}\\
i\tau_G &\text{if $i\in [n]\setminus A_G$.}
\end{cases}
\]
However, $G$ is not uniquely determined by $\mu_G$.  For example, in $\Gamma_3$, the two balanced subgraphs shown in Figure~\ref{fig_2balanced} both have $\mu_G=(1,2,3)$.
\end{remark}

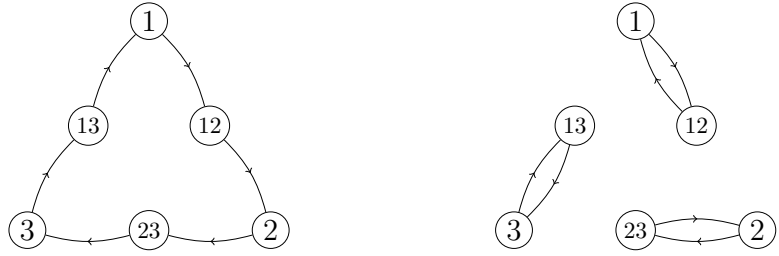
\begin{figure}
\begin{center}
\scalebox{0.8}
{
\begin{tikzpicture}[scale=2]
\tikzstyle{vertex}=[circle,draw=black, fill=white, inner sep = 0.07cm]
\node[vertex] (1) at (1,1.732) {\Large  $1$ };
\node[vertex] (2) at (2,0) {\Large  $2$ };
\node[vertex] (3) at (0,0) {\Large  $3$ };
\node[vertex] (12) at (1.5,0.866) {  $12$ };
\node[vertex] (13) at (0.5,0.866) {  $13$ };
\node[vertex] (23) at (1,0) {  $23$ };
\draw [->-=0.5] (1) to  [bend left=15] (12);
\draw [->-=0.5] (12) to [bend left=15]  (2);
\draw [->-=0.5] (2) to [bend left=15]  (23);
\draw [->-=0.5] (23) to [bend left=15]  (3);
\draw [->-=0.5] (3) to [bend left=15]  (13);
\draw [->-=0.5] (13) to [bend left=15]  (1);
\node[vertex] (1') at (1+4,1.732) {\Large  $1$ };
\node[vertex] (2') at (2+4,0) {\Large  $2$ };
\node[vertex] (3') at (0+4,0) {\Large  $3$ };
\node[vertex] (12') at (1.5+4,0.866) {  $12$ };
\node[vertex] (13') at (0.5+4,0.866) {  $13$ };
\node[vertex] (23') at (1+4,0) {  $23$ };
\draw [->-=0.5] (1') to [bend left=15] (12');
\draw [->-=0.5] (12') to [bend left=15] (1');
\draw [->-=0.5] (2') to [bend left=15] (23');
\draw [->-=0.5] (23') to [bend left=15] (2');
\draw [->-=0.5] (3') to [bend left=15] (13');
\draw [->-=0.5] (13') to [bend left=15] (3');
\end{tikzpicture}
}
\end{center}
\caption{Two balanced subgraphs in $\Gamma_3$, both with $\mu_G=(1,2,3)$.}
\label{fig_2balanced}
\end{figure}

For $0\leq k\leq n$, let
\begin{align*}
A_k &= \{ \pi\in\S_k : \text{$\pi$ has no fixed points or $2$-cycles} \} \\
B_{nk} &= \{ \tau:[k]\to[n] : \text{$\tau$ has no $2$-cycles} \},
\end{align*}
and put
\[
a_k=|A_k| \quad\text{and}\quad b_{nk}=|B_{nk}|.
\]
Note that for any subset $X\subseteq[n]$ with $|X|=k$, the set
\[
\{ \tau:X\to[n] : \text{$\tau$ has no $2$-cycles} \}
\]
has cardinality $b_{nk}$.  If $A\subseteq[n]$ with $|A|=k$, the number of pairs $(\pi,\tau)$ satisfying conditions (B1) and (B2) above is equal to $a_kb_{n,n-k}$.  It follows that
\[
|\G_n| = \sum_{k=0}^n \binom nk a_k b_{n,n-k},
\]
so it remains to evaluate the sequences $a_k$ and $b_{nk}$.  The sequence $a_k$ is well-known; it is A038205 on \cite{Sloan:OEIS}, but we prove the next result for completeness.

\begin{lemma}\label{lem_ak}
The sequence $a_k$ satisfies the recurrence
\[
a_0=1,\ \ a_1=a_2=0, \ \ a_{k+1}=ka_k+k(k-1)a_{k-2} \ \ \text{for $k\geq2$.}
\]
\end{lemma}

\begin{proof}
The values for $a_0,a_1,a_2$ are clear.  Now consider an element $\pi\in A_{k+1}$ where $k\geq2$.  There are two possibilities: either (i) $k+1$ is in an $l$-cycle of $\pi$ where $l\geq4$, or (ii) $k+1$ is in a $3$-cycle of $\pi$.  It is easy to see that there are $ka_k$ elements of type (i), and $k(k-1)a_{k-2}$ of type (ii).
\end{proof}

The first few values of $a_k$ are given in Table \ref{tab_ak}.

\begin{table}[ht]%
\begin{center}
\begin{tabular}{|c|ccccccccccc|}
\hline
$k$ & $0$ & $1$ & $2$ & $3$ & $4$ & $5$ & $6$ & $7$ & $8$ & $9$ & $10$ \\
\hline
$a_k$  & $1$ &  $0$ &  $0$ &  $2$  & $6$  & $24$  & $160$  & $1140$  & $8988$  & $80864$ & $809856$ \\
\hline
\end{tabular}
\end{center}
\caption{The sequence $a_k$.}
\label{tab_ak}
\end{table}

\begin{lemma}\label{lem_bnk}
For any $0\leq k\leq n$, we have
\[
b_{nk} = \sum_{i=0}^{\lfloor \frac{k}{2} \rfloor} (-1)^i\binom k{2i}(2i-1)!!n^{k-2i},
\]
where $(2i-1)!!=(2i-1)(2i-3)\cdots3\cdot1$ and we interpret $(-1)!!=1$.
\end{lemma}

\begin{proof}
For $1\leq r<s\leq k$, let
\[
C_{nk}(r,s) = \{ \tau:[k]\to[n]: \text{$(r,s)$ is a $2$-cycle of $\tau$}\}.
\]
Then, since there are $n^k$ functions $[k]\to[n]$,
\begin{equation}\label{eq_bnk}
b_{nk}=n^k-\left| \bigcup_{1\leq r<s\leq k} C_{nk}(r,s) \right|.
\end{equation}
Now, if $1\leq i\leq\lfloor \frac{k}{2}\rfloor$ and $(r_1,s_1),\ldots,(r_i,s_i)$ are disjoint cycles, then
\[
\big| C_{nk}(r_1,s_1)\cap\cdots\cap C_{nk}(r_i,s_i) \big| = n^{k-2i},
\]
and there are $\binom k{2i}(2i-1)!!$ ways to choose $i$ disjoint cycles from $[k]$.  So the inclusion-exclusion formula gives
\[
\left| \bigcup_{1\leq r<s\leq k} C_{nk}(r,s) \right| = \sum_{i=1}^{\lfloor \frac{k}{2}\rfloor} (-1)^{i+1}\binom k{2i}(2i-1)!!n^{k-2i}.
\]
The result now follows from \eqref{eq_bnk}, since $n^k$ is the $i=0$ term of the sum in the statement of the lemma.
\end{proof}

The numbers $b_{nk}$ do not appear in \cite{Sloan:OEIS}.  The first few values are included in Table \ref{tab_bnk}.
We have proved the following.

\begin{table}[ht]%
\begin{tabular}{|c|ccccccccccc|}
\hline
$n$ $\setminus$ $k$	&0	&1	&2	&3	&4	&5	&6	&7 &8 &9 &10     \\
\hline
0	&1	&	&	&	&	&	&	&	&&&      \\
1	&1	&1	&	&	&	&	&	&	&&&       \\
2	&1	&2	&3	&	&	&	&	&	&&&       \\
3	&1	&3	&8	&18	&	&	&	&	&&&       \\
4	&1	&4	&15	&52	&163	&	&	&	&&&       \\
5	&1	&5	&24	&110	&478	&1950	&	&	&&&       \\
6	&1	&6	&35	&198	&1083	&5706	&28821	&	&&&       \\
7	&1	&7	&48	&322	&2110	&13482	&83824	&505876 &&& \\
8 &1	&8	&63	&488	&3715	&27768	&203569	&1461944	&10270569		& & \\
9 &1	&9	&80	&702	&6078	&51894	&436656	&3618540	&29510268	&236644092	& \\
10 &1	&10	&99	&970	&9403	&90150	&854485	&8003950	&74058105	&676549450	&6098971555 \\
\hline
\end{tabular}
\caption{The numbers $b_{nk}$.}
\label{tab_bnk}
\end{table}

\begin{theorem}
The number of minimal idempotent generating sets for the singular part $\PnSn$ of the partition monoid $\P_n$ is equal to
\[
|\G_n| = \sum_{k=0}^n\binom nka_kb_{n,n-k},
\]
where formulae for the numbers $a_k$ and $b_{nk}$ are given in Lemmas \ref{lem_ak} and \ref{lem_bnk}.
\end{theorem}

The first few values of $|\G_n|$ are given in Table \ref{tab_gn}; this sequence also does not appear in~\cite{Sloan:OEIS}.

\begin{table}[ht]%
\begin{center}
\begin{tabular}{|c|ccccccccccc|}
\hline
$n$ & $0$ & $1$ & $2$ & $3$ & $4$ & $5$ & $6$ & $7$ & $8$ & $9$ & $10$ \\
\hline
$|\mathcal{G}_n|$  & $1$ &  $1$ &  $3$ &  $20$  & $201$  & $2604$  & $40915$  & $754368$  & $15960945$  & $381141008$ & $10139372451$ \\
\hline
\end{tabular}
\end{center}
\caption{The numbers $|\G_n|$, which give the number of minimal idempotent generating sets for $\PnSn$.}
\label{tab_gn}
\end{table}

\subsection{Arbitrary idempotent generating sets for $\boldmath \PnSn$}

Given a set $F$ consisting of idempotents from $J_{n-1}(\P_n)$, we would like to know whether $F$ is a generating set of $\P_n \setminus \S_n$.
For a subset $F$ of
\[
\{\al\in E(\P_n):\rank(\al)=n-1\} = \{\pi_i:1\leq i\leq n\}\cup\{\pi_{ij},\lambda_{ij},\lambda_{ji},\rho_{ij},\rho_{ji}:1\leq i<j\leq n\},
\]
let $\Gamma_n(F)$ be the two-coloured digraph obtained by colouring each edge of $\Gamma_n$ blue, and then adding red edges corresponding to the idempotents from $F$: 
\[
\]
\vspace{-1.4cm}
\begin{itemize}
\begin{multicols}{3}
	\item $i{\boldsymbol{\red\to}}i$ if $\pi_i\in F$,
	\item $ij{\boldsymbol{\red\to}}ij$ if $\pi_{ij}\in F$,
	\item $ij{\boldsymbol{\red\to}}j$ if $\lam_{ij}\in F$,
	\item $ij{\boldsymbol{\red\to}}i$ if $\lam_{ji}\in F$,
	\item $i{\boldsymbol{\red\to}}ij$ if $\rho_{ij}\in F$,
	\item $j{\boldsymbol{\red\to}}ij$ if $\rho_{ji}\in F$.
\end{multicols}
\end{itemize}
\vspace{-0.2cm}
(As above, we will denote the vertices of $\Gamma_n$ by $i$ and $ij$ rather than $\pi_i$ and $\pi_{ij}$.) Applying the general result Theorem~\ref{thm_generalRBR} we obtain the following.

\begin{theorem}\label{thm_PnMinusSnGeneral}
For $F\subseteq \{\al\in E(\P_n):\rank(\al)=n-1\}$, the following are equivalent\emph{:}
\begin{enumerate}
	\item[(i)] $\PnSn=\la F\ra$\emph{;}
	\item[(ii)] each vertex of $\Gamma_n(F)$ is the base point of an RBR-alternating circuit. 
\end{enumerate}
\end{theorem}

We currently do not know of any simpler necessary and sufficient condition for such a subset $F$ of idempotents 
to be a generating set.  It would be desirable to give such a condition in terms of only the red edges of the graph $\Gamma_n(F)$.  As mentioned in Remark~\ref{rem_generalRBR}, one might be tempted to conjecture that $\lb F\rb = \PnSn$ if and only if, in $\Gamma_n(F)$, either:
\begin{itemize}
	\item[(i)] each vertex has at least one red edge coming in to it and at least one going out of it, or
	\item[(ii)] each vertex is contained in a red circuit.
\end{itemize}
As we also mentioned in Remark~\ref{rem_generalRBR}, Theorem~\ref{thm_generalRBR} tells us that condition (i) is necessary, while condition (ii) is sufficient.  But the following two examples show that neither condition is necessary \emph{and} sufficient.

\begin{example}
\label{Example3}
Consider the set of idempotents $F=\{\pi_1,\pi_2,\lambda_{12},\rho_{12}\}$ in the partition monoid $\P_2$. The digraph $\Gamma_2(F)$ is illustrated in Figure~\ref{fig_G2F}. Clearly, there is no $RBR$-alternating circuit based at the vertex $\{12\}$. It follows that $\pi_{12}\not\in\la F\ra$, and so $F$ does not generate $\P_2 \setminus \S_2$.  This example shows that (i) is not a sufficient condition. 
\end{example}
\begin{figure}
\begin{center}
%\begin{tikzpicture}[scale=1]
%\tikzstyle{vertex}=[circle,draw=black, fill=white, inner sep = 0.07cm]
%%
%\node[vertex] (1) at (0,0) {  $1$ };
%\node[vertex] (12) at (2,0) { $\scriptstyle{12}$};
%\node[vertex] (2) at (4,0) {  $2$ };
%%
%\draw [->-=0.5, color=red] (1)        to [bend right=40] (12);
%\draw [->-=0.5, color=red] (12)        to [bend right=40] (2);
%\draw [->, color=red] (1) edge [out=130+180,in=50+180,loop] ();
%\draw [->, color=red] (2) edge [out=130+180,in=50+180,loop] ();
%%
%%
%\draw [ultra thick, ->, color=blue] (1) edge [out=130,in=50,loop] ();
%\draw [ultra thick, ->, color=blue] (2) edge [out=130,in=50,loop] ();
%\draw [ultra thick, ->, color=blue] (12) edge [out=130,in=50,loop] ();
%\draw [ultra thick, ->-=0.5, color=blue] (1)        to [bend right=15] (12);
%\draw [ultra thick, ->-=0.5, color=blue] (12)        to [bend right=15] (2);
%\draw [ultra thick, ->-=0.5, color=blue] (12)        to [bend right=15] (1);
%\draw [ultra thick, ->-=0.5, color=blue] (2)        to [bend right=15] (12);
%\end{tikzpicture}
%\qquad
\begin{tikzpicture}[scale=1]
\tikzstyle{vertex}=[circle,draw=black, fill=white, inner sep = 0.07cm]
\node[vertex] (1) at (0,0) {  $1$ };
\node[vertex] (12) at (2,0) { $\scriptstyle{12}$};
\node[vertex] (2) at (4,0) {  $2$ };
\draw [ultra thick, ->-=0.5, color=red] (1)        to [bend right=40] (12);
\draw [ultra thick, ->-=0.5, color=red] (12)        to [bend right=40] (2);
\draw [ultra thick, ->, color=red] (1) edge [out=130+180,in=50+180,loop] ();
\draw [ultra thick, ->, color=red] (2) edge [out=130+180,in=50+180,loop] ();
\draw [->, color=blue] (1) edge [out=130,in=50,loop] ();
\draw [->, color=blue] (2) edge [out=130,in=50,loop] ();
\draw [->, color=blue] (12) edge [out=130,in=50,loop] ();
\draw [->-=0.5, color=blue] (1)        to [bend right=15] (12);
\draw [->-=0.5, color=blue] (12)        to [bend right=15] (2);
\draw [->-=0.5, color=blue] (12)        to [bend right=15] (1);
\draw [->-=0.5, color=blue] (2)        to [bend right=15] (12);
\end{tikzpicture}
\end{center}
\caption{The digraph $\Gamma_2(F)$ where $F=\{\pi_1,\pi_2,\lambda_{12},\rho_{12}\}$ in the partition monoid $\P_2$.}
\label{fig_G2F}
\end{figure}
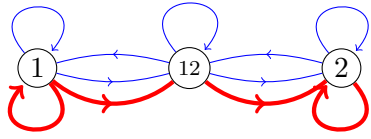
\begin{example}
\label{Example6}
Consider the set of idempotents 
$
F=\{\pi_2,\pi_3,\pi_{12},\pi_{23},\pi_{13},\lambda_{31},\rho_{12}\}
$
in the partition monoid $\P_3$. The digraph $\Gamma_3(F)$ is illustrated in Figure~\ref{Fig_G3F}. 
Every vertex has a red loop (the simplest kind of RBR-circuit) with the exception of vertex~$1$. Vertex $1$ is the basepoint of the RBR-alternating circuit:
\[
1{\boldsymbol{\red\to}} 12{\blue\to} 2{\boldsymbol{\red\to}} 2{\blue\to} 23{\boldsymbol{\red\to}} 23{\blue\to} 3{\boldsymbol{\red\to}} 3{\blue\to} 13{\boldsymbol{\red\to}} 1.
\]
This path corresponds to the product of idempotents
$
\rho_{12}\pi_2\pi_{23}\pi_3\lambda_{31},
$
which is $\gh$-related to the projection $\pi_1$; in fact, $\pi_1=(\rho_{12}\pi_2\pi_{23}\pi_3\lambda_{31})^2$.
It follows from Theorem~\ref{thm_PnMinusSnGeneral} that $\P_3\setminus\S_3=\la F\ra$.
This example shows that (ii) is not a necessary condition.
It is also easy to check that for any $f\in F$, the graph $\Gamma_n(F\setminus\{f\})$ does not satisfy condition (i).  It follows that $F$ is an irreducible generating set, even though it is not of the minimal size $\binom 42=6$.  This contrasts to the situation for the singular part $\T_n\setminus \S_n$ of the full transformation semigroup $\T_n$, where every idempotent generating set contains an idempotent generating set of minimal size \cite{DE2014}.  (The previous statement does not hold if ``idempotent generating set'' is replaced by ``generating set''.)
\end{example}

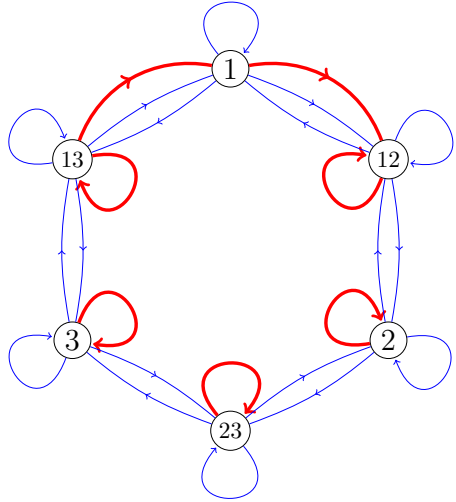
\begin{figure}
\begin{center}
\scalebox{0.8}
{
\begin{tikzpicture}[scale=3]
\tikzstyle{vertex}=[circle,draw=black, fill=white, inner sep = 0.07cm]
\node[vertex] (12) at (0.866025625	,0.499999617) {  $12$ };
\node[vertex] (1) at (0	        ,1) {\Large  $1$ };
\node[vertex] (13) at (-0.866024298	,0.5) {  $13$ };
\node[vertex] (3) at (-0.866026952	,-0.499997319) {\Large  $3$ };
\node[vertex] (23) at (0	        ,-1) {  $23$ };
\node[vertex] (2) at (0.866022971	,-0.500004213) {\Large  $2$ };
\draw [ultra thick, ->-=0.5, color=red] (1) to [bend left=40] (12);
\draw [ultra thick, ->-=0.5, color=red] (13) to  [bend left=40] (1);
\draw [ultra thick, ->, color=red] (12) edge [out=30+40+180,in=30-40+180,loop] ();
\draw [ultra thick, ->, color=red] (2) edge [out=30+40-60+180,in=30-40-60+180,loop] ();
\draw [ultra thick, ->, color=red] (23) edge [out=30+40-120+180,in=30-40-120+180,loop] ();
\draw [ultra thick, ->, color=red] (3) edge [out=30+40-180+180,in=30-40-180+180,loop] ();
\draw [ultra thick, ->, color=red] (13) edge [out=30+40-240+180,in=30-40-240+180,loop] ();
\draw [->, color=blue] (12) edge [out=30+40,in=30-40,loop] ();
\draw [->, color=blue] (2) edge [out=30+40-60,in=30-40-60,loop] ();
\draw [->, color=blue] (23) edge [out=30+40-120,in=30-40-120,loop] ();
\draw [->, color=blue] (3) edge [out=30+40-180,in=30-40-180,loop] ();
\draw [->, color=blue] (13) edge [out=30+40-240,in=30-40-240,loop] ();
\draw [->, color=blue] (1) edge [out=30+40-120+180,in=30-40-120+180,loop] ();
\draw [->-=0.5, color=blue] (1) to [bend left=10] (12);
\draw [->-=0.5, color=blue] (12) to  [bend left=10] (1);
\draw [->-=0.5, color=blue] (1) to [bend left=10] (13);
\draw [->-=0.5, color=blue] (13) to  [bend left=10] (1);
\draw [->-=0.5, color=blue] (3) to [bend left=10] (13);
\draw [->-=0.5, color=blue] (13) to  [bend left=10] (3);
\draw [->-=0.5, color=blue] (2) to [bend left=10] (12);
\draw [->-=0.5, color=blue] (12) to  [bend left=10] (2);
\draw [->-=0.5, color=blue] (3) to [bend left=10] (23);
\draw [->-=0.5, color=blue] (23) to  [bend left=10] (3);
\draw [->-=0.5, color=blue] (2) to [bend left=10] (23);
\draw [->-=0.5, color=blue] (23) to  [bend left=10] (2);
\end{tikzpicture}
}
\end{center}
\caption{
The digraph $\Gamma_3(F)$ where 
$
F=\{\pi_2,\pi_3,\pi_{12},\pi_{23},\pi_{13},\lambda_{31},\rho_{12}\}
$
in the partition monoid $\P_3$.
}
\label{Fig_G3F}
\end{figure}

We leave it as an open problem to determine a necessary and sufficient condition for $F$ to be a generating set, stated in terms of the structure of the subgraph of $\Gamma_n(F)$ determined by $F$ (that is, the subgraph determined by the red edges).

\section{The Brauer monoid}
\label{sec_Brauer}

Recall that the Brauer monoid $\B_n$ is the subsemigroup of $\P_n$ consisting of all partitions whose blocks have cardinality $2$.  See Figure \ref{fig_Bn} for an example.  Note that an element
\[
   \partn{A_i}{B_i}{C_j}{D_k}_{i \in I, \ j \in J, \ k \in K}
\]
of $\B_n$ satisfies $|A_i| = |B_i| = 1$ and $|C_j| = |D_k| = 2$ for all $i,j,k$.  In addition we must have $|J| = |K|$, and $|I|=n-2|J|$, which means that the ranks of elements of $\B_n$ are restricted to natural numbers of the form $n-2k$ where $k$ is a natural number.  Note also that $|\B_n|=(2n-1)!!=(2n-1)(2n-3)\cdots3\cdot1$.

\begin{figure}[ht]
\begin{center}
\begin{tikzpicture}[xscale=.7,yscale=0.7]
	\fill (0,0)circle(.1)
	      (1,0)circle(.1)
	      (2,0)circle(.1)
	      (3,0)circle(.1)
	      (4,0)circle(.1)
	      (5,0)circle(.1)
	      (6,0)circle(.1)
	      (0,2)circle(.1)
	      (1,2)circle(.1)
	      (2,2)circle(.1)
	      (3,2)circle(.1)
	      (4,2)circle(.1)
	      (5,2)circle(.1)
	      (6,2)circle(.1);
	\arcup24
	\catarc16{.7}
	\arcdn23
	\arcdn56
	\cve00
	\cve34
	\cve51
  \draw(0,2)node[above]{{\tiny $1$}};
  \draw(1,2)node[above]{{\tiny $2$}};
  \draw(2,2)node[above]{{\tiny $3$}};
  \draw(3,2)node[above]{{\tiny $4$}};
  \draw(4,2)node[above]{{\tiny $5$}};
  \draw(5,2)node[above]{{\tiny $6$}};
  \draw(6,2)node[above]{{\tiny $7$}};
	\end{tikzpicture}
\end{center}
\caption{An element of the Brauer monoid $\B_7$.}
\label{fig_Bn}
\end{figure}
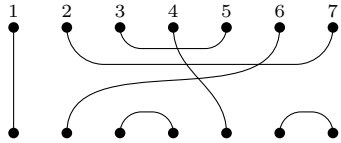

Since the symmetric group $\S_n$ is clearly contained in $\B_n$, it follows that $\S_n$ is the group of units of $\B_n$.  And, since $\B_n$ is closed under the $*$ operation, it is regular, so Green's relations on $\B_n$ are still given by the formulae in Theorem \ref{thm_wilcox}.  In fact, since the domain (respectively, codomain) of an element of $\B_n$ is determined by its kernel (respectively, cokernel), we have the following.  See also \cite[Theorem 7]{Maz1998} where an equivalent characterization is given.

\begin{theorem}
\label{thm_wilcox_new}
For each $\alpha, \beta \in \B_n$, we have\emph{:}
\begin{enumerate}
\item[(i)] $\alpha \gr \beta$ if and only if $\ker(\alpha)=\ker(\beta)$\emph{;}
\item[(ii)] $\alpha \gl \beta$ if and only if $\coker(\alpha)=\coker(\beta)$\emph{;} 
\item[(iii)] $\alpha \gj \beta$ if and only if $\rank(\alpha) = \rank(\beta)$.
\end{enumerate}
\end{theorem}

For $1\leq i<j\leq n$, we define the partition
\[
\tau_{ij} = \partn{x}{x}{i,j}{i,j}_{x\in[n]\setminus\{i,j\}}.
\]
See Figure \ref{fig_tau_ij} for an illustration.
\begin{figure}[ht]
\begin{center}
\begin{tikzpicture}[xscale=.7,yscale=0.7]
\draw(3,2)arc(180:270:.4) (3.4,1.6)--(6.6,1.6) (6.6,1.6) arc(270:360:.4);
\draw(3,0)arc(180:90:.4) (3.4,.4)--(6.6,.4) (6.6,.4) arc(90:0:.4);
	\fill (0,0)circle(.1)
	      (2,0)circle(.1)
	      (3,0)circle(.1)
	      (4,0)circle(.1)
	      (6,0)circle(.1)
	      (7,0)circle(.1)
	      (8,0)circle(.1)
	      (10,0)circle(.1)
	      (0,2)circle(.1)
	      (2,2)circle(.1)
	      (3,2)circle(.1)
	      (4,2)circle(.1)
	      (6,2)circle(.1)
	      (7,2)circle(.1)
	      (8,2)circle(.1)
	      (10,2)circle(.1);
	\draw (0,2)--(0,0)
	      (2,2)--(2,0)
%	      (3,2)--(3,0)
	      (4,2)--(4,0)
	      (6,2)--(6,0)
%	      (7,2)--(7,0)
	      (8,2)--(8,0)
	      (10,2)--(10,0);
  \draw [dotted] (0,2)--(2,2)
                 (4,2)--(6,2)
                 (8,2)--(10,2)
                 (0,0)--(2,0)
                 (4,0)--(6,0)
                 (8,0)--(10,0);
  \draw(3,2)node[above]{{\tiny $\phantom{j}i\phantom{j}$}};
  \draw(7,2)node[above]{{\tiny $\phantom{j}j\phantom{j}$}};
  \draw(0,2)node[above]{{\tiny $\phantom{j}1\phantom{j}$}};
  \draw(10,2)node[above]{{\tiny $\phantom{j}n\phantom{j}$}};
	\end{tikzpicture}
\end{center}
\caption{The projection $\tau_{ij}\in\B_n$.}
\label{fig_tau_ij}
\end{figure}
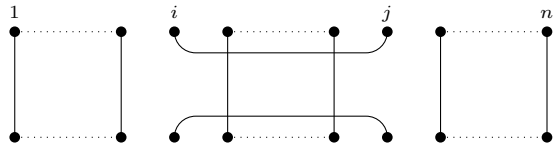
The following result was proved in \cite{Maltcev2007}. 

\begin{theorem} 
\label{lem_Brauer}
The singular part $\BnSn$ of the Brauer monoid $\B_n$ is idempotent generated.  The set $\set{\tau_{ij}}{1\leq i<j\leq n}$ is a minimal idempotent generating set.
\end{theorem}

In \cite{Maltcev2007}, the authors go on to give a presentation for $\BnSn$ with respect to the above generating set of projections, but we will not require this presentation here.  

As before, for $r=n-2k$ with $k\geq0$ and $0\leq r\leq n$, we write
\[
\JrBn=\set{\al\in\B_n}{\rank(\al)=r}=\JrPn\cap\B_n.
\]
By Theorem \ref{thm_wilcox}, these are precisely the $\gj$-classes of $\B_n$, and they form a chain:
\[
J_m(\B_n)<J_{m+2}(\B_n)<\cdots<\JnmBn<\JnBn,
\]
where $m$ denotes $0$ if $n$ is even, or $1$ otherwise.  By Proposition \ref{prop_chain}, it follows that the ideals of $\B_n$ are precisely the sets
\[
\IrBn = J_m(\B_n)\cup J_{m+2}(\B_n)\cup\cdots\cup \JrBn = \set{\al\in\B_n}{\rank(\al)\leq r} = \IrPn\cap\B_n.
\]
Note that $\InBn=\B_n$, $\JnBn=\S_n$, $\InmBn=\BnSn$, and that the maximal subgroups contained in $J_r(\B_n)$ are all isomorphic to $\S_r$.  In what follows, we will apply the general results of Sections \ref{sect_0-simple} and \ref{sect_regular-*} to the (finite idempotent generated regular~$*$-) semigroup $S=\InmBn=\B_n\setminus\S_n$.

\subsection{Rank and idempotent rank of ideals of $\boldmath \B_n$}

Again, the key step is to show that elements of small rank in $\B_n$ may be expressed as a product of higher rank elements.

\begin{lemma}\label{lem_dropdownBn}
If $0\leq r\leq n-4$, then $\JrBn \subseteq \lb \JrpBn \rb$. 
\end{lemma}

\begin{proof}
Let $\alpha \in \JrBn$ be a projection.  By Proposition \ref{prop_chain}, it suffices to show that $\al\in\la\JrpBn\ra$.  Write
\[
\alpha=\partn{A_i}{A_i}{C_j}{C_j}_{i\in I,\ j\in J}.
\]
Without loss of generality, we may assume that $J=[k]$, where $r=n-2k$.  For $j\in J$, let $C_j=\{a_j,b_j\}$, where $a_j<b_j$.  So $\al=\tau_{a_1b_1}\cdots\tau_{a_kb_k}$.  Then $\al=\be\ga$, where $\be=\tau_{a_1b_1}\cdots\tau_{a_{k-1}b_{k-1}}$ and $\ga=\tau_{a_2b_2}\cdots\tau_{a_kb_k}$
both belong to $\JrpBn$.
\end{proof}

\begin{theorem}
\label{thm_main3}
For $0\leq r=n-2k\leq n-2$, the ideal $\IrBn$ is idempotent generated, and
\[
\rank(\IrBn) = \idrank(\IrBn) = \binom n{2k}(2k-1)!!=\frac{n!}{2^k k! r!}.
\] 
Moreover, a subset $A \subseteq \IrBn$ of this cardinality is a generating set for $\IrBn$ if and only if the following three conditions hold\emph{:}
\begin{enumerate}
\item $\rank(\alpha) = r$ for all $\alpha \in A$\emph{;}
\item for all $\alpha, \beta \in A$ with $\alpha \neq \beta$, $\ker(\al)\not=\ker(\be)$\emph{;}
\item for all $\alpha, \beta \in A$ with $\alpha \neq \beta$, $\coker(\al)\not=\coker(\be)$.
\end{enumerate}
\end{theorem}

\begin{proof}
It follows from Lemma~\ref{lem_dropdownBn} and Proposition \ref{prop_chain} that $\IrBn$ is idempotent generated, and that $\rank(\IrBn) = \idrank(\IrBn) = \rho_{nr}$, where $\rho_{nr}$ is the number of $\gr$-classes in $\JrBn$.  It follows from Theorem~\ref{thm_wilcox_new} that $\rho_{nr}$ is equal to the number of equivalence relations on $[n]$ that have exactly $k$ non-trivial blocks, each of size $2$.  There are $\binom n{2k}$ ways to choose the elements belonging to the non-trivial blocks, and then $(2k-1)!!$ ways to choose the blocks of size $2$; multiplying these together gives the required formula.
The final clause follows by applying Theorem~\ref{thm_czsproj} in a virtually identical way to the argument in the final paragraph of the proof of Theorem \ref{thm_main1}.
\end{proof}

\begin{remark}
Note that the $r=n-2$ case agrees with Theorem \ref{lem_Brauer}.  Note also that $\rank(I_0(\B_n))=(n-1)!!$ if $n$ is even, while $\rank(I_1(\B_n))=n!!$ if $n$ is odd.
Some calculated values of $\rank(\IrBn)=\idrank(\IrBn)$ are given in Table~\ref{tab_rankIrBn}.
\end{remark}

\begin{remark}\label{rem:|IrBn|}
As in Remark \ref{rem:|IrPn|},  we may derive formulae for the sizes of the $\gj$-classes and ideals of $\B_n$:
\[
|\JrBn| = \rho_{nr}^2r! \qquad\text{and}\qquad  |\IrBn| = \sum_{i=0}^r \rho_{ni}^2i! \qquad\text{for any $0\leq r\leq n$},
\]
where the numbers $\rho_{nr}=\rank(\IrBn)$ are as in Theorem \ref{thm_main3}, and for convenience we define $\rho_{nn}=1$ and $\rho_{ni}=0$ if $i\not\equiv n \pmod 2$.
\end{remark}

\begin{table}[ht]%
\begin{tabular}{|c|ccccccccc|}
\hline
$n$ $\setminus$ $r$	&0	&1	&2	&3	&4	&5	&6	&7 &8      \\
\hline
 2&           1&            &            &            &            &            &            &            &                        \\
 3&            &           3&            &            &            &            &            &            &                        \\
 4&           3&            &           6&            &            &            &            &            &                        \\
 5&            &          15&            &          10&            &            &            &            &                        \\
 6&          15&            &          45&            &          15&            &            &            &                        \\
 7&            &         105&            &         105&            &          21&            &            &                        \\
 8&         105&            &         420&            &         210&            &          28&            &                        \\
 9&            &         945&            &        1260&            &         378&            &          36&                        \\
10&         945&            &        4725&            &        3150&            &         630&            &          45            \\
\hline
\end{tabular}
\caption{Values of $\rank(I_r(\B_n)) = \idrank(I_r(\B_n))$.}
\label{tab_rankIrBn}
\end{table}

\subsection{Minimal idempotent generating sets of $\boldmath \BnSn$}

As in Section \ref{sect_PnSn}, an enumeration of the minimal idempotent generating sets of $\BnSn$ amounts to an enumeration of the balanced subgraphs of the projection graph $\Gamma(\BnSn)=\Gamma(\JnmBn^*)$, which we will denote by $\Lambda_n$.

For distinct $i,j,k\in[n]$, we define
\[
\si_{ijk} = \partnlong{x}{x}{k}{i}{i,j}{j,k}_{x\in[n]\setminus\{i,j,k\}}.
\]
See Figure \ref{fig_si_ijk} for an illustration where, for convenience, we have only pictured the restriction of $\si_{ijk}$ to $\{i,j,k\}$.  It is easy to check that these partitions are idempotents.
\begin{figure}[ht]
\begin{center}
\begin{tikzpicture}[xscale=.7,yscale=0.7]
\siijk012
\arcup01
\arcdn12
\end{tikzpicture}
\begin{tikzpicture}[xscale=.7,yscale=0.7]
\siijk102
\arcup01
\arcdn02
\end{tikzpicture}
\begin{tikzpicture}[xscale=.7,yscale=0.7]
\siijk120
\arcup12
\arcdn02
\end{tikzpicture}
\\
\
\\
\begin{tikzpicture}[xscale=.7,yscale=0.7]
\siijk021
\arcup02
\arcdn12
\end{tikzpicture}
\begin{tikzpicture}[xscale=.7,yscale=0.7]
\siijk201
\arcup02
\arcdn01
\end{tikzpicture}
\begin{tikzpicture}[xscale=.7,yscale=0.7]
\siijk210
\arcup12
\arcdn01
\end{tikzpicture}
\end{center}
\caption{Simplified illustrations of the idempotents $\si_{ijk}\in\B_n$ for all possible orderings of $i,j,k$.  See text for more details.}
\label{fig_si_ijk}
\end{figure}
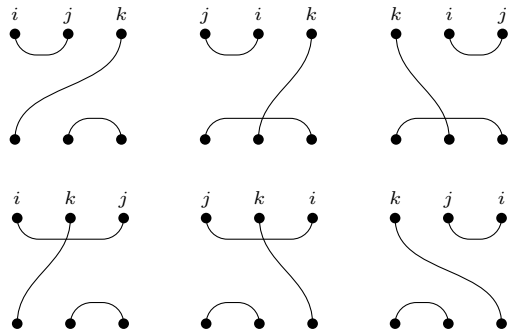

The next lemma is verified in similar fashion to Lemma~\ref{lem_idempotents_in_J_{n-1}}, and its proof is omitted.  For simplicity, we will use symmetric notation and allow ourselves to write $\tau_{ji}=\tau_{ij}$ for all $\oijn$.

\begin{lemma}\label{lem_idempotents_in_J_{n-2}_Bn}
The set of idempotents of $\JnmBn$ is
\[
\set{\tau_{ij}}{\oijn}\cup\set{\si_{ijk}}{\text{$i,j,k\in[n]$ \emph{distinct}}}.
\]
The set of projections of $\JnmBn$ is
\[
\set{\tau_{ij}}{\oijn}.
\]
In the principal factor $\JnmBn^*$, the only nonzero products of pairs of projections are, using symmetric notation for the projections,
\[
\tau_{ij}^2=\tau_{ij}, \ \ \tau_{ij}\tau_{jk}=\si_{ijk}.
\]
\end{lemma}

It follows that the graph $\Lambda_n$ has vertex set $\set{\tau_{ij}}{\oijn}$ with edges $\tau_{ij}\to\tau_{kl}$ if and only if $\{i,j\}\cap\{k,l\}\not=\varnothing$.  The graphs $\Lambda_4$ and $\Lambda_5$ are pictured in Figure \ref{fig_Lambda_5}, where we have simplified matters by labelling the vertices $ij$ instead of $\tau_{ij}$, omitting loops, and displaying pairs of directed edges $ij\leftrightarrows kl$ as single undirected edges $ij-kl$.
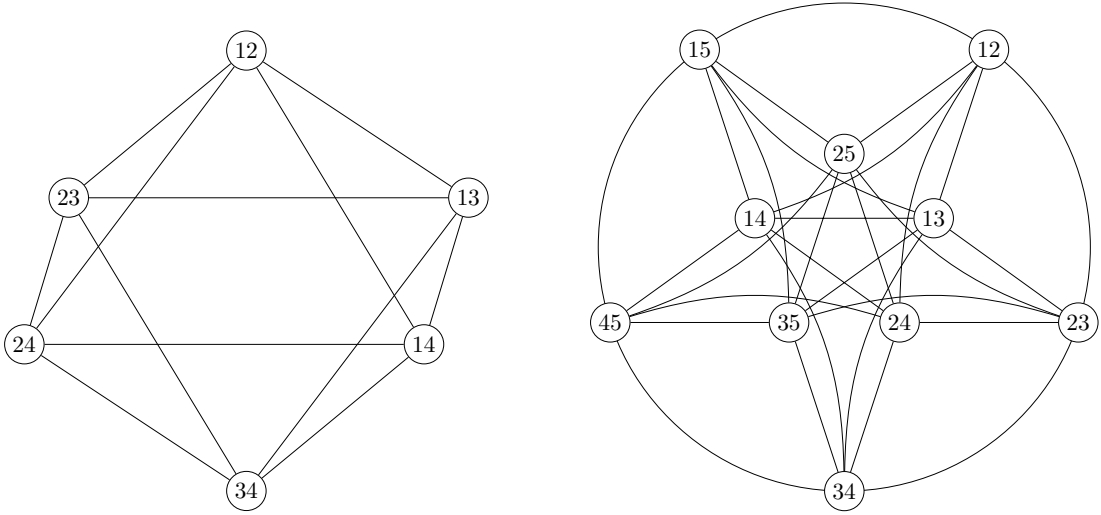
\begin{figure}[ht]
\begin{center}
\scalebox{0.8}
{
\begin{tikzpicture}[scale=7.3]
\tikzstyle{vertex}=[circle,draw=black, fill=white, inner sep = 0.07cm]
\node[vertex] (12) at (0	,1) {$12$};
\node[vertex] (13) at (.5	,.666) {$13$};
\node[vertex] (14) at (0.4	,.333) {$14$};
\node[vertex] (23) at (-0.4	,.666) {$23$};
\node[vertex] (24) at (-.5	,.333) {$24$};
\node[vertex] (34) at (0	,0) {$34$};
\draw (12)--(13);
\draw (12)--(23);
\draw (12)--(24);
\draw (12)--(14);
\draw (34)--(13);
\draw (34)--(14);
\draw (34)--(23);
\draw (34)--(24);
\draw (23)--(13);
\draw (23)--(24);
\draw (24)--(14);
\draw (14)--(13);
\end{tikzpicture}
\qquad\qquad
\begin{tikzpicture}[scale=5]
\tikzstyle{vertex}=[circle,draw=black, fill=white, inner sep = 0.07cm]
\draw (0,0) circle (0.809);
\node[vertex] (12) at (0.475529119	,0.654505847) {  $12$ };
\node[vertex] (13) at (0.293890801	,0.095490177) {  $13$ };
\node[vertex] (14) at (-0.293894022	,0.095492517) {  $14$ };
\node[vertex] (15) at (-0.475527889	,0.654509633) {  $15$ };
\node[vertex] (23) at (0.769419921	,-0.250003976) {  $23$ };
\node[vertex] (24) at (0.181635098	,-0.250001636) {  $24$ };
\node[vertex] (25) at (0	,0.30901548) {  $25$ };
\node[vertex] (34) at (0	,-0.809017306) {  $34$ };
\node[vertex] (35) at (-0.181637088	,-0.25000019) {  $35$ };
\node[vertex] (45) at (-0.769421911	,-0.24999785) {  $45$ };
\draw (12)--(13)--(23)--(24)--(34)--(35)--(45)--(14)--(15)--(25)--(12);
\draw (25)--(24)--(14)--(13)--(35)--(25);
\draw (12) to [bend left=17] (14);
\draw (12) to [bend right=17] (24);
\draw (23) to [bend left=17] (25);
\draw (23) to [bend right=17] (35);
\draw (34) to [bend left=17] (13);
\draw (34) to [bend right=17] (14);
\draw (45) to [bend left=17] (24);
\draw (45) to [bend right=17] (25);
\draw (15) to [bend left=17] (35);
\draw (15) to [bend right=17] (13);
\end{tikzpicture}
}
\end{center}
\caption{Simplified illustrations of the graphs $\Lambda_4=\Gamma(\B_4\setminus\S_4)$ and $\Lambda_5=\Gamma(\B_5\setminus\S_5)$; see text for further details.}
\label{fig_Lambda_5}
\end{figure}

Recall that the \emph{Johnson graph} $J(n,k)$ is the graph with vertex set $\{A\subseteq[n]:|A|=k\}$, and with edges $A - B$ if and only if $|A \cap B| = k-1$. In particular $J(n,2)$ has vertex set $\{A\subseteq[n]:|A|=2\}$ and edges $A - B$ if and only if $A$ and $B$ overlap in precisely one element. More background on Johnson graphs in the context of algebraic graph theory may be found in \cite[Chapter~1.6]{GodsilRoyle}. Note that the underlying undirected loop-free graph of $\Lambda_n=\Gamma(\B_n\setminus\S_n)$ is isomorphic to $J(n,2)$. So, in fact, Figure~\ref{fig_Lambda_5} pictures the Johnson graphs $J(4,2)$ and $J(5,2)$.

Factorizations of Johnson graphs have been considered, for instance, in \cite{Praeger2008}. Recall that a $1$-factor of a graph is a collection of edges that spans the graph, while a $2$-factor is a collection of cycles that spans all vertices of the graph. We define a $(0,1,2)$-factor of a graph as a decomposition of the graph into vertices, edges, and cycles, such that each vertex is contained in precisely one of these pieces. An oriented $(0,1,2)$-factor is a $(0,1,2)$-factor such that all the cycles are assigned an orientation (clockwise or anticlockwise).

Note that there is a one-one correspondence between balanced subgraphs of $\Lam_n$ and oriented $(0,1,2)$-factors of $J(n,2)$.  Indeed, consider a balanced subgraph $H$ of $\Lam_n$.  The connected components of $H$ come in three forms: 
\begin{itemize}
\item[(i)] a component of size $1$: $\{i,j\}$\!\!\!\!\begin{tikzcd}\arrow[loop right]{l}{}\end{tikzcd}\!; 
\item[(ii)] a component of size $2$: $\{i,j\}\leftrightarrows\{j,k\}$; or 
\item[(iii)] a component of size $k\geq3$: $\{i_1,i_2\}\to\{i_2,i_3\}\to\cdots\to\{i_{k-1},i_k\}\to\{i_1,i_2\}$.  
\end{itemize}
These three types of components give rise (respectively) to the vertices, edges and oriented cycles in a corresponding $(0,1,2)$-factor of $J(n,2)$.  Combined with Theorem~\ref{thm_balanced}, which tells us that the minimal idempotent generating sets of $\B_n\setminus\S_n$ are in one-one correspondence with the balanced subgraphs of $\Lambda_n=\Gamma(\B_n\setminus\S_n)$, this proves the following.

\begin{proposition}\label{012factorsJn2}
There is a one-one correspondence between the minimal idempotent generating sets of $\B_n\setminus\S_n$ and the oriented $(0,1,2)$-factors of the Johnson graph $J(n,2)$. 
\end{proposition}
%\begin{proof}
%This is an immediate consequence of Theorem~\ref{thm_balanced}, the above-mentioned connection between $\Lambda_n=\Gamma(\B_n\setminus\S_n)$ and $J(n,2)$, and the associated correspondence between balanced subgraphs of $\Lambda_n$ and oriented $(0,1,2)$-factors of $J(n,2)$.
%\end{proof}

\begin{remark}\label{rem:dn}
We do not know of any formula or recurrence relation for the number $d_n$ of minimal idempotent generating sets of $\B_n\setminus\S_n$. However, the sequence $d_n$ grows rapidly, as we now explain.  By Proposition \ref{012factorsJn2}, $d_n$ is equal to the number of $(0,1,2)$-factors of the Johnson graph $J(n,2)$. 
Consider the graph $J(n+1,2)$. We decompose the vertex set of $J(n+1,2)$ into two subsets: 
\[
V_1=\bigset{\{i,n+1\}}{1\leq i\leq n} \qquad\text{and}\qquad V_2=\bigset{\{i,j\}}{\oijn}.
\]
Let $G_1$ and $G_2$ be the induced subgraphs of $\Lam_{n+1}$ with these vertex sets (respectively).  Any oriented $(0,1,2)$-factor of $G_1$ and any oriented $(0,1,2)$-factor of $G_2$ may be pieced together to give an oriented $(0,1,2)$-factor of $J(n+1,2)$.  But $G_1$ is the complete graph on vertex set $V_1$, and $G_2$ is in fact the Johnson graph $J(n,2)$.  Since there are $n!$ oriented $(0,1,2)$-factors of $G_1$ (corresponding to permutations of the vertex set), and $d_n$ oriented $(0,1,2)$-factors of $G_2$, it follows that $d_{n+1}\geq n!\times d_n$.
%The $2$-sets containing the number $n+1$ induce a copy of the complete graph $K_n$ in the graph, and the number of $(0,1,2)$-factors of this subgraph is $n!$, and the remaining vertices induce a copy of the graph $J(n,2)$. It follows that
%$
%d_{n+1} \geq n! d_n,
%$
This then gives
\[
d_n \geq \prod_{i=1}^{n-1}i! = \prod_{i=1}^{n-1}i^{n-i}.
\]
\end{remark}

\begin{remark}
The first few values of $d_n$ were calculated by James Mitchell using the \textsf{Semigroups} package in \textsf{GAP}; see \cite{East2013_4,Mitchell2013}.  These values are given in Table~\ref{tab_cndnen}, along with the sequence $c_n=\prod_{i=1}^{n-1}i!$.  
The calculation of $d_3=6$ is trivial, the $(0,1,2)$-factors of $J(3,2)$ being in one-one correspondence with the permutations of a three element set.  To see that $d_4=265$, consider the graph $J(4,2)$, depicted as an octahedron in Figure~\ref{fig_Lambda_5}.  Labelling the vertices $a,b,c,A,B,C$ in such a way that vertex $x$ is opposite vertex $X$ for each $x\in\{a,b,c\}$, we see that the $(0,1,2)$-factors of $J(4,2)$ are in one-one correspondence with the permutations of $\{a,b,c,A,B,C\}$ such that no lower-case letter is mapped to its corresponding upper-case letter, and vice versa.  These permutations are obviously in one-one correspondence with the fixed point free permutations of the set $[6] = \{1, 2, 3, 4, 5, 6\}$, of which there are $265$.  In general, for $n\geq4$, $d_n$ is bounded above by the number $e_n$ of fixed point free permutations of a set of size $\binom n2$; this sequence is also given in Table~\ref{tab_cndnen}.  The sequences $c_n$, $d_n$ and $e_n$ are A000178, A244493 and (a subsequence of) A000166, respectively, on \cite{Sloan:OEIS}; the sequence $d_n$ was added to \cite{Sloan:OEIS} after this article was written.
\end{remark}

\begin{table}[ht]%
\begin{center}
\begin{tabular}{|c|cccccc|}
\hline
$n$ & $2$ & $3$ & $4$ & $5$ & $6$ & $7$ \\
\hline\hline
$c_n$  & $1$ &  $2$ &  $12$  & $288$  & $34560$  & $24883200$   \\
\hline
$d_n$  & $1$ &  $6$ &  $265$  & $126140$  & $855966411$  & ?   \\
\hline
$e_n$  & $0$ &  $2$ &  $265$  & $1334961$  & $481066515734$  & $895014631192902121$   \\
\hline\hline
$d_n/c_n$  & $1$ &  $3$ &  $\approx22\phantom{\approx}$  & $\approx438\phantom{\approx}$  & $\approx24768\phantom{\approx}$  & ?   \\
\hline
$e_n/d_n$  & $0$ &  $1/3$ & $1$ & $\approx11\phantom{\approx}$  & $\approx562\phantom{\approx}$  &  ?   \\
\hline
\end{tabular}
\end{center}
\caption{The sequences $c_n,d_n,e_n$.  For $n\geq2$, $d_n$ is equal to the number of minimal idempotent generating sets of $\BnSn$.}
\label{tab_cndnen}
\end{table}

\subsection*{The Brauer monoid and Pfaffian orientations}

We have established a correspondence between minimal idempotent generating sets of the singular part of the Brauer monoid, and certain factorizations of the Johnson graph $J(n,2)$. As we mentioned above, we do not know of a formula or recurrence relation that gives the number of such factorizations. The following result shows that one cannot find such a formula by trying to compute a Pfaffian orientation for the corresponding Graham--Houghton graph. 

A subgraph $H$ of a graph $G$ is called \emph{central} if $G \setminus H$ has a perfect matching (here,~$\setminus$ stands for deletion, where we remove the vertices from $H$ as well as any edges involving one or more vertices from $H$). An even circuit $C$ in a directed graph $D$ is called \emph{oddly oriented} if for either choice of direction of traversal around $C$, the number of edges of $C$ directed in the direction of traversal is odd. This is independent of the initial choice of direction of traversal, since $C$ is even. An orientation $D$ of the edges of a graph $G$ is \emph{Pfaffian} if every even central circuit of G is oddly oriented in $D$. We say that a graph $G$ is Pfaffian if it has a Pfaffian orientation. The significance of Pfaffian orientations comes from the fact that if a bipartite graph $G$ has one, then the number of perfect matchings of $G$ can be computed in polynomial time. More on Pfaffian orientations may be found in \cite{Little1975, Robertson1999, ThomasSurvey2006, Vazirani1989}. 

These ideas are relevant to us, since the number of distinct minimal idempotent generating sets is precisely the number of perfect matchings of the (bipartite) Graham--Houghton graph. So, if the Graham--Houghton graph in question did have a Pfaffian orientation this would mean that the number of distinct minimal idempotent generating sets could be computed in polynomial time. The following result shows that we cannot use the theory of Pfaffian orientations as an approach to computing the number of minimal generating sets for $\B_n \setminus \S_n$.

\begin{proposition}\label{prop_nonPfaf}
Let $\Delta_n$ be the unlabeled undirected Graham--Houghton graph of the $\gj$-class $\JnmBn$. If $n \geq 3$, then $\Delta_n$ does not admit a Pfaffian orientation.  
\end{proposition}
\begin{proof}
When $n=3$, $\Delta_3$ is isomorphic to the complete bipartite graph $K_{3,3}$, which is known not to admit a Pfaffian orientation; see \cite[Theorem~4.1]{ThomasSurvey2006}. Now we shall apply results from \cite{ThomasSurvey2006} to show that $\Delta_n$ does not admit a Pfaffian orientation for any $n \geq 4$. 

A digraph $D$ is called \emph{even} if whenever the arcs of $D$ are assigned weights $0$ or $1$, $D$ contains a cycle of even total weight. Now let $G = A \sqcup B$ be a bipartite graph and $M$ a matching of $G$. Let $D = D(G,M)$ be the digraph obtained from $G$ by (i) orienting arcs from $A$ to $B$, and (ii) contracting every edge of $M$. Little \cite{Little1975} showed that a bipartite graph $G$ is Pfaffian if and only if $D(G,M)$ is \emph{not} even. 

Now, carrying out this process with the bipartite graph $\Delta_n$ together with the natural choice of matching given by the set of projections, we obtain a directed version of the Johnson graph $J(n,2)$; the graph obtained is essentially the graph $\Lambda_n=\Gamma(\BnSn)$, defined above, but with the loops at each vertex removed. 
We claim that the resulting digraph, which we will denote by $\Lambda_n'$, is even. Once established, from the results above, this will complete the proof that $\Delta_n$ is not Pfaffian. 
To see that $\Lambda_n'$ is even, consider an assignment of labels $\{0,1\}$ to the arcs of this digraph. Since $\Lambda_n'$ embeds triangles it follows that at least one of the four configurations displayed in Figure~\ref{Fig_4config} must arise. By inspection it is now clear that in each case $\Lambda_n'$ must contain a cycle of even total weight.  
\end{proof}

\begin{figure}
\begin{center}
\begin{tabular}{ccccccc}
% 1
\begin{tikzpicture}[scale=1]
\tikzstyle{vertex}=[circle,draw=black, fill=white, inner sep = 0.06cm]
\node[vertex] (A) at (0,0) {};
\node[vertex] (B) at (2,0) {};
\draw [->-=0.5] (A)        to [bend right=40] (B);
\draw [->-=0.5] (B)        to [bend right=40] (A);
\end{tikzpicture}
& &
% 2
\begin{tikzpicture}[scale=1]
\tikzstyle{vertex}=[circle,draw=black, fill=white, inner sep = 0.06cm]
\node[vertex] (A) at (0,0) {};
\node[vertex] (B) at (2,0) {};
\draw [->-=0.5, dashed] (A)        to [bend right=40] (B);
\draw [->-=0.5, dashed] (B)        to [bend right=40] (A);
\end{tikzpicture}
& &
% 3
\begin{tikzpicture}[scale=1]
\tikzstyle{vertex}=[circle,draw=black, fill=white, inner sep = 0.06cm]
\node[vertex] (A) at (0,0) {};
\node[vertex] (B) at (2,0) {};
\node[vertex] (C) at (1,1.732) {};
\draw [->-=0.5] (A)        to [bend right=20] (B);
\draw [->-=0.5] (B)        to [bend right=20] (C);
\draw [->-=0.5] (C)        to [bend right=20] (A);
\draw [->-=0.5, dashed] (B)        to [bend right=20] (A);
\draw [->-=0.5, dashed] (A)        to [bend right=20] (C);
\draw [->-=0.5, dashed] (C)        to [bend right=20] (B);
\end{tikzpicture}
& &
% 4
\begin{tikzpicture}[scale=1]
\tikzstyle{vertex}=[circle,draw=black, fill=white, inner sep = 0.06cm]
\node[vertex] (A) at (0,0) {};
\node[vertex] (B) at (2,0) {};
\node[vertex] (C) at (1,1.732) {};
\draw [->-=0.5, dashed] (A)        to [bend right=20] (B);
\draw [->-=0.5] (B)        to [bend right=20] (C);
\draw [->-=0.5] (C)        to [bend right=20] (A);
\draw [->-=0.5] (B)        to [bend right=20] (A);
\draw [->-=0.5, dashed] (A)        to [bend right=20] (C);
\draw [->-=0.5, dashed] (C)        to [bend right=20] (B);
\end{tikzpicture}
\end{tabular}
\end{center}
\caption{
Four configurations, one of which must arise in  $\Lambda_n'$ after assigning labels $\{0,1\}$ to the arcs of the digraph $\Lambda_n'$. Here dotted arcs are arcs labeled by $0$, and the solid represent those labeled by $1$.}
\label{Fig_4config}
\end{figure}
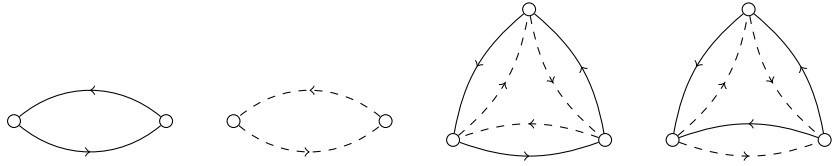

\subsection{Arbitrary idempotent generating sets for $\boldmath \BnSn$.}

For a subset $F$ of
\[
\{\al\in E(\B_n):\rank(\al)=n-2\} = \set{\tau_{ij}}{\oijn}\cup\set{\si_{ijk}}{\text{$i,j,k\in[n]$ distinct}},
\]
let $\Lam_n(F)$ be the two-coloured digraph obtained by colouring each edge of $\Lam_n$ blue, and then adding red edges corresponding to the idempotents from $F$:
\begin{itemize}
	\item $ij{\boldsymbol{\red\to}}ij$ if $\tau_{ij}\in F$,
	\item $ij{\boldsymbol{\red\to}}jk$ if $\si_{ijk}\in F$.
\end{itemize}
(As above, we will denote the vertices of $\Lam_n$ by $ij$ rather than $\tau_{ij}$.)

\begin{theorem}\label{arb_gen_set_Bn}
For $F\subseteq \{\al\in E(\B_n):\rank(\al)=n-2\}$, the following are equivalent\emph{:}
\begin{enumerate}
	\item[(i)] $\BnSn=\la F\ra$\emph{;}
	\item[(ii)] each vertex of $\Lam_n(F)$ is the base point of an RBR-alternating circuit. 
\end{enumerate}
\end{theorem}

\newcommand{\LnF}{\Lam_n(F)}
\newcommand{\rto}{{\boldsymbol{\red\to}}}
\newcommand{\bto}{{\blue\to}}

The authors are not aware of any formula for the number of subsets $$F\subseteq \{\al\in E(\B_n):\rank(\al)=n-2\}$$ that generate $\BnSn$.

\begin{remark}\label{rem810}
Neither condition (i) nor condition (ii) of Remark~\ref{rem_generalRBR} gives a necessary and sufficient condition for $F$ to be a generating set in the case of $\BnSn$. 

Condition (i) of Remark~\ref{rem_generalRBR}, though necessary, is not sufficient. 
For example, consider the graph $\Lambda_4(F)$ where $\Lambda_4 = \Gamma(\B_4 \setminus \S_4)$ 
and
\[
F = \{  
\tau_{14}, 
\tau_{23}, 
\sigma_{214}, 
\sigma_{231}, 
\sigma_{234}, 
\sigma_{241}, 
\sigma_{314}, 
\sigma_{321}, 
\sigma_{324},
\sigma_{341}
\}. 
\]
This graph is illustrated in Figure~\ref{fig_Lambda_Fails} (with blue edges simplified as in Figure \ref{fig_Lambda_5}). In this example, each vertex of $\Lam_4(F)$ has at least one red edge coming in to it and at least one going out of it. However, the vertex $12$ is not the base point of any RBR-alternating circuit, and therefore $F$ is not a generating set for $\B_4 \setminus \S_4$. 

Condition (ii) of Remark~\ref{rem_generalRBR}, though sufficient, is not necessary.
For example, consider the graph $\Lam_3(F)$, where $F=\{\tau_{13},\tau_{23},\si_{213},\si_{321}\}$, which is illustrated in Figure~\ref{Fig_Lam3F}. It is easy to check that each vertex is the base point of an RBR-alternating circuit, so that $\B_3\setminus\S_3=\la F\ra$, even though vertex $12$ is not contained in a red circuit.  
\end{remark}

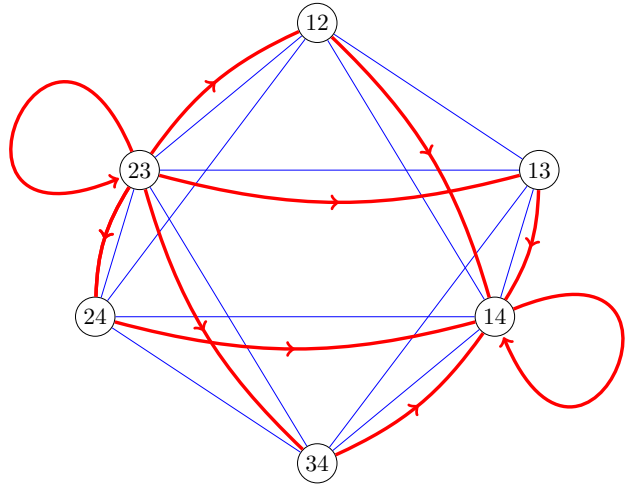
\begin{figure}
\begin{center}
\scalebox{0.8}
{
\begin{tikzpicture}[scale=7.3]
\tikzstyle{vertex}=[circle,draw=black, fill=white, inner sep = 0.07cm]
\node[vertex] (12) at (0	,1) {$12$};
\node[vertex] (13) at (.5	,.666) {$13$};
\node[vertex] (14) at (0.4	,.333) {$14$};
\node[vertex] (23) at (-0.4	,.666) {$23$};
\node[vertex] (24) at (-.5	,.333) {$24$};
\node[vertex] (34) at (0	,0) {$34$};
\draw [color=blue] (12)--(13);
\draw [color=blue] (12)--(23);
\draw [color=blue] (12)--(24);
\draw [color=blue] (12)--(14);
\draw [color=blue] (34)--(13);
\draw [color=blue] (34)--(14);
\draw [color=blue] (34)--(23);
\draw [color=blue] (34)--(24);
\draw [color=blue] (23)--(13);
\draw [color=blue] (23)--(24);
\draw [color=blue] (24)--(14);
\draw [color=blue] (14)--(13);
\draw [ultra thick, ->-=0.5, color=red] (23)        to [bend left=15] (12);
\draw [ultra thick, ->-=0.5, color=red] (23)        to [bend right=15] (24);
\draw [ultra thick, ->-=0.5, color=red] (23)        to [bend right=15] (13);
\draw [ultra thick, ->-=0.5, color=red] (23)        to [bend right=15] (34);
\draw [ultra thick, ->-=0.5, color=red] (24)        to [bend right=15] (14);
\draw [ultra thick, ->-=0.5, color=red] (34)        to [bend right=15] (14);
\draw [ultra thick, ->-=0.5, color=red] (12)        to [bend left=15] (14);
\draw [ultra thick, ->-=0.5, color=red] (23)        to [bend right=15] (24);
\draw [ultra thick, ->-=0.5, color=red] (13)        to [bend left=15] (14);
\draw [ultra thick, ->, color=red] (23) edge [out=90+22.5,in=90+90+22.5,loop] ();
\draw [ultra thick, ->, color=red] (14) edge [out=0+22.5,in=-90+22.5,loop] ();
\end{tikzpicture}
}
\end{center}
\vspace{-8mm}
\caption{
The graph $\Lambda_4(F)$, where 
$$
F = \{  
\tau_{14}, 
\tau_{23}, 
\sigma_{214}, 
\sigma_{231}, 
\sigma_{234}, 
\sigma_{241}, 
\sigma_{314}, 
\sigma_{321}, 
\sigma_{324},
\sigma_{341}
\}. 
$$
}
\label{fig_Lambda_Fails}
\end{figure}

\begin{figure}
\begin{center}
\begin{tikzpicture}[scale=1]
\tikzstyle{vertex}=[circle,draw=black, fill=white, inner sep = 0.06cm]
\node[vertex] (A) at (0,0) {$23$};
\node[vertex] (B) at (2,0) {$13$};
\node[vertex] (C) at (1,1.732) {$12$};
\draw [blue,->-=0.5] (A)        to [bend left=10] (B);
\draw [blue,->-=0.5] (B)        to [bend left=10] (C);
\draw [blue,->-=0.5] (C)        to [bend left=10] (A);
\draw [blue,->-=0.5] (B)        to [bend left=10] (A);
\draw [blue,->-=0.5] (A)        to [bend left=10] (C);
\draw [blue,->-=0.5] (C)        to [bend left=10] (B);
\draw [ultra thick, red,->-=0.5] (A)        to [bend left=30] (C);
\draw [ultra thick, red,->-=0.5] (C)        to [bend left=30] (B);
\draw [blue,->] (C) edge [out=130,in=50,loop] ();
\draw [blue,->] (B) edge [out=130+240,in=50+240,loop] ();
\draw [blue,->] (A) edge [out=130+120,in=50+120,loop] ();
\path (B) edge[ultra thick, red, out=140+240, in=40+240
                , looseness=0.8, loop
                , distance=2.5cm, ->]
            node[above=3pt] {} (B);
\path (A) edge[ultra thick, red, out=140+120, in=40+120
                , looseness=0.8, loop
                , distance=2.5cm, ->]
            node[above=3pt] {} (A);
\end{tikzpicture}
\end{center}
\vspace{-15mm}
\caption{The graph $\Lam_3(F)$, where $F=\{\tau_{13},\tau_{23},\si_{213},\si_{321}\}$.}
\label{Fig_Lam3F}
\end{figure}
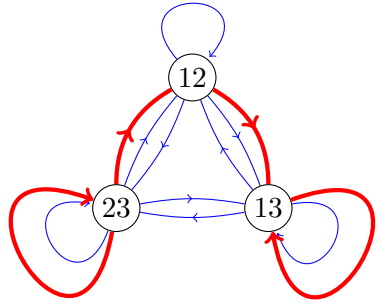

\section{The Jones monoid}\label{sec_Jones}

Recall that the Jones monoid $\J_n$ is the subsemigroup of $\B_n$ consisting of all partitions whose blocks have cardinality $2$ and may be drawn in a planar fashion.  By stretching the diagram of an element of $\J_n$ so that the vertices appear in a single straight line $1,2,\ldots,n,n',\ldots,2',1'$, it is straightforward to verify that the elements of $\J_n$ are in one-one correspondence with the proper bracketings with $n$ pairs of brackets; see Figure \ref{fig_Jn} for an example.  It follows that $|\J_n|=C_n$, where $C_n=\frac{1}{n+1}\binom {2n}n$ is the $n$th Catalan number.  The only planar permutation is the identity element of $\P_n$, denoted $1$, so the group of units of $\J_n$ is equal to~$\{1\}$.  But $\J_n\subseteq\B_n$ is again closed under the $*$ operation, so Green's relations are still described by Theorem \ref{thm_wilcox_new}.  Note, however, that the planarity condition implies that $\J_n$ is $\gh$-trivial.

\begin{figure}[ht]
\begin{center}
\begin{tikzpicture}[xscale=.7,yscale=0.7]
	\fill (0,0)circle(.1)
	      (1,0)circle(.1)
	      (2,0)circle(.1)
	      (3,0)circle(.1)
	      (4,0)circle(.1)
	      (5,0)circle(.1)
	      (6,0)circle(.1)
	      (0,2)circle(.1)
	      (1,2)circle(.1)
	      (2,2)circle(.1)
	      (3,2)circle(.1)
	      (4,2)circle(.1)
	      (5,2)circle(.1)
	      (6,2)circle(.1);
  \draw[white] (13,2)circle(.1);
	\catarc25{.6}
	\catarc34{.3}
	\arcdn12
	\arcdn56
	\cve00
	\cve13
	\cve64
  \draw(0,2)node[above]{{\tiny $1$}};
  \draw(1,2)node[above]{{\tiny $2$}};
  \draw(2,2)node[above]{{\tiny $3$}};
  \draw(3,2)node[above]{{\tiny $4$}};
  \draw(4,2)node[above]{{\tiny $5$}};
  \draw(5,2)node[above]{{\tiny $6$}};
  \draw(6,2)node[above]{{\tiny $7$}};
  \draw(0,0)node[below]{{\tiny $\phantom{'}1'$}};
  \draw(1,0)node[below]{{\tiny $\phantom{'}2'$}};
  \draw(2,0)node[below]{{\tiny $\phantom{'}3'$}};
  \draw(3,0)node[below]{{\tiny $\phantom{'}4'$}};
  \draw(4,0)node[below]{{\tiny $\phantom{'}5'$}};
  \draw(5,0)node[below]{{\tiny $\phantom{'}6'$}};
  \draw(6,0)node[below]{{\tiny $\phantom{'}7'$}};
  \draw(13,0)node[below]{{\white $)$}};
	\end{tikzpicture}
	\\ \ \\ %\ \\
\begin{tikzpicture}[xscale=.7,yscale=0.7]
	\fill (0,2)circle(.1)
	      (1,2)circle(.1)
	      (2,2)circle(.1)
	      (3,2)circle(.1)
	      (4,2)circle(.1)
	      (5,2)circle(.1)
	      (6,2)circle(.1)
	      (7,2)circle(.1)
	      (8,2)circle(.1)
	      (9,2)circle(.1)
	      (10,2)circle(.1)
	      (11,2)circle(.1)
	      (12,2)circle(.1)
	      (13,2)circle(.1);
	\catarc0{13}{1.2}
	\catarc1{10}{.9}
	\catarc25{.6}
	\catarc34{.3}
	\catarc69{.6}
	\catarc78{.3}
	\catarc{11}{12}{.3}
  \draw(0,2)node[above]{{\tiny $1$}};
  \draw(1,2)node[above]{{\tiny $2$}};
  \draw(2,2)node[above]{{\tiny $3$}};
  \draw(3,2)node[above]{{\tiny $4$}};
  \draw(4,2)node[above]{{\tiny $5$}};
  \draw(5,2)node[above]{{\tiny $6$}};
  \draw(6,2)node[above]{{\tiny $7$}};
  \draw(13,2)node[above]{{\tiny $\phantom{'}1'$}};
  \draw(12,2)node[above]{{\tiny $\phantom{'}2'$}};
  \draw(11,2)node[above]{{\tiny $\phantom{'}3'$}};
  \draw(10,2)node[above]{{\tiny $\phantom{'}4'$}};
  \draw(9,2)node[above]{{\tiny $\phantom{'}5'$}};
  \draw(8,2)node[above]{{\tiny $\phantom{'}6'$}};
  \draw(7,2)node[above]{{\tiny $\phantom{'}7'$}};
  \draw(0,0)node[below]{{ $($}};
  \draw(1,0)node[below]{{ $($}};
  \draw(2,0)node[below]{{ $($}};
  \draw(3,0)node[below]{{ $($}};
  \draw(4,0)node[below]{{ $)$}};
  \draw(5,0)node[below]{{ $)$}};
  \draw(6,0)node[below]{{ $($}};
  \draw(7,0)node[below]{{ $($}};
  \draw(8,0)node[below]{{ $)$}};
  \draw(9,0)node[below]{{ $)$}};
  \draw(10,0)node[below]{{ $)$}};
  \draw(11,0)node[below]{{ $($}};
  \draw(12,0)node[below]{{ $)$}};
  \draw(13,0)node[below]{{ $)$}};
	\end{tikzpicture}\end{center}
\caption{An element of the Jones monoid $\J_7$ (above) along with its corresponding Catalan bracketing diagram (below).}
\label{fig_Jn}
\end{figure}
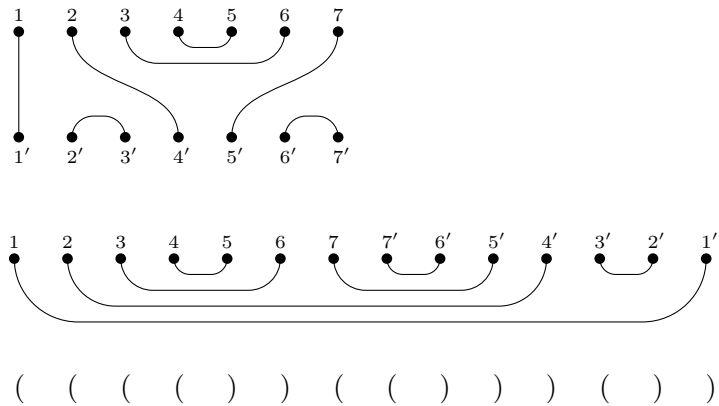

For $1\leq i\leq n-1$, write $\tau_i=\tau_{i,i+1}$.  (The partitions $\tau_{ij}$ were defined in the previous section.)  See Figure \ref{fig_tau_i} for an illustration.
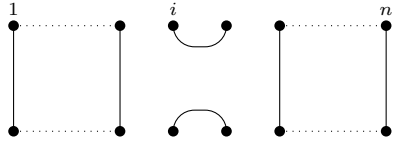
\begin{figure}[ht]
\begin{center}
\begin{tikzpicture}[xscale=.7,yscale=0.7]
\draw(3,2)arc(180:270:.4) (3.4,1.6)--(3.6,1.6) (3.6,1.6) arc(270:360:.4);
\draw(3,0)arc(180:90:.4) (3.4,.4)--(3.6,.4) (3.6,.4) arc(90:0:.4);
	\fill (0,0)circle(.1)
	      (2,0)circle(.1)
	      (3,0)circle(.1)
	      (4,0)circle(.1)
	      (5,0)circle(.1)
	      (7,0)circle(.1)
	      (0,2)circle(.1)
	      (2,2)circle(.1)
	      (3,2)circle(.1)
	      (4,2)circle(.1)
	      (5,2)circle(.1)
	      (7,2)circle(.1);
	\draw (0,2)--(0,0)
	      (2,2)--(2,0)
	      (5,2)--(5,0)
	      (7,2)--(7,0);
  \draw [dotted] (0,2)--(2,2)
                 (5,2)--(7,2)
                 (0,0)--(2,0)
                 (5,0)--(7,0);
  \draw(3,2)node[above]{{\tiny $i$}};
  \draw(0,2)node[above]{{\tiny $1$}};
  \draw(7,2)node[above]{{\tiny $n$}};
	\end{tikzpicture}
\end{center}
\caption{The projection $\tau_i\in\J_n$.}
\label{fig_tau_i}
\end{figure}

The following result is well known; see for example \cite{FitzGerald2006}, where presentations are also discussed. 

\begin{theorem}
The singular part $\J_n\setminus\{1\}$ of the Jones monoid $\J_n$ is idempotent generated.  The set $\set{\tau_i}{1\leq i\leq n-1}$ is a minimal idempotent generating set.
\end{theorem}

\begin{remark}
Because the singular part of the Jones monoid $\J_n$ is simply $\J_n\setminus\{1\}$, we could state many of the results of this section in terms of \emph{monoid} generating sets of $\J_n$, rather than semigroup generating sets of $\J_n\setminus\{1\}$.  However, for reasons of consistency, we will not do this.
\end{remark}

For $0\leq r=n-2k\leq n$, let
\[
\JrJn=\set{\al\in\J_n}{\rank(\al)=r} \quad\text{and}\quad \IrJn=\set{\al\in\J_n}{\rank(\al)\leq r}.
\]
As before, these are precisely the $\gj$-classes and ideals of $\J_n$, and the $\gj$-classes form a chain; but this time all maximal subgroups are trivial.

\subsection{Rank and idempotent rank of ideals of $\J_n$}

In the proof of Lemma \ref{lem_dropdownBn}, it was shown that any projection $\al\in\JrBn$ with $r=n-2k$ and $k\geq2$ was the product of two partitions $\be,\ga\in\JrpBn$, and we wish to establish the corresponding result for~$\J_n$.  However, in the proof of Lemma \ref{lem_dropdownBn}, even if $\al$ is planar, the partitions $\be,\ga$ constructed in the proof need not be planar themselves.  For example, with $\al=\tau_{14}\tau_{23}$, the proof of Lemma~\ref{lem_dropdownBn} gives $\be=\tau_{14}$ and $\ga=\tau_{23}$, with $\ga$ being planar but not $\be$.  So we must work a bit harder to prove the next result.

\begin{lemma}\label{lem_dropdownJn}
If $0\leq r\leq n-4$, then $\JrJn \subseteq \lb \JrpJn \rb$. 
\end{lemma}

\begin{proof}
As usual, it suffices to show that any projection from $\JrJn$ belongs to $\la\JrpJn\ra$, so suppose $\al\in\JrJn$ is a projection.  We consider two cases.  For the duration of this proof, it will be convenient to introduce some terminology.  If $1\leq i<j\leq n$ and $1\leq r<s\leq n$, we will say that $\{i,j\}$ is \emph{surrounded by} $\{r,s\}$ if $r<i<j<s$.

\

\noindent \textbf{Case 1:} Suppose first that $\al$ has two blocks of the form $\{i,i+1\}$ and $\{j,j+1\}$ that are not surrounded by any other blocks of $\al$; see Figure \ref{fig_dropdownJn_1}.  In this case, $\al=\be\ga$ where $\be,\ga\in\JrpJn$ are also illustrated in Figure \ref{fig_dropdownJn_1}.  In the diagram, it is understood that the shaded parts of $\be,\ga$ are the same as the corresponding shaded parts of $\al$.  (Since $\al$ is an idempotent, these shaded parts are all idempotents in isomorphic copies of $\J_{i-1}$, $\J_{j-i-2}$ and $\J_{n-j-1}$.)

\

\noindent \textbf{Case 2:} If we are not in Case 1, then $\al$ must have at least one block of the form $\{i,j\}$ where $j\geq i+3$, and $\{i,j\}$ is not surrounded by any other block of $\al$; see Figure \ref{fig_dropdownJn_2}.  Due to planarity, this implies that vertices $i+1,i+2,\ldots,j-1$ are all involved in non-transversal blocks.  In this case, $\al=\be\ga$ where $\be,\ga\in\JrpJn$ are also illustrated in Figure \ref{fig_dropdownJn_2}.
\end{proof}

\begin{figure}[ht]
\begin{center}
\begin{tikzpicture}[xscale=.7,yscale=0.7]
\bluebox0032
\bluebox6092
\bluebox{12}0{15}2
\arcup45
\arcup{10}{11}
\arcdn45
\arcdn{10}{11}
	\fill (0,0)circle(.1)
	      (3,0)circle(.1)
	      (4,0)circle(.1)
	      (5,0)circle(.1)
	      (6,0)circle(.1)
	      (9,0)circle(.1)
	      (10,0)circle(.1)
	      (11,0)circle(.1)
	      (12,0)circle(.1)
	      (15,0)circle(.1)
	      (0,2)circle(.1)
	      (3,2)circle(.1)
	      (4,2)circle(.1)
	      (5,2)circle(.1)
	      (6,2)circle(.1)
	      (9,2)circle(.1)
	      (10,2)circle(.1)
	      (11,2)circle(.1)
	      (12,2)circle(.1)
	      (15,2)circle(.1);
  \draw [dotted] (0,2)--(3,2)
                 (0,0)--(3,0)
                 (6,2)--(9,2)
                 (6,0)--(9,0)
                 (12,0)--(15,0)
                 (12,2)--(15,2);
  \draw(0,2)node[above]{{\tiny $\phantom{j}1\phantom{j}$}};
  \draw(4,2)node[above]{{\tiny $\phantom{j}i\phantom{j}$}};
  \draw(10,2)node[above]{{\tiny $\phantom{j}j\phantom{j}$}};
  \draw(15,2)node[above]{{\tiny $\phantom{j}n\phantom{j}$}};
  \draw(-.2,1)node[left]{$\al=$};
	\end{tikzpicture}
\\ \ \\ \ \\
\begin{tikzpicture}[xscale=.7,yscale=0.7]
\bluebox0032
\bluebox6092
\bluebox{12}0{15}2
\cve44
\cve55
\arcup{10}{11}
\arcdn{10}{11}
	\fill (0,0)circle(.1)
	      (3,0)circle(.1)
	      (4,0)circle(.1)
	      (5,0)circle(.1)
	      (6,0)circle(.1)
	      (9,0)circle(.1)
	      (10,0)circle(.1)
	      (11,0)circle(.1)
	      (12,0)circle(.1)
	      (15,0)circle(.1)
	      (0,2)circle(.1)
	      (3,2)circle(.1)
	      (4,2)circle(.1)
	      (5,2)circle(.1)
	      (6,2)circle(.1)
	      (9,2)circle(.1)
	      (10,2)circle(.1)
	      (11,2)circle(.1)
	      (12,2)circle(.1)
	      (15,2)circle(.1);
  \draw [dotted] (0,2)--(3,2)
                 (0,0)--(3,0)
                 (6,2)--(9,2)
                 (6,0)--(9,0)
                 (12,0)--(15,0)
                 (12,2)--(15,2);
  \draw(0,2)node[above]{{\tiny $\phantom{j}1\phantom{j}$}};
  \draw(4,2)node[above]{{\tiny $\phantom{j}i\phantom{j}$}};
  \draw(10,2)node[above]{{\tiny $\phantom{j}j\phantom{j}$}};
  \draw(15,2)node[above]{{\tiny $\phantom{j}n\phantom{j}$}};
  \draw(-.2,1)node[left]{$\be=$};
\end{tikzpicture}
\\ \ \\ \ \\
\begin{tikzpicture}[xscale=.7,yscale=0.7]
\bluebox0032
\bluebox6092
\bluebox{12}0{15}2
\arcup45
\arcdn45
\cve{10}{10}
\cve{11}{11}
	\fill (0,0)circle(.1)
	      (3,0)circle(.1)
	      (4,0)circle(.1)
	      (5,0)circle(.1)
	      (6,0)circle(.1)
	      (9,0)circle(.1)
	      (10,0)circle(.1)
	      (11,0)circle(.1)
	      (12,0)circle(.1)
	      (15,0)circle(.1)
	      (0,2)circle(.1)
	      (3,2)circle(.1)
	      (4,2)circle(.1)
	      (5,2)circle(.1)
	      (6,2)circle(.1)
	      (9,2)circle(.1)
	      (10,2)circle(.1)
	      (11,2)circle(.1)
	      (12,2)circle(.1)
	      (15,2)circle(.1);
  \draw [dotted] (0,2)--(3,2)
                 (0,0)--(3,0)
                 (6,2)--(9,2)
                 (6,0)--(9,0)
                 (12,0)--(15,0)
                 (12,2)--(15,2);
  \draw(0,2)node[above]{{\tiny $\phantom{j}1\phantom{j}$}};
  \draw(4,2)node[above]{{\tiny $\phantom{j}i\phantom{j}$}};
  \draw(10,2)node[above]{{\tiny $\phantom{j}j\phantom{j}$}};
  \draw(15,2)node[above]{{\tiny $\phantom{j}n\phantom{j}$}};
  \draw(-.2,1)node[left]{$\ga=$};
	\end{tikzpicture}
\end{center}
\caption{The partitions $\al\in\JrJn$ and $\be,\ga\in\JrpJn$ from Case 1 of the proof of Lemma \ref{lem_dropdownJn}.}
\label{fig_dropdownJn_1}
\end{figure}

\begin{figure}[ht]
\begin{center}
\begin{tikzpicture}[xscale=.7,yscale=0.7]
\bluebox0032
\bluebox5{1.6}{11}2
\bluebox50{11}{.4}
\longarcup4{12}
\longarcdn4{12}
\bluebox{13}0{16}2
	\fill (0,0)circle(.1)
	      (3,0)circle(.1)
	      (4,0)circle(.1)
	      (5,0)circle(.1)
	      (11,0)circle(.1)
	      (12,0)circle(.1)
	      (13,0)circle(.1)
	      (16,0)circle(.1)
	      (0,2)circle(.1)
	      (3,2)circle(.1)
	      (4,2)circle(.1)
	      (5,2)circle(.1)
	      (11,2)circle(.1)
	      (12,2)circle(.1)
	      (13,2)circle(.1)
	      (16,2)circle(.1);
  \draw [dotted] (0,2)--(3,2)
                 (0,0)--(3,0)
                 (5,2)--(11,2)
                 (5,0)--(11,0)
                 (13,0)--(16,0)
                 (13,2)--(16,2);
  \draw(0,2)node[above]{{\tiny $\phantom{j}1\phantom{j}$}};
  \draw(4,2)node[above]{{\tiny $\phantom{j}i\phantom{j}$}};
  \draw(12,2)node[above]{{\tiny $\phantom{j}j\phantom{j}$}};
  \draw(16,2)node[above]{{\tiny $\phantom{j}n\phantom{j}$}};
  \draw(-.2,1)node[left]{$\al=$};
	\end{tikzpicture}
\\ \ \\ \ \\
\begin{tikzpicture}[xscale=.7,yscale=0.7]
\bluebox0032
\bluebox5{1.6}{11}2
\arcdn67
\arcdn9{10}
\arcdn{11}{12}
\cve44
\draw(12,2)arc(0:-90:3.5 and 1);
\draw(5,0)arc(180:90:3.5 and 1);
\bluebox{13}0{16}2
	\fill (0,0)circle(.1)
	      (3,0)circle(.1)
	      (4,0)circle(.1)
	      (5,0)circle(.1)
	      (6,0)circle(.1)
	      (7,0)circle(.1)
	      (9,0)circle(.1)
	      (10,0)circle(.1)
	      (11,0)circle(.1)
	      (12,0)circle(.1)
	      (13,0)circle(.1)
	      (16,0)circle(.1)
	      (0,2)circle(.1)
	      (3,2)circle(.1)
	      (4,2)circle(.1)
	      (5,2)circle(.1)
	      (11,2)circle(.1)
	      (12,2)circle(.1)
	      (13,2)circle(.1)
	      (16,2)circle(.1);
  \draw [dotted] (0,2)--(3,2)
                 (0,0)--(3,0)
                 (5,2)--(11,2)
                 (7,0)--(9,0)
                 (13,0)--(16,0)
                 (13,2)--(16,2);
  \draw(0,2)node[above]{{\tiny $\phantom{j}1\phantom{j}$}};
  \draw(4,2)node[above]{{\tiny $\phantom{j}i\phantom{j}$}};
  \draw(12,2)node[above]{{\tiny $\phantom{j}j\phantom{j}$}};
  \draw(16,2)node[above]{{\tiny $\phantom{j}n\phantom{j}$}};  \draw(-.2,1)node[left]{$\be=$};
\end{tikzpicture}
\\ \ \\ \ \\
\begin{tikzpicture}[xscale=.7,yscale=0.7]
\bluebox0032
\bluebox50{11}{.4}
\arcup45
\arcup67
\arcup9{10}
\cve{12}{12}
\bluebox{13}0{16}2
\draw(11,2)arc(0:-90:3.5 and 1);
\draw(4,0)arc(180:90:3.5 and 1);
	\fill (0,0)circle(.1)
	      (3,0)circle(.1)
	      (4,0)circle(.1)
	      (5,0)circle(.1)
	      (11,0)circle(.1)
	      (12,0)circle(.1)
	      (13,0)circle(.1)
	      (16,0)circle(.1)
	      (0,2)circle(.1)
	      (3,2)circle(.1)
	      (4,2)circle(.1)
	      (5,2)circle(.1)
	      (6,2)circle(.1)
	      (7,2)circle(.1)
	      (9,2)circle(.1)
	      (10,2)circle(.1)
	      (11,2)circle(.1)
	      (12,2)circle(.1)
	      (13,2)circle(.1)
	      (16,2)circle(.1);
  \draw [dotted] (0,2)--(3,2)
                 (0,0)--(3,0)
                 (7,2)--(9,2)
                 (5,0)--(11,0)
                 (13,0)--(16,0)
                 (13,2)--(16,2);
  \draw(0,2)node[above]{{\tiny $\phantom{j}1\phantom{j}$}};
  \draw(4,2)node[above]{{\tiny $\phantom{j}i\phantom{j}$}};
  \draw(12,2)node[above]{{\tiny $\phantom{j}j\phantom{j}$}};
  \draw(16,2)node[above]{{\tiny $\phantom{j}n\phantom{j}$}};  \draw(-.2,1)node[left]{$\ga=$};
	\end{tikzpicture}
\end{center}
\caption{The partitions $\al\in\JrJn$ and $\be,\ga\in\JrpJn$ from Case 2 of the proof of Lemma \ref{lem_dropdownJn}.}
\label{fig_dropdownJn_2}
\end{figure}
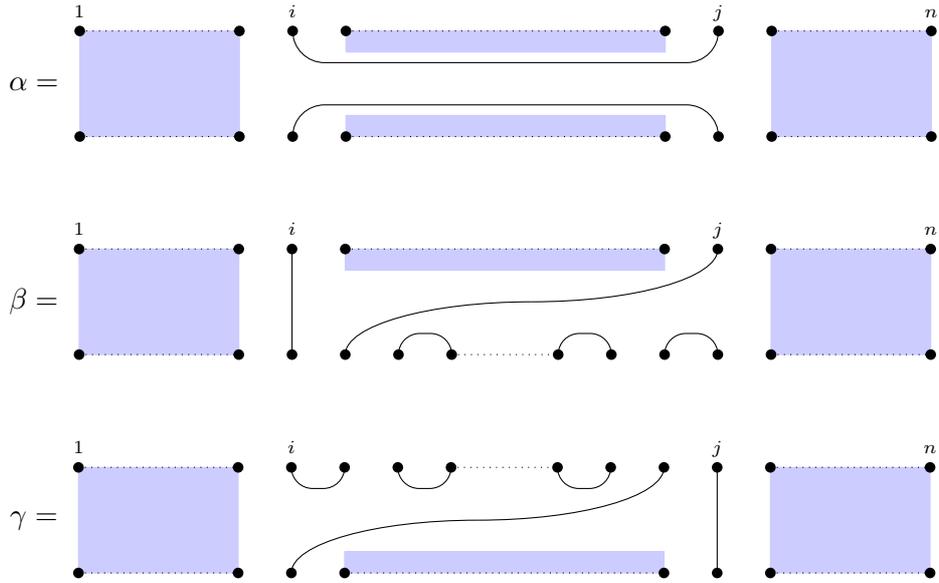

\begin{remark}
In Case 1 of the above proof, $\be$ and $\ga$ were both projections from $\JrpJn$.  In Case 2, they were not (and could not be; consider the above example of $\al=\tau_{14}\tau_{23}$).  However, they are both idempotents. Indeed,
consider $\beta$ for example.  Because $\alpha$ is a projection, the only part of $\beta$ we do not automatically know is idempotent is the portion contained between points $i+1, \ldots ,j$.  But this is essentially a (planar) rank $1$ Brauer diagram, and all of these are idempotents.  
Similarly we see that $\gamma$ is an idempotent. 
Since $\beta$ and $\gamma$ are both idempotents, it follows that any projection from $\JrJn$ can be expressed as the product of at most four projections from $\JrpJn$ if $0\leq r\leq n-4$.  In general, this bound (of four projections) is sharp, as the above example of $\al=\tau_{14}\tau_{23}\in\J_4$ demonstrates.
\end{remark}

\begin{theorem}
\label{thm_main4}
For $0\leq r=n-2k\leq n-2$, the ideal $\IrJn$ is idempotent generated, and
\[
\rank(\IrJn) = \idrank(\IrJn) = \frac{r+1}{n+1}\binom {n+1}k.
\] 
Moreover, a subset $A \subseteq \IrJn$ of this cardinality is a generating set for $\IrJn$ if and only if the following three conditions hold\emph{:}
\begin{enumerate}
\item $\rank(\alpha) = r$ for all $\alpha \in A$\emph{;}
\item for all $\alpha, \beta \in A$ with $\alpha \neq \beta$, $\ker(\al)\not=\ker(\be)$\emph{;}
\item for all $\alpha, \beta \in A$ with $\alpha \neq \beta$, $\coker(\al)\not=\coker(\be)$.
\end{enumerate}
\end{theorem}

\begin{proof}
Fix some $0\leq r=n-2k\leq n$, and denote by $\rho_{nr}$ the number of $\gr$-classes in $\JrJn$.  As usual, it suffices to show that $\rho_{nr}=\frac{r+1}{n+1}\binom {n+1}k$.
As noted above, an element $\al\in\J_n$ is completely determined by a string of length $2n$, consisting of a proper bracketing of $n$ pairs of brackets.  We denote by $L(\al)$ the substring consisting of the first $n$ symbols in this string.  If $\al\in\JrJn$, then this string will have~$k$ right brackets and $k+r$ left brackets, $r$ of which are unmatched.  These unmatched left brackets of $L(\al)$ correspond to the elements of $\dom(\al)$, while the pairs of matched brackets in $L(\al)$ correspond to the non-trivial $\ker(\al)$-classes.  Hence, by Theorem~\ref{thm_wilcox_new}, it follows that $\rho_{nr}$ is equal to the number of all such strings of $n$ brackets (with the specified number of unmatched left brackets), and it is well known that there are $\frac{r+1}{n+1}\binom {n+1}k$ such strings; see for example \cite[p.~303]{Brualdi2010}. 
\end{proof}

\begin{remark}\label{rem:|IrJn|}
As in Remarks \ref{rem:|IrPn|} and \ref{rem:|IrBn|},  but noting that $\gh$-classes of $\J_n$ are singletons, we obtain
\[
|\JrJn| = \rho_{nr}^2 \qquad\text{and}\qquad  |\IrJn| = \sum_{i=0}^r \rho_{ni}^2 \qquad\text{for any $0\leq r\leq n$},
\]
where the numbers $\rho_{nr}$ are as in (the proof of) Theorem \ref{thm_main4}, and we interpret $\rho_{ni}=0$ if $i\not\equiv n \pmod 2$.
\end{remark}

\begin{remark}\label{rem:irreducibles}
The sequence $\rho_{nr}$ is A053121 on \cite{Sloan:OEIS}, and appears in a variety of contexts, most notably as the Catalan triangle (concerning U/D walks in the first quadrant).  Of particular present interest is that these numbers give the dimensions of the irreducible representations of the \emph{Temperley--Lieb algebra} of degree $n$ in the semisimple case.
%
%\footnote{The authors thank Prof.~Arun Ram for bringing this fact to our attention.}  
%
The dimensions of the irreducible representations of the \emph{Brauer} and \emph{partition algebras} may also be obtained from our results on (idempotent) ranks of ideals of the corresponding diagram monoids.  See Section~\ref{sec_irreps} below for full details of this. 
\end{remark}

%For example, the irreducible representations, denoted $A_n^\mu$ in \cite{Halverson2005}, of the partition algebra of degree $n$ are indexed by all integer partitions $\mu$ of all integers $r$ satisfying $0\leq r\leq n$.  It may then be shown (in light of the discussion of Munn--Ponizovski{\u\i} theory in the introduction) that if $\mu$ is such a partition and $r<n$, then
%\[
%\dim(A_n^\mu) = \rank(\IrPn)\cdot|\operatorname{Std}(\mu)|,
%\]
%where 
%$\operatorname{Std}(\mu)$ is the set of all standard tableau with shape $\mu$.  While this formula does not hold for $r=n$, the term $\rank(\IrPn)$ may simply be replaced by $1$ (the number of $\gl$- and $\gr$-classes in the $\gj$-class $\JnPn=\S_n$) to obtain the correct formula.  The dimensions (in the semisimple case) of the irreducible representations of the partition, Brauer and Temperley--Lieb algebras have been obtained in a number of different ways, by a number of different authors; see for example \cite{Halverson2005,Pan95,Wenzl88,GrahamLehrer1996,Martin1990_2,Martin1990,GL1998,BR1999,EG1999,Martin1987,Grood1999,
%KX1998,KX2001} and especially the recent survey \cite{Ridout14} and the many references therein.
%The first few values of $\rho_{nr}$ are given in Table~\ref{tab_rankIrJn}.

\begin{table}[ht]%
\begin{tabular}{|c|ccccccccccc|}
\hline
$n$ $\setminus$ $r$	&0	&1	&2	&3	&4	&5	&6	&7 &8 &9 &10 \\
\hline
 0&           1&            &            &            &            &            &            &            &          &            &                        \\
 1&            &           1&            &            &            &            &            &            &          &            &                        \\
 2&           1&            &           1&            &            &            &            &            &          &            &                        \\
 3&            &           2&            &           1&            &            &            &            &          &            &                        \\
 4&           2&            &           3&            &           1&            &            &            &          &            &                        \\
 5&            &           5&            &           4&            &           1&            &            &          &            &                        \\
 6&           5&            &           9&            &           5&            &           1&            &          &            &                        \\
 7&            &          14&            &          14&            &           6&            &           1&          &            &                        \\
 8&          14&            &          28&            &          20&            &           7&            &         1&            &                        \\
 9&            &          42&            &          48&            &          27&            &           8&          &           1&                        \\
10&          42&            &          90&            &          75&            &          35&            &         9&            &           1            \\
\hline
\end{tabular}
\caption{Values of $\rho_{nr}$ from the proof of Theorem~\ref{thm_main4}.  For $0\leq r\leq n-2$, $\rho_{nr}=\rank(I_r(\J_n)) = \idrank(I_r(\J_n))$.}
\label{tab_rankIrJn}
\end{table}

\subsection{Minimal idempotent generating sets of $\J_n\setminus\{1\}$}

Again, an enumeration of the minimal idempotent generating sets of $\J_n\setminus\{1\}$ amounts to an enumeration of the balanced subgraphs of the projection graph $\Gamma(\JnSn)=\Gamma(\JnmJn^*)$, which we will denote by~$\Xi_n$.

For $1\leq i\leq n-2$, put
$
\lam_i=\si_{i,i+1,i+2} $ and $ \rho_i=\si_{i+2,i+1,i}.
$
(The idempotents $\si_{ijk}$ were defined in the previous section.)
See Figure \ref{fig_lam_i_rho_i} for an illustration.
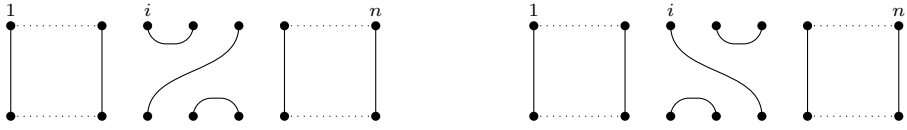
\begin{figure}[ht]
\begin{center}
\begin{tikzpicture}[xscale=.6,yscale=0.6]
	\fill (0,0)circle(.1)
	      (2,0)circle(.1)
	      (3,0)circle(.1)
	      (4,0)circle(.1)
	      (5,0)circle(.1)
	      (6,0)circle(.1)
	      (8,0)circle(.1)
	      (0,2)circle(.1)
	      (2,2)circle(.1)
	      (3,2)circle(.1)
	      (4,2)circle(.1)
	      (5,2)circle(.1)
	      (6,2)circle(.1)
	      (8,2)circle(.1);
	\arcup34
	\arcdn45
	\cve00
	\cve22
	\cve53
	\cve66
	\cve88
	\draw[dotted](0,2)--(2,2);
	\draw[dotted](0,0)--(2,0);
	\draw[dotted](6,2)--(8,2);
	\draw[dotted](6,0)--(8,0);
  \draw(0,2)node[above]{{\tiny $1$}};
  \draw(3,2)node[above]{{\tiny $i$}};
  \draw(8,2)node[above]{{\tiny $n$}};
	\end{tikzpicture}
	\qquad\qquad
\begin{tikzpicture}[xscale=.6,yscale=0.6]
	\fill (0,0)circle(.1)
	      (2,0)circle(.1)
	      (3,0)circle(.1)
	      (4,0)circle(.1)
	      (5,0)circle(.1)
	      (6,0)circle(.1)
	      (8,0)circle(.1)
	      (0,2)circle(.1)
	      (2,2)circle(.1)
	      (3,2)circle(.1)
	      (4,2)circle(.1)
	      (5,2)circle(.1)
	      (6,2)circle(.1)
	      (8,2)circle(.1);
	\arcup45
	\arcdn34
	\cve00
	\cve22
	\cve35
	\cve66
	\cve88
	\draw[dotted](0,2)--(2,2);
	\draw[dotted](0,0)--(2,0);
	\draw[dotted](6,2)--(8,2);
	\draw[dotted](6,0)--(8,0);
  \draw(0,2)node[above]{{\tiny $1$}};
  \draw(3,2)node[above]{{\tiny $i$}};
  \draw(8,2)node[above]{{\tiny $n$}};
	\end{tikzpicture}		\end{center}
\caption{The idempotents $\lam_i$ (left) and $\rho_i$ (right) from $\J_n$.}
\label{fig_lam_i_rho_i}
\end{figure}

Again, the next lemma is easily verified, and its proof is omitted.

\begin{lemma}\label{lem_idempotents_in_J_{n-2}_Jn}
The set of idempotents of $\JnmJn$ is
\[
\set{\tau_i}{1\leq i\leq n-1}\cup\set{\lam_i,\rho_i}{1\leq i\leq n-2}.
\]
The set of projections of $\JnmJn$ is
\[
\set{\tau_i}{1\leq i\leq n-1}.
\]
In the principal factor $\JnmJn^*$, the only nonzero products of pairs of projections are
\[
\tau_{i}^2=\tau_{i}, \ \ \tau_{i}\tau_{i+1}=\lam_i, \ \ \tau_{i+1}\tau_{i}=\rho_i.
\]
\end{lemma}

It follows that the graph $\Xi_n$ has vertex set $\set{\tau_{i}}{1\leq i\leq n-1}$ with edges $\tau_{i}\to\tau_{j}$ if and only if $|i-j|\leq1$.  The graph $\Xi_5$ is pictured in Figure~\ref{fig_Xi_5}, with vertices labeled $i$ instead of $\tau_i$.
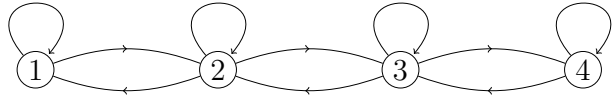
\begin{figure}[ht]
\begin{center}
\scalebox{0.8}
{
\begin{tikzpicture}[scale=3]
\tikzstyle{vertex}=[circle,draw=black, fill=white, inner sep = 0.07cm]
\node[vertex] (1) at (0,0) { \Large $1$ };
\node[vertex] (2) at (1,0) { \Large $2$ };
\node[vertex] (3) at (2,0) { \Large $3$ };
\node[vertex] (4) at (3,0) { \Large $4$ };
\draw[->-=0.5] (1) to [bend left=20] (2);
\draw[->-=0.5] (2) to [bend left=20] (3);
\draw[->-=0.5] (3) to [bend left=20] (4);
\draw[-<-=0.5] (1) to [bend right=20] (2);
\draw[-<-=0.5] (2) to [bend right=20] (3);
\draw[-<-=0.5] (3) to [bend right=20] (4);
\draw [->] (1) edge [out=130,in=50,loop] ();
\draw [->] (2) edge [out=130,in=50,loop] ();
\draw [->] (3) edge [out=130,in=50,loop] ();
\draw [->] (4) edge [out=130,in=50,loop] ();
\end{tikzpicture}
}
\end{center}
\caption{The graph $\Xi_5=\Gamma(\J_5\setminus\{1\})$.}
\label{fig_Xi_5}
\end{figure}

\begin{theorem}
The number of minimal idempotent generating sets for the singular part $\J_n\setminus\{1\}$ of the Jones monoid $\J_n$ is equal to $F_n$, the $n$th Fibonacci number, where $F_1=F_2=1$ and $F_n=F_{n-1}+F_{n-2}$ for $n\geq3$.
\end{theorem}

\begin{proof}
Let $x_n$ be the number of balanced subgraphs of $\Xi_n$.  It suffices to show that the numbers $x_n$ satisfy the Fibonacci recurrence.  It is clear that $x_1=x_2=1$.  (Note that $\J_1\setminus\{1\}=\varnothing$, so there is only one generating set, $\varnothing$.)  Next, suppose $n\geq3$.  A balanced subgraph of $\Xi_n$ must contain either the loop at $\tau_{n-1}$ and a balanced subgraph of $\Xi_{n-1}$, or else the edges $\tau_{n-2}\leftrightarrows\tau_{n-1}$ and a balanced subgraph of $\Xi_{n-2}$.  So $x_n=x_{n-1}+x_{n-2}$, as required.
\end{proof}

Although these numbers are well-known (see A000045 on \cite{Sloan:OEIS}), we include the first few values, for completeness, in Table~\ref{tab_Fn}.

\begin{table}[ht]%
\begin{center}
\begin{tabular}{|c|cccccccccc|}
\hline
$n$ & $1$ & $2$ & $3$ & $4$ & $5$ & $6$ & $7$ & $8$ & $9$ & $10$ \\
\hline
$F_n$  & $1$ &  $1$ &  $2$  & $3$  & $5$  & $8$  & $13$  & $21$  & $34$ & $55$ \\
\hline
\end{tabular}
\end{center}
\caption{The Fibonacci sequence $F_n$.  For $n\geq1$, $F_n$ is equal to the number of minimal idempotent generating sets of $\J_n\setminus\{1\}$.}
\label{tab_Fn}
\end{table}

\subsection{Arbitrary idempotent generating sets for $\JnSn$.}

For a subset $F$ of
\[
\{\al\in E(\J_n):\rank(\al)=n-2\} = \{\tau_1,\ldots,\tau_{n-1}\} \cup \{\lam_1,\ldots,\lam_{n-2}\} \cup \{\rho_1,\ldots,\rho_{n-2}\},
\]
let $\Xi_n(F)$ be the two-coloured digraph obtained by colouring each edge of $\Xi_n$ blue, and then adding red edges corresponding to the idempotents from $F$:
\begin{itemize}
	\item $i{\boldsymbol{\red\to}}i$ if $\tau_i\in F$,
	\item $i{\boldsymbol{\red\to}}i+1$ if $\lam_i\in F$,
	\item $i+1{\boldsymbol{\red\to}}i$ if $\rho_i\in F$.
\end{itemize}
(As above, we will denote the vertices of $\Xi_n$ by $1,\ldots,n-1$ rather than $\tau_1,\ldots,\tau_{n-1}$.)
  
\begin{theorem}\label{arb_gen_set_Jn}
For $F\subseteq \{\al\in E(\J_n):\rank(\al)=n-2\}$, the following are equivalent\emph{:}
\begin{enumerate}
	\item[(i)] $\JnSn=\la F\ra$\emph{;}
	\item[(ii)] each vertex of $\Xi_n(F)$ is the base point of an RBR-alternating circuit\emph{;}
	\item[(iii)] each vertex of $\Xi_n(F)$ is contained in a red circuit.
\end{enumerate}
The number $f_n$ of such subsets is given by the recurrence
\[
f_2=1, \quad f_3=7, \quad f_n=5f_{n-1}+6f_{n-2} \quad\text{for $n\geq4$.}
\]
\end{theorem}

\begin{proof}
By Lemma~\ref{lem_dropdownJn}, $F$ generates $\JnSn$ if and only if it generates the principle factor $\JnmJn^*$.  So the equivalence of (i) and (ii) follows from Theorem~\ref{thm_generalRBR}.  Any red circuit beginning and ending at vertex $i$ clearly gives rise to an RBR-circuit based at $i$, since $\Xi_n(F)$ has a blue loop at each vertex, so (iii) implies (ii).

Now assume (ii), and suppose $1\leq i\leq n-1$.  We must show that $i$ is contained in a red circuit.  Consider first the case in which $i=1$.  If $1{\boldsymbol{\red\to}}1$ is an edge, then $1$ is contained in a red circuit.  If not, then, since there is an RBR-alternating circuit based at $1$, we see that $\Xi_n(F)$ contains the edges $1{\boldsymbol{\red\leftrightarrows}}2$, so $1$ is contained in a red circuit.  By symmetry, $n-1$ is contained in a red circuit.

Finally, suppose $1<i<n-1$.  If $\Xi_n(F)$ contains the loop $i{\boldsymbol{\red\to}}i$ or the edges $i{\boldsymbol{\red\leftrightarrows}}i+1$ or $i+1{\boldsymbol{\red\leftrightarrows}}i$, then $i$ is contained in a red circuit.  In order to obtain a contradiction, suppose this is not the case.  By symmetry, we may suppose that $\Xi_n(F)$ contains the edges $i-1{\boldsymbol{\red\to}}i{\boldsymbol{\red\to}}i+1$, but no other red edges at $i$.  Now consider an RBR-alternating path beginning at $i$:
\[
i=j_0 {\boldsymbol{\red\to}} j_1 {\blue\to} j_2 {\boldsymbol{\red\to}} \cdots {\blue\to} j_{2k} {\boldsymbol{\red\to}} j_{2k+1}.
\]
A simple induction shows that $j_{2r}\geq i$ and $j_{2r+1}\geq i+1$ for each $0\leq r\leq k$.  In particular, $j_{2k+1}\geq i+1$, so it follows that $\Xi_n(F)$ contains no RBR-alternating circuit based at vertex~$i$, a contradiction.  This completes the proof that (ii) implies (iii).

Now let $\X_n$ denote the set of all graphs $\Xi_n(F)$ satisfying property (iii).  We must show that the cardinalities $f_n=|\X_n|$ satisfy the given recurrence.  It is easy to check that $f_2=1$ and $f_3=7$, so suppose now that $n\geq4$.  Let $G\in\X_n$.  Since the vertex $n-1$ is contained in a red circuit, there are three cases we must consider:
\begin{itemize}
	\item[(1)] $G$ does not contain both of the edges $n-1{\boldsymbol{\red\leftrightarrows}}n-2$, or
	\item[(2)] $G$ contains both of the edges $n-2{\boldsymbol{\red\leftrightarrows}}n-1$ and the loop $n-1{\boldsymbol{\red\to}}n-1$, or
	\item[(3)] $G$ contains both of the edges $n-2{\boldsymbol{\red\leftrightarrows}}n-1$ but not the loop $n-1{\boldsymbol{\red\to}}n-1$.
\end{itemize}
If (1) holds, then $G$ contains the loop $n-1{\boldsymbol{\red\to}}n-1$, and the subgraph of $G$ induced by the vertices $1,\ldots,n-2$ belongs to $\X_{n-1}$.  Since there are three ways to choose at most one of the edges $n-1{\boldsymbol{\red\leftrightarrows}}n-2$, there are $3f_{n-1}$ graphs $G$ satisfying (1).

Now suppose (2) holds.  Let $H$ be the subgraph of $G$ induced by the vertices $1,\ldots,n-2$.  Note that $G$ is completely determined by $H$.  There are two subcases:
\begin{itemize}
	\item[(2.1)] $H$ belongs to $\X_{n-1}$, or
	\item[(2.2)] $H$ does not belong to $\X_{n-1}$.
\end{itemize}
There are obviously $f_{n-1}$ graphs $G$ satisfying (2.1).  Now suppose (2.2) holds.  Now, each vertex $1,\ldots,n-3$ belongs to a red circuit of $G$.  Since such a circuit does not need to include the edges $n-2{\boldsymbol{\red\leftrightarrows}}n-1$, we see that, in fact, each vertex $1,\ldots,n-3$ belongs to a red circuit of $H$.  Since we assumed that $H\not\in\X_{n-1}$, it follows that $H$ does not contain the loop $n-2{\boldsymbol{\red\to}}n-2$ and contains at most one of the edges $n-3{\boldsymbol{\red\leftrightarrows}}n-2$.  It also follows that the subgraph of $H$ induced by the vertices $1,\ldots,n-3$ belongs to $\X_{n-2}$.  So there are $3f_{n-2}$ graphs $G$ satisfying (2.2).  Thus, there are $f_{n-1}+3f_{n-2}$ graphs $G$ satisfying~(2).  The same argument shows that there are the same number of graphs $G$ satifying (3).  Putting this together,
$
f_n = 3f_{n-1} +2(f_{n-1}+3f_{n-2}) = 5f_{n-1} +6f_{n-2}.
$
This completes the proof.
\end{proof}

\begin{remark}
Solving the above recurrence gives 
\[
f_n = \frac2{63}\cdot6^n + \frac17\cdot(-1)^{n+1}.
\]
This is sequence A108983 on \cite{Sloan:OEIS}.  The first few values of the sequence $f_n$ are given in Table~\ref{tab_fn}.  
\end{remark}

\begin{table}[ht]%
\begin{center}
\begin{tabular}{|c|ccccccccc|}
\hline
$n$ & $2$ & $3$ & $4$ & $5$ & $6$ & $7$ & $8$ & $9$ & $10$ \\
\hline
$f_n$  & $1$ &  $7$  & $41$  & $247$  & $1481$  & $8887$  & $53321$  & $319927$ & $1919561$ \\
\hline
\end{tabular}
\end{center}
\caption{The sequence $f_n$, which gives the number of generating sets for $\JnSn$ consisting of idempotents belonging to $\JnmJn$.
}
\label{tab_fn}
\end{table}

\begin{remark}\label{rem911}
Theorem \ref{arb_gen_set_Jn} shows that condition (i) from Remark~\ref{rem_generalRBR} is necessary and sufficient in the case of $\JnSn$.  While condition (ii) from Remark~\ref{rem_generalRBR} is, as ever, necessary, it is not sufficient in general.  For example, consider the set of idempotents $F=\{\tau_1,\tau_3,\lam_1,\lam_2\}$ from~$\J_4$.  The digraph $\Xi_4(F)$ is displayed in Figure~\ref{Fig_Xi4F}.  
As in Example~\ref{Example6}, there is no RBR-alternating circuit at vertex $2$, so $\JnSn\not=\la F\ra$.  
\end{remark}

\begin{figure}
\begin{center}
\begin{tikzpicture}[scale=1]
\tikzstyle{vertex}=[circle,draw=black, fill=white, inner sep = 0.07cm]
\node[vertex] (1) at (0,0) {  $\Large{1}$ };
\node[vertex] (12) at (2,0) { $\Large{2}$};
\node[vertex] (2) at (4,0) {  $\Large{3}$ };
\draw [ultra thick, ->-=0.5, color=red] (1)        to [bend right=40] (12);
\draw [ultra thick, ->-=0.5, color=red] (12)        to [bend right=40] (2);
\draw [ultra thick, ->, color=red] (1) edge [out=130+180,in=50+180,loop] ();
\draw [ultra thick, ->, color=red] (2) edge [out=130+180,in=50+180,loop] ();
\draw [->, color=blue] (1) edge [out=130,in=50,loop] ();
\draw [->, color=blue] (2) edge [out=130,in=50,loop] ();
\draw [->, color=blue] (12) edge [out=130,in=50,loop] ();
\draw [->-=0.5, color=blue] (1)        to [bend right=15] (12);
\draw [->-=0.5, color=blue] (12)        to [bend right=15] (2);
\draw [->-=0.5, color=blue] (12)        to [bend right=15] (1);
\draw [->-=0.5, color=blue] (2)        to [bend right=15] (12);
\end{tikzpicture}
\end{center}
\caption{The digraph $\Xi_4(F)$, where $F=\{\tau_1,\tau_3,\lam_1,\lam_2\}$.}
\label{Fig_Xi4F}
\end{figure}
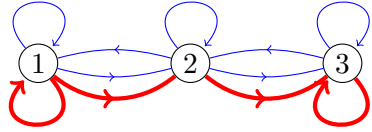

\subsection{The monoid of planar partitions}

Recall that the monoid of planar partitions is denoted $\PP_n$.  It is well known that $\PP_n$ is isomorphic to the Jones monoid $\J_{2n}$; see for example \cite{Halverson2005}.  This isomorphism is easiest to describe diagrammatically, and we do so in Figure \ref{fig_PPn_J2n}.  Because of this isomorphism, we will not state the results concerning the ideals 
\[
\IrPPn = \set{\al\in\PP_n}{\rank(\al)\leq r}
\]
and minimal idempotent generating sets, as these can be deduced easily from the results for $\J_n$ in the preceding sections.
\begin{figure}[ht]
\begin{center}
\begin{tikzpicture}[xscale=.8,yscale=0.8]
	\fill (0,0)circle(.1)
	      (1,0)circle(.1)
	      (2,0)circle(.1)
	      (3,0)circle(.1)
	      (4,0)circle(.1)
	      (5,0)circle(.1)
	      (6,0)circle(.1)
	      (7,0)circle(.1)
	      (0,2)circle(.1)
	      (1,2)circle(.1)
	      (2,2)circle(.1)
	      (3,2)circle(.1)
	      (4,2)circle(.1)
	      (5,2)circle(.1)
	      (6,2)circle(.1)
	      (7,2)circle(.1);
	\barcup01{.3}
	\barcup34{.3}
	\barcup25{.9}
	\barcup56{.3}
	\barcdn34{.3}
	\barcdn56{.3}
	\barcdn67{.3}
	\cve02
	\cve12
	\cve23
	\cve64
  \draw(0,2)node[above]{{\tiny $1$}};
  \draw(1,2)node[above]{{\tiny $2$}};
  \draw(2,2)node[above]{{\tiny $3$}};
  \draw(3,2)node[above]{{\tiny $4$}};
  \draw(4,2)node[above]{{\tiny $5$}};
  \draw(5,2)node[above]{{\tiny $6$}};
  \draw(6,2)node[above]{{\tiny $7$}};
  \draw(7,2)node[above]{{\tiny $8$}};
  \draw[color=lightgray,fill=lightgray] (0-.25,0)circle(.1);
  \draw[color=lightgray,fill=lightgray] (1-.25,0)circle(.1);
  \draw[color=lightgray,fill=lightgray] (2-.25,0)circle(.1);
  \draw[color=lightgray,fill=lightgray] (3-.25,0)circle(.1);
  \draw[color=lightgray,fill=lightgray] (4-.25,0)circle(.1);
  \draw[color=lightgray,fill=lightgray] (5-.25,0)circle(.1);
  \draw[color=lightgray,fill=lightgray] (6-.25,0)circle(.1);
  \draw[color=lightgray,fill=lightgray] (7-.25,0)circle(.1);
  \draw[color=lightgray,fill=lightgray] (0-.25,2)circle(.1);
  \draw[color=lightgray,fill=lightgray] (1-.25,2)circle(.1);
  \draw[color=lightgray,fill=lightgray] (2-.25,2)circle(.1);
  \draw[color=lightgray,fill=lightgray] (3-.25,2)circle(.1);
  \draw[color=lightgray,fill=lightgray] (4-.25,2)circle(.1);
  \draw[color=lightgray,fill=lightgray] (5-.25,2)circle(.1);
  \draw[color=lightgray,fill=lightgray] (6-.25,2)circle(.1);
  \draw[color=lightgray,fill=lightgray] (7-.25,2)circle(.1);
  \draw[color=lightgray,fill=lightgray] (0.25,0)circle(.1);
  \draw[color=lightgray,fill=lightgray] (1.25,0)circle(.1);
  \draw[color=lightgray,fill=lightgray] (2.25,0)circle(.1);
  \draw[color=lightgray,fill=lightgray] (3.25,0)circle(.1);
  \draw[color=lightgray,fill=lightgray] (4.25,0)circle(.1);
  \draw[color=lightgray,fill=lightgray] (5.25,0)circle(.1);
  \draw[color=lightgray,fill=lightgray] (6.25,0)circle(.1);
  \draw[color=lightgray,fill=lightgray] (7.25,0)circle(.1);
  \draw[color=lightgray,fill=lightgray] (0.25,2)circle(.1);
  \draw[color=lightgray,fill=lightgray] (1.25,2)circle(.1);
  \draw[color=lightgray,fill=lightgray] (2.25,2)circle(.1);
  \draw[color=lightgray,fill=lightgray] (3.25,2)circle(.1);
  \draw[color=lightgray,fill=lightgray] (4.25,2)circle(.1);
  \draw[color=lightgray,fill=lightgray] (5.25,2)circle(.1);
  \draw[color=lightgray,fill=lightgray] (6.25,2)circle(.1);
  \draw[color=lightgray,fill=lightgray] (7.25,2)circle(.1);
  \garcup{0.25}{0.75}{.15}
  \garcup{2.25}{4.75}{.7}
  \garcup{2.75}{4.25}{.5}
  \garcup{3.25}{3.75}{.15}
  \garcup{5.25}{5.75}{.15}
  \garcup{6.75}{7.25}{.2}
  \garcdn{-.25}{.25}{.2}
  \garcdn{1-.25}{1.25}{.2}
  \garcdn{3.25}{3.75}{.15}
  \garcdn{4.75}{7.25}{.5}
  \garcdn{5.25}{5.75}{.15}
  \garcdn{6.25}{6.75}{.15}
  \gcve{-.25}{1.75}
  \gcve{1.25}{2.25}
  \gcve{1.75}{2.75}
  \gcve{6.25}{4.25}
	\end{tikzpicture}
\end{center}
\caption{A planar partition from $\PP_8$ (black) and its image (gray) under the isomorphism $\PP_8\to\J_{16}$.}
\label{fig_PPn_J2n}
\end{figure}
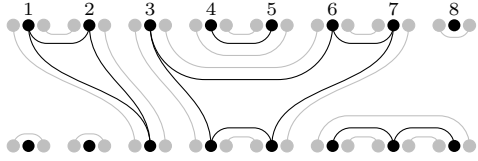

\section{Ranks of ideals and dimensions of cell modules \\ and irreducible representations}
\label{sec_irreps} 

As mentioned in the introduction, further motivation for the rank and idempotent rank formulae obtained above comes from the connection between these numbers and the dimensions of the irreducible representations of the corresponding algebras (realised as twisted semigroup algebras): namely, the partition, Brauer and Temperley--Lieb algebras. 
Specifically, the rank and idempotent rank formulae we obtained above can be used to recover formulae for dimensions of cell modules which, in the semisimple case, correspond to dimensions of irreducible representations. This fact was brought to the attention of the authors by Arun Ram (at the Workshop on Diagram Algebras, Stuttgart, 2014) who pointed out that the set of ranks (and idempotent ranks) of the two-sided ideals of the Jones monoid give precisely the dimensions of the irreducible representations of the Temperley--Lieb algebras. 
Analogous statements hold for the partition and Brauer algebras, and the purpose of this section is to explain this connection. 
We note that the formulae for the dimensions of the cell modules that we derive here are not new (we provide the relevant references below).
Our aim here is to give a new way of deriving these formulae, and at the same time explain the relationship between the problem of computing such numbers and that of computing ranks (and idempotent ranks) of ideals of diagram monoids. 

In each case the main idea is the same: the ranks of ideals of the semigroup are given by numbers of $\mathscr{R}$-classes in $\mathscr{J}$-classes of the semigroup. By \cite{Wilcox2007}, it follows that these numbers in turn arise in the construction of the cell modules  of the corresponding twisted semigroup algebras when they are realised as cellular algebras. In the case that the algebras are semisimple, the cell modules give a complete set of irreducible representations of the algebra. The questions of determining the dimensions of these representations can then easily be reduced to the question of determining the ranks of the ideals of the semigroups, which are given by the formulae in 
Theorems~\ref{thm_main1}, \ref{thm_main3} and \ref{thm_main4} above (depending on the semigroup--algebra pair under consideration). 

We shall now explain this in more detail for the partition monoid and partition algebra. The arguments for the other pairs (the Brauer semigroup and Brauer algebra, and
the Jones monoid and Temperley--Lieb algebra) are analogous, and for these we shall just state the corresponding results
and give relevant references. 

\subsection*{Dimensions of irreducible representations of partition algebras}

In this subsection we follow \cite[Section~7]{Wilcox2007} to define the partition algebra and explain its cellular structure. Let $R$ be a commutative ring with identity, let $\delta \in R$ be fixed, and let $\psi$ be the mapping
\begin{align*}
& \psi: \P_n \times \P_n \rightarrow R, \quad (\alpha,\beta) \mapsto \delta^{m(\alpha,\beta)},
\end{align*}
called a \emph{twisting} from $\P_n$ to $R$, where $m(\alpha,\beta)$ denotes the number of connected components removed from the middle row when constructing the composition $\alpha \beta$ in $\P_n$. The resulting twisted semigroup algebra $R^{\psi}[\P_n]$ is called the \emph{partition algebra}; see \cite{Martin1991, Martin1994, Jones1994_2}.  The algebra $R^{\psi}[\P_n]$ has $\P_n$ as its basis, and its multiplication $\star$ is defined on basis elements (and then extended linearly) by $\al\star\be=\psi(\al,\be)(\al\be)$.

Recall from Section~\ref{sect_Pn} above that there is a natural anti-involution $*$ on $\P_n$ that reflects graphs representing elements in the horizontal axis. By linearity, this extends to an $R$-linear anti-involution (also denoted $*$) on the partition algebra $R^{\psi}[\P_n]$.  Also recall from Section~\ref{sect_Pn} that $\rank(\alpha)$ is the number of transversal blocks of $\alpha\in\P_n$, 
that the $\gj$-classes of $\P_n$ are the sets 
\[
\JrPn= \set{\al\in\P_n}{\rank(\al)=r}
\]
where $0\leq r\leq n$, that they form a chain:
\[
J_0(\P_n)< J_1(\P_n)<\cdots< J_{n-1}(\P_n)< \JnPn,
\]
and that the ideals of $\P_n$ are precisely the sets
\[
\IrPn=J_0(\P_n)\cup J_1(\P_n)\cup\cdots\cup J_r(\P_n)=\set{\al\in\P_n}{\rank(\al)\leq r}.
\]
For each $\gj$-class $J=\JrPn$, we may choose and fix a maximal subgroup $G_{J}$ of $J$ that is fixed setwise by $*$ (such a maximal subgroup contains a projection) and in this way the $*$ operation restricted to $G_{J} \cong \S_r$ corresponds to inversion in $\S_r$. Since the group algebra $R[\S_r]$ is cellular with the anti-involution induced by inversion (see for example \cite{Mathas1999}), it follows that $R[G_{J}]$ is cellular with respect to $*$, and thus (by\cite[Corollary~7]{Wilcox2007}) the partition algebra $R^{\psi}[\P_n]$ is cellular with anti-involution $*$. 
Viewing the partition algebra $R^{\psi}[\P_n]$ in this way, as a cellular algebra, gives information about its representation theory, by appealing to the general theory of cellular algebras in the following way. 

%To make sense of this in the case of the partition algebra we must first recall some general definitions and results from the theory of cellular algebras. 
%Specifically we shall need the notion of cell module, and the observation that in the semisimple case the cell modules give a complete set of irreducible representations for the algebra. This is a general fact about cellular algebras which will, in particular, apply to the partition algebra, and also the other algebras we consider below. 

%
%

Let $A$ be a cellular algebra with cell datum $(\Lambda, M, \mathcal{C}, *)$ (see \cite{GrahamLehrer1996}). Here $\Lambda$ is a partially ordered set, the algebra $A$ has an $R$-basis
\[
\mathcal{C} = \{
C_{\mfs, \mft}^{\lambda} :
\lambda \in \Lambda, \; \mfs, \mft \in M(\lambda)
\},
\]
and an anti-involution $*: A \rightarrow A$ is given by $(C_{\mfs, \mft}^{\lambda})^* = C_{\mft, \mfs}^{\lambda}$.
Moreover, for 
all $\lambda \in \Lambda$, $\mfs, \mft \in M(\lambda)$ and $a \in A$
\[
a C_{\mfs, \mft}^{\lambda} \equiv
\sum_{\mfs' \in M(\lambda)} r_a(\mfs', \mfs)
C_{\mfs', \mft}^{\lambda} \mod{A ( < \lambda)}
\]
where each $r_a(\mfs', \mfs) \in R$ is independent of $\mft$, and where $A ( < \lambda)$ is the $R$-submodule of~$A$ generated by $\{ C_{\mfu, \mfv}^{\mu} :  \mu < \lambda, \ \mfu, \mfv \in M(\mu) \}$. It follows that for all $\lambda \in \Lambda$ and all $ \mfs_1, \mfs_2, \mft_1, \mft_2 \in M(\lambda)$ we have
\[
C_{\mfs_1, \mft_1}^{\lambda}C_{\mfs_2, \mft_2}^{\lambda} \equiv
\phi(\mft_1, \mfs_2) C_{\mfs_1, \mft_2}^{\lambda} 
\mod{A ( < \lambda)}
\]
for some $\phi(\mft_1, \mfs_2) \in R$ that depends only on $\mft_1$ and  $\mfs_2$.  %(A new class of algebras, called \emph{cell algebras}, was introduced by May in \cite{May2015,May2015_2}.  These algebras generalise the cellular algebras of \cite{GrahamLehrer1996} by removing the assumption that an anti-involution exists.  Much of the general theory of cellular algebras extends to the more general cell algebras, and many more natural examples may be treated; such examples include semigroup algebras of finite full and partial transformation semigroups.)

For each $\lambda \in \Lambda$, let $W(\lambda)$ denote the left $A$-module with $R$-basis $\{ C_\mfs : \mfs \in M(\lambda) \}$ and $A$-action given by
\[
aC_\mfs = \sum_{\mfs' \in M(\lambda)} r_a(\mfs', \mfs)C_{\mfs'}
\]
for each $a \in A$. 
The $W(\lambda)$ are called \emph{cell modules}. 
For $\lambda \in \Lambda$ define the bilinear form
\[
\phi_\lambda: W(\lambda) \times W(\lambda) \rightarrow R, \quad
\phi_\lambda(C_\mfs, C_\mft) = \phi(\mfs,\mft), 
\]
for $\mfs, \mft \in M(\lambda)$. The radical of $\lambda \in \Lambda$ is then the $A$-submodule
\[
\mathrm{rad}(\lambda) = \{
x \in W(\lambda) : 
\phi_\lambda(x,y) = 0 \ \forall y \in W(\lambda)\}
\] 
of $W(\lambda)$. 
%
% Next -- quote the two theorems -- direct from the cellular algebras paper -- the original paper
% where it is assumed that R is a field.  
%
Let us now quote two important results from \cite{GrahamLehrer1996}.

\begin{theorem}[Graham and Lehrer {\cite[Theorem~3.4(i)]{GrahamLehrer1996}}]
Let $R$ be a field and let $(\Lambda, M, \mathcal{C}, *)$ be a cell datum for the $R$-algebra $A$.
For $\lambda \in \Lambda$ let $L_\lambda = W(\lambda) / \mathrm{rad}(\lambda)$. Then
\[
\{ L_\lambda : \lambda \in \Lambda, \phi_\lambda \neq 0 \}
\] 
is a complete set of (representatives of equivalence classes of) absolutely irreducible $A$-modules. 
\end{theorem}

\begin{theorem}[Graham and Lehrer {\cite[Theorem~3.8]{GrahamLehrer1996}}]
Let $A$ be an $R$-algebra ($R$ a field) with cell datum $(\Lambda, M, \mathcal{C}, *)$. 
Then the following are equivalent.
\begin{enumerate}
\item[(i)] The algebra $A$ is semisimple. 
\item[(ii)] The nonzero cell representations $W(\lambda)$ are irreducible and pairwise inequivalent. 
\item[(iii)] The form $\phi_{\lambda}$ is nondegenerate (i.e. $\mathrm{rad}(\lambda)=0$) for each $\lambda \in \Lambda$. 
\end{enumerate}
\end{theorem}

It follows from these results that the irreducible left $A$-modules are paramaterised by by the set
\[
\Lambda_0 = \{ \lambda \in \Lambda : \phi_\lambda \neq 0 \},
\]
and the dimensions of the irreducible left $A$-modules are given by 
\[
\mathrm{dim}_R(L_\lambda) = |M(\lambda)| - \mathrm{dim}_R(\mathrm{rad}(\lambda)). 
\]
In particular, in the case of semisimple cellular algebras, these dimensions are given simply by $|M(\lambda)|$ with $\lambda \in \Lambda$. 

Returning our attention to the partition algebra $R^{\psi}[\P_n]$, applying \cite[Theorem~5 and Corollary~7]{Wilcox2007}, Wilcox shows that the partition algebra $R^{\psi}[\P_n]$ is cellular with cell datum 
$(\Lambda, M, C, *)$ where  
\[
\Lambda = \{ (J, \lambda) : J \in \mathbb{J} \ \mbox{and} \ \lambda \in \Lambda_J \},
\]
with a partial order defined on it (which we shall not need here), 
the symbol $\mathbb{J}$ denotes the set of $\gj$-classes of the semigroup $\P_n$, and $\Lambda_J$ comes from the cell datum $(\Lambda_J, M_J, \mathcal C_J, *)$ for the (non-twisted) group algebra $R[G_J]$ where $G_J$ is the a maximal subgroup of $D$ fixed set-wise by $*$ defined above. 
Moreover, for $(J,\lambda) \in \Lambda$ we have
\[
M(J,\lambda) = \mathbb{L}_J \times M_J(\lambda),
\]
where $\mathbb{L}_J$ is the set of $\gl$-classes of the $\gj$-class $J$, and so in particular 
\begin{equation}\label{eq_connection}
|M(J,\lambda)| = |\mathbb{L}_J| \cdot |M_J(\lambda)|.
\end{equation}

Finally, let us now consider the case that $R$ is the field $\mathbb{C}$ of complex numbers and consider the complex partition algebra $\mathbb{C}^{\psi}[\P_n]$. 
For fixed $n$, semisimplicity of the algebra $\mathbb{C}^{\psi}[\P_n]$ depends on $\psi$. This is explored in more detail in  \cite{Halverson2005}. Specifically it is  observed that, with $\delta$ in the definition of $\psi$, for all but a finite number of $\delta \in \mathbb{C}$, the algebra  $\mathbb{C}^{\psi}[\P_n]$ is semisimple. In the cases that $\mathbb{C}^{\psi}[\P_n]$ is semisimple, a method for computing dimensions of irreducible $\mathbb{C}^{\psi}[\P_n]$-modules in terms of counting paths in a certain graph $\hat{A}$, whose vertices are labelled by integer partitions,  is given in \cite[Theorem~2.24]{Halverson2005}. Here our aim is to state a result analogous to \cite[Theorem~2.24(b)]{Halverson2005} but where we express the dimensions of the irreducible  $\mathbb{C}^{\psi}[\P_n]$-modules in terms of the rank formula we obtained in Theorem~\ref{thm_main1}. 

In order to do this we need to recall some basic notions about integer partitions and the representation theory of the symmetric group. Any integer partition $\lambda$ can be identified with a sequence $\lambda = (\lambda_1 \geq \lambda_2 \geq \cdots \geq \lam_k)$. Integer partitions are represented using Young diagrams; for example, the diagram corresponding to $\lambda = (543311)$ is 
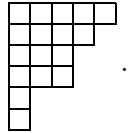
\[
\etab(5,4,3,3,1,1).
\]
The \emph{hook length} of the box $b$ of $\lambda$ is 
\[
h(b) = (\lambda_i -j) + (\lambda_j'-i) + 1 \
\mbox{if $b$ is in position $(i,j)$ of $\lambda$.}
\] 
Here, $\lam'=(\lam_1'\geq\lam_2'\geq\cdots\geq\lam_{\lam_1}')$ denotes the partition obtained by reflecting (the Young diagram corresponding to) $\lam$ in the leading diagonal.
In other words, $h(b)$ is the number of boxes to the right of $b$ plus the number below $b$, plus one (to include the box $b$ itself in the count).  Write $\lambda \vdash r$ and $|\lam|=r$ if $\lambda$ is a partition with $r$ boxes (so $\lam_1+\lam_2+\cdots+\lam_k=r$). 

Let $0\leq r\leq n-1$ and consider the $\gj$-class  $J = J_r(\P_n)$ of the partition monoid $\P_n$. The maximal subgroups of this $\gj$-class are isomorphic to the symmetric group $\S_r$. As noted above, the algebra $\mathbb{C}[\S_r]$ is known to be cellular with cell datum $(\Lambda_J, M_J, \mathcal C_J, *)$ where $\Lambda_J$ is the set $\hat{S}_r = \{ \lam:\lambda \vdash r\}$ of all partitions with $r$ boxes, carrying a natural partial order  (which will not be needed here). For a partition $\lambda$, let $\mathrm{Std}(\lambda)$ denote the set of standard $\lambda$-tableaux (that is, all ways of filling the boxes of $\lambda$ with the symbols $1$ up to $r$ so that both rows and columns are strictly increasing). Then for $\lambda \in \Lambda_J$, we have $M_J(\lambda) = \mathrm{Std}(\lambda)$. 

So the irreducible $\mathbb{C}[\S_r]$-modules, denoted $S_r^{\lambda}$, are indexed by the elements of $\Lambda_J = \hat{S}_r = \{ \lam:\lambda \vdash r\}$. For a given $\lambda \in \Lambda_J$ we have $M_J(\lambda) = \mathrm{Std}(\lambda)$ and the dimension of the corresponding irreducible module $S_r^{\lambda}$ is given by $\dim(S_r^\lam)=|M_J(\lambda)| = |\mathrm{Std}(\lambda)|$. This number is well known to be given by 
\[
|M_J(\lambda)| = |\mathrm{Std}(\lambda)| = \frac{r!}{\prod_{b \in \lambda}h(b)},
\]
where $h(b)$ denotes the hook-length of the box $b$ of $\lambda$, as defined above. In the last expression, we write $b\in\lam$ to indicate that $b$ is a box of (the Young diagram representing)~$\lam$. Now combining these observations, Equation~\ref{eq_connection} and Theorem~\ref{thm_main1} we obtain the following. 

%It is well known (see \cite{}) that the irreducible $\mathbb{C}S_r$-modules $S_r^{\lambda}$ are indexed by the elements of $\hat{S}_r = \{ \lambda \vdash r\}$ and 
%\[
%\mathrm{dim}(\hat{S}_r) = \frac{r!}{\prod_{b \in \lambda}{h(b)}}. 
%\]

%, defined as follows. For a partition $\lambda$, let $\lambda_i$ denote the number of boxes in row $i$ of $\lambda$, let $\lambda_j'$ denote the number of boxes in column $j$ of $\lambda$, and $|\lambda|$ denote the total number of boxes in $\lambda$. Then
%\[
%h(b) = (\lambda_i -i) + (\lambda_j'-j) + 1 \
%\mbox{if $b$ is in position $(i,j)$ of $\lambda$.}
%\] 
%In other words, $h(b)$ is the number of boxes to the right of $b$ plus the number below $b$, plus one (to include the box $b$ itself in the count).  
%
%Combining all of this, we conclude:

%
%\
%
%\
%
%\
%
%\
%
%*** 11th February 21:00 stopped here to eat: The next task on this is to insert the necessary material here to make the proposition below understandable ***
%
%*** HERE ***

%\
%
%[INSERT HERE JUST THE NECESSARY MATERIAL FOR THE FOLLOWING PROPOSITION TO MAKE SENSE]

\begin{proposition}\label{prop_partition_alg}
If $\mathbb{C}^{\psi}[\P_n]$ is semisimple, then the irreducible $\mathbb{C}^{\psi}[\P_n]$-modules $A_n^\mu$ are indexed by elements of the set 
\[
\hat{A}_n = 
\{
\text{\emph{partitions}} \ \mu :  0\leq|\mu|\leq n
\}.
\]
Moreover, for $|\mu|<n$, we have  
\begin{align*}
\mathrm{dim}(A_n^{\mu}) 
& =
\mathrm{rank}({I}_{|\mu|}(\mathcal{P}_n)) \cdot 
\left(
\frac{|\mu|!}{\prod_{b \in \mu}h(b)} 
\right) =
\left(
\sum_{j=|\mu|}^n S(n,j){j\choose |\mu|}
\right) 
\cdot 
\frac{|\mu|!}{\prod_{b \in \mu}h(b)}.
\end{align*}
%where $h(b)$ denotes the hook length of the box $b$ in $\mu$. 
%
\end{proposition}
%
%
\begin{comment}
\begin{proof} Just a case of tying together all the observations from above. Also one needs to note that the $|M_J|$ correspond to dimensions of irreducible representations of the symmetric group, which relies on the $\mathbb{C} S_r$ being semisimple and cellular.   
\begin{align*}
\mathrm{dim}(A_n^{\mu}) 
& = |M(J_{|\mu|},\mu)| = |\mathbb{L}_{J_{|\mu|}(\P_n)}| \cdot |M_J(\mu)|
\\
& =
\mathrm{rank}({I}_{|\mu|}(\mathcal{P}_n)) \cdot 
\left(
\frac{|\mu|!}{\prod_{b \in \mu}h(b)} 
\right) \\
& =
\left(
\sum_{j=|\mu|}^n S(n,j){j\choose |\mu|}
\right) 
\cdot 
\frac{|\mu|!}{\prod_{b \in \mu}h(b)}.
\end{align*}
\end{proof}
\end{comment}
%
%
%
We note that in the cases $|\mu|=n$, the right hand side of the formula in the above proposition still holds, but the expression with the term $\mathrm{rank}({I}_{|\mu|}(\mathcal{P}_n))$ is no longer correct (since the partition monoid $\P_n=\InPn$ itself does not have rank $1$). 
It is interesting to compare this statement with \cite[Theorem~2.24(b)]{Halverson2005}, where the dimensions are given by  
counting paths in a certain graph $\hat{A}$, whose vertices are labelled by partitions (so called, Bratteli diagrams).

\subsection*{Dimensions of irreducible representations of Brauer algebras}

Recall from Section~\ref{sec_Brauer} above that the Brauer monoid $\B_n$ is the subsemigroup of $\P_n$ consisting of all partitions whose blocks have cardinality $2$,  that for $r=n-2k$ with $k\geq0$ and $0\leq r\leq n$, we write
\[
\JrBn=\set{\al\in\B_n}{\rank(\al)=r}=\JrPn\cap\B_n,
\]
that these are precisely the $\gj$-classes of $\B_n$, and that they form a chain:
\[
J_m(\B_n)<J_{m+2}(\B_n)<\cdots<\JnmBn<\JnBn,
\]
where $m$ denotes $0$ if $n$ is even, or $1$ otherwise.  
As in the partition monoid, the maximal subgroups of the $\gj$-class $J_r(\B_n)$ are isomorphic to the symmetric group $\S_r$. 
The ideals of $\B_n$ are precisely the sets
\[
\IrBn = J_m(\B_n)\cup J_{m+2}(\B_n)\cup\cdots\cup \JrBn = \set{\al\in\B_n}{\rank(\al)\leq r}.
\]
The twisted semigroup algebra $R^{\psi}[\B_n]$ is called the Brauer algebra. (For simplicity, we write $\psi$ for the restriction of the twisting $\psi$ to $\B_n$.) This algebra has been studied extensively in the literature; see for example 
\cite{
Brauer1937,
Hanlon94,
Wenzl88
}.
Semisimplicity of Brauer algebras is considered in 
\cite{
Wenzl88,
Rui05,Rui06
}.
Cellularity of the Brauer algebra $R^{\psi}[\B_n]$ may be proved, as for the partition algebra, by appealing to cellularity of symmetric group algebras; see \cite[Section~8]{Wilcox2007}. Following the same argument  used for the partition algebra above, and applying Theorem~\ref{thm_main3},  we obtain the following.  

\begin{proposition}\label{prop_dim_B}
If $\mathbb{C}^{\psi}[\B_n]$ is semisimple, then the irreducible $\mathbb{C}^{\psi}[\B_n]$-modules, $B_n^\mu$ are indexed by elements of the set 
\[
\hat{B}_n = 
\{
\text{\emph{partitions}} \ \mu : 0 \leq |\mu|=r=n-2k \leq n
\}.
\]
Moreover, for $0 \leq |\mu|=r=n-2k \leq n-2$ we have 
\begin{align*}
\mathrm{dim}(B_n^{\mu}) 
& = 
\mathrm{rank}({I}_{|\mu|}(\mathcal{P}_n)) \cdot
\left(
\frac{|\mu|!}{\prod_{b \in \mu}h(b)} 
\right) 
 = 
\frac{n!}{2^k k! r!}
\left(
\frac{|\mu|!}{\prod_{b \in \mu}h(b)} 
\right) =
\frac{n!}{
2^k k!\prod_{b \in \mu}h(b)
}.
\end{align*}
%where $k = \frac{n-|\mu|}{2}$.
\end{proposition}
We have not been able to find this formula as stated in the above proposition anywhere in the literature. 
However, a similar (and equivalent) formula may be found in~
\mbox{\cite[Equation 3.4]{Pan95}}.
Also, in a similar way as for the partition algebra above, there is an alternative approach to computing these dimensions using Bratteli diagrams; see 
\cite[Theorem 3.2]{Wenzl88}.

\subsection*{Dimensions of irreducible representations of Temperly--Lieb algebras}

Recall from Section~\ref{sec_Jones} that the Jones monoid $\J_n$ is the subsemigroup of $\B_n$ consisting of all partitions whose blocks have cardinality $2$ and may be drawn in a planar fashion, that for $0\leq r=n-2k\leq n$,
\[
\JrJn=\set{\al\in\J_n}{\rank(\al)=r} \quad\text{and}\quad \IrJn=\set{\al\in\J_n}{\rank(\al)\leq r}
\]
are precisely the $\gj$-classes and ideals of $\J_n$, and that the $\gj$-classes form a chain. In contrast to the partition and Brauer monoids, the maximal subgroups of the Jones monoid are all trivial. 

The twisted semigroup algebra $R^{\psi}[\J_n]$ is called the Temperley--Lieb algebra; see \cite{Temperley71, Goodman1993, Jones1994_2}. In \cite[Section 5]{Westbury95} a simple criterion for semisimplicity of Temperley--Lieb algebras is given. 
Cellularity of the Temperly--Lieb algebras may be proved as for the partition and Brauer algebras above, but this time appealing to the fact that the maximal subgroups are trivial, and thus the corresponding group algebras are trivially cellular; see \cite[Section~8]{Wilcox2007}. 

Following the same approach as used for the partition and Brauer algebras above, and applying Theorem~\ref{thm_main4}, we obtain the following. 

\begin{proposition}\label{prop_dim_TL}
If $\mathbb{C}^{\psi}[\J_n]$ is semisimple, then the irreducible $\mathbb{C}^{\psi}[\J_n]$-modules, $J_n^r$ are indexed by elements of the set 
\[
\hat{J}_n = \{ r \in \mathbb{Z}_{\geq 0} :  0 \leq r = n-2k \leq n \}.
\]
Moreover, for $0 \leq r=n-2k \leq n-2$, we have 
\begin{align*}
\mathrm{dim}(J_n^r) 
& = 
\rank(\IrJn) = \frac{r+1}{n+1}{n+1 \choose k}.
\end{align*}
\end{proposition}
This  agrees with known results from the literature; see the recent survey article \cite{Ridout14} for details.

\subsection*{Acknowledgements}

We thank the referees for their careful reading of the article, for several helpful suggestions that improved readability, and for drawing our attention to a number of interesting references.

\bibliographystyle{abbrv} 
\bibliography{MasterBibliography} 
\end{document}